\documentclass[reqno,draft,intlimits]{amsart}

\usepackage[active]{srcltx}

\usepackage[reqno]{amsmath}
\usepackage{amssymb} 
\usepackage{wrapfig}

\makeatletter
\long\def\@makecaption#1#2{%
  \vskip3pt
  \sbox\@tempboxa{\small#1. #2}%
    \ifdim \wd\@tempboxa >\hsize
    \small#1. #2\par
  \else
    \global \@minipagefalse
    \hb@xt@\hsize{\hfil\box\@tempboxa\hfil}%
  \fi
  \vskip0pt}
\makeatother

\def\figone{%
\begin{picture}(50,110)
\multiput(10,123)(7,0){10}{-}
\put(78,123){0}
\multiput(10,80)(7,0){10}{-}
\put(78,80){1}
\multiput(10,40)(7,0){10}{-}
\put(78,40){2}
\multiput(10,0)(7,0){10}{-}
\put(78,0){3}
\put(35,123){\circle*{3}}%0-1
\put(35,123){\line(-2,-3){27}}
\put(35,123){\line(2,-5){33}}
\put(35,123){\line(0,-1){40}}
\put(9,83){\circle*{3}}
\put(9,83){\line(0,-1){40}}
\put(9,83){\line(1,-2){40}}
\put(35,83){\circle*{3}}
\put(51,83){\circle*{3}}
\put(51,83){\line(0,-1){40}}
\put(51,83){\line(-1,-6){7}}
\put(9,43){\circle*{3}}
\put(29,43){\circle*{3}}
\put(29,43){\line(0,-1){40}}
\put(29,43){\line(-1,-2){20}}
\put(44,43){\circle*{3}}
\put(51,43){\circle*{3}}
\put(67,43){\circle*{3}}
\put(67,43){\line(0,-1){40}}
\put(67,43){\line(-1,-6){7}}
\put(9,3){\circle*{3}}
\put(29,3){\circle*{3}}
\put(49,3){\circle*{3}}
\put(60,3){\circle*{3}}
\put(67,3){\circle*{3}}
\end{picture}}

\newtheorem{thm}{Theorem}[section] 
  
\newtheorem{proc}{Proposition}[section] 
       
\newtheorem{cor}{Corollary}[section]   
          
\newtheorem{lem}{Lemma}[section]

\newtheorem{defi}{Definition}[section]
\newtheorem{rmk}{Remark}[section]
\newtheorem{ex} {Example}[section]

\renewcommand{\a}{\alpha}

\def\b{\beta}

\def\e{\varepsilon} 
\def\i{\iota}

\def\ls{{\rm [\![}}
\def\rs{{\rm ]\!]}}
 \DeclareMathOperator{\id}{id} 
\DeclareMathOperator{\Div}{div} \DeclareMathOperator{\Ker}{Ker}

 \DeclareMathOperator{\const}{const}

\DeclareMathOperator{\Hom}{Hom} \DeclareMathOperator{\defin}{def} 
  
\DeclareMathOperator{\End}{End}  
 \DeclareMathOperator{\Ann}{Ann}

\newcommand{\tr}{\mathrm{tr}}
\newcommand{\ad}{\mathrm{ad}}
\newcommand{\df}{\stackrel{\mathrm{def}}{=}}

\newcommand{\V}{\mathcal V}

\newcommand{\R}{\mathbb R}
\newcommand{\N}{\mathbb N}
\newcommand{\Z}{\mathbb Z}
\newcommand{\C}{\mathbb C}
\newcommand{\Q}{\mathbb Q}

\newcommand{\dF}{\mathbb F}

\newcommand{\gk}{\boldsymbol{k}}

\newcommand{\bg}{\boldsymbol{\gamma}}

\newcommand{\gG}{\boldsymbol{\mathfrak g}}
\newcommand{\gH}{\boldsymbol{\mathfrak h}}
\newcommand{\gS}{\boldsymbol{\mathfrak S}}
\newcommand{\gs}{\boldsymbol{\mathfrak s}}
\newcommand{\gA}{\boldsymbol{\mathfrak a}}

\newcommand{\gR}{\boldsymbol{\mathfrak r}}
\newcommand{\gl}{\boldsymbol{\mathfrak l}}
\newcommand{\go}{\boldsymbol{\mathfrak o}}
\newcommand{\gu}{\boldsymbol{\mathfrak u}}
\newcommand{\gp}{\boldsymbol{\mathfrak p}}
\newcommand{\DEF}{\overset\defin =}

\DeclareFontFamily{OT1}{wncyi}{} \DeclareFontShape{OT1}{wncyi}{m}{it}{
   <5> <6> <7> <8> <9> gen * wncyi
   <10> <10.95> <12> <14.4> <17.28> <20.74> <24.88> wncyi10
  }{}
\DeclareSymbolFont{cyrletters}{OT1}{wncyi}{m}{it} 
\DeclareSymbolFontAlphabet{\cyrmath}{cyrletters} 
\DeclareMathSymbol{\rE}{\cyrmath}{cyrletters}{003} 
\DeclareMathSymbol{\rD}{\cyrmath}{cyrletters}{068} 
\DeclareMathSymbol{\rG}{\cyrmath}{cyrletters}{017} 
\DeclareMathSymbol{\rI}{\cyrmath}{cyrletters}{073} 
\DeclareMathSymbol{\rL}{\cyrmath}{cyrletters}{076} 
\DeclareMathSymbol{\rZ}{\cyrmath}{cyrletters}{090}

\newdimen\theight
\def \refright#1{%
             \vadjust{\setbox0=\hbox{\quad\vtop{\hsize5cm\bf\noindent #1}}%
             \theight=\ht0
             \advance\theight by \dp0    \advance\theight by \lineskip
             \kern -\theight \vbox to \theight{\rightline{\rlap{\box0}}%
             \vss}%
             }}%
\begin{document}
% end of leading block
%\Rus

\newcommand{\PsP}{Poisson structure $P$}
\def\ldb{{\rm [\![}}
\def\rdb{{\rm ]\!]}}
\def\a{\alpha}
\def\o{\omega}
\def\r{\rho}
\def\th{\theta}
\def\s{\sigma}
\def\t{\tau}
\def\be{\begin{equation}}
\def\ee{\end{equation}}
\def\bea{\begin{eqnarray}}
\def\eea{\end{eqnarray}}
\def\nn{\nonumber}

%\begin{document}

\begin{center}
{\LARGE \bf  PARTICLE-LIKE STRUCTURE OF LIE \newline 

ALGEBRAS} 

\end{center}
\vspace{.5cm}
\begin{center}

{\Large A.M. Vinogradov}\footnote{Address: via Roma 27, 40042 Lizzano 
in Belvedere, Italia;  e-mail : vinogradov.unisa@gmail.com.}
\end{center}

\bigskip

%\begin{center}
%$^{1}$ Levi-Civita Institute, Italia.\\
%${^2}$ tel. +39-053451250,  \;e-mail : vinogradov.unisa@gmail.com
%\end{center}
%\smallskip

%\vspace{1.0cm}

\noindent {\bf Abstract.}
If a Lie algebra structure $\gG$ on a vector space is the sum of a family of
mutually compatible Lie algebra structures $\gG_i$'s, we say that $\gG$ is simply assembled
from the $\gG_i$'s. Repeating this procedure with a number of Lie algebras, themselves simply assembled from the $\gG_i$'s, one obtains a Lie algebra assembled in two steps from the 
$\gG_i$'s, and so on. We describe the process of modular
disassembling of a Lie algebra into a unimodular and a non-unimodular part. We then study two inverse questions: which Lie algebras
can be assembled from a given family of Lie algebras, and from which Lie algebras can
a given Lie algebra be assembled? We develop some
basic assembling and disassembling techniques that constitute the elements of a
new approach to the general theory of Lie algebras.
The main result of our theory is that any finite-dimensional Lie algebra over an algebraically closed field of
characteristic zero or over $\R$ can be assembled in a finite number of steps from two elementary constituents, which we call dyons and triadons. Up to an abelian summand, a dyon is a Lie algebra structure isomorphic to the non-abelian 2-dimensional Lie algebra, while a triadon is isomorphic to the 3-dimensional Heisenberg Lie algebra. As an example, we describe constructions of classical Lie algebras from triadons.

%%%%%%%%%%%%%%%%%%%%%%%%%%%%%%%%%%%%%%%%%%%%%

\vspace{5.0cm}

\pagebreak

\tableofcontents

%\begin{center}
\section{Introduction}
%\end{center}
According to the modern view, matter is of a compound structure. The constituents, elementary 
particles,  are characterised by their symmetry properties. These properties are formalised 
in terms of Lie algebras, and one may hypothesise that this compound nature of matter should be somehow mirrored in the structure of the symmetry algebras. This, and other similar considerations, suggests that Lie algebras have, in some sense, a compound structure. Our study, some results of which are presented in this paper, was motivated by this question. Our main result is that the finite-dimensional Lie algebras 
over $\R$ or over an algebraically closed field of characteristic  zero are made from two elementary constituents/``particles", which we call \emph{dyons} and \emph{triadons}. We also use the unifying term \emph{lieons} (from Lie-on) for both types.

The exact meaning of ``made from" in the context of an ``elementary particle theory" 
for the ``Lie matter" 
should be duly formalised, a first, nontrivial question one
has to answer. Surprisingly, a suggestion comes from Poisson geometry. In fact, 
on the one hand, a Lie algebra structure on a vector space is naturally interpreted as a linear Poisson structure 
on its dual. On the other hand, it is natural to think of a Poisson structure defined by 
a Poisson bi-vector $P$ as ``made from" Poisson structures defined by bi-vectors $P_1$ 
and $P_2$ if 
$P=P_1+P_2$. In such a case,  $P_1$ and $P_2$ are called \emph{compatible}. Translating this idea into the language of Lie algebras, we obtain the following definition: 
Lie algebra structures $[\cdot,\cdot]_1$ and $[\cdot,\cdot]_2$ on a vector space $V$ 
are \emph{compatible} if $[\cdot,\cdot]_1+[\cdot,\cdot]_2$ is  a Lie algebra structure 
as well. If a Lie algebra structure $[\cdot,\cdot]$ is presented in the form
$$
[\cdot,\cdot]=[\cdot,\cdot]_1+\dots+[\cdot,\cdot]_m
$$
with mutually compatible structures $[\cdot,\cdot]_i$'s, we speak of a \emph{disassembling} 
of $[\cdot,\cdot]$, i.e., we consider that $[\cdot,\cdot]$ is ``made from" the $[\cdot,\cdot]_i$'s. 
Since compatibility is not a transitive notion, it makes sense to go ahead and 
try to disassemble all \emph{compounds} $[\cdot,\cdot]_i$, and so on.  This disassembling
process can be repeated so as to finally come to a number of dyons and triadons,
assuming that the ground field is $\R$ or an algebraically closed field of characteristic zero.

We denote by $\between$ (resp., $\pitchfork$) the non-abelian $2$-dimensional Lie 
algebra (resp., the $3$-dimensional Heisenberg Lie algebra), and by $\gamma_m$ the
$m$-dimensional abelian Lie algebra. An $n$-dimensional dyon ($n$-dyon), denoted 
by $\between_n$, is a Lie algebra structure isomorphic to $\between\oplus\gamma_{n-2}$. 
An $n$-dimensional triadon ($n$-triadon), denoted by $\pitchfork_n$, is a Lie algebra structure
isomorphic to $\pitchfork\oplus\gamma_{n-3}$. Lieons, i.e., dyons and triadons, are the
simplest non-abelian Lie algebra structures.

It is remarkable that the symmetry algebra $\gu(2)=\gs\go(3)$ of a nucleon 
can be assembled in one step from three triadons. Speculatively, one might think 
that this structure of the symmetry reflects the fact that a nucleon is made from  
three ``quarks".
Of course, the physical validity of this and similar speculations requires a deeper 
analysis, which will be the subject of a separate work. 
In this paper we concentrate on the mathematical aspects of the theory.

The concept of compatible Poisson structures originates in  F.\,Magri's work \cite{Mag} 
on bi-Hamiltonian systems. It was subsequently developed and exploited  by many authors 
in the context of integrable systems and Poisson geometry. Informally, the compatibility 
may be viewed as the coexistence of two or more dynamical structures. Similarly, if the symmetry algebra of a physical system is a Lie 
algebra made up of several pieces of 
the same nature, it is natural to think that, in turn, this algebra is made 
up from the symmetry algebras of its compounds. Roughly, this is the reason for
formalising the vague idea of ``made from" as the notion of ``compatibility".  This point is delicate and,  
in our opinion, it  merits to be stressed.
As far as we know,  Lie algebras were not viewed and studied from this point of view until now.

The translation of the techniques and constructions of 
Poisson geometry, in particular, the Schouten-Nijenhuis machinery, into the context 
of Lie algebras is very useful, and we exploit it fully. 
Consequently, a few pages in this paper are dedicated to the necessary elements 
of Poisson geometry.

The paper is structured as follows. The necessary ``initial data" concerning 
differential forms, multi-vector fields and the Schouten bracket are 
collected in Section\,\ref{preliminaries}. The notion of compatibility of Poisson structures 
and the notion of compatibility of Lie algebra 
structures, which are central for this 
paper, are discussed there.
   
The modularity properties of Poisson and Lie algebra structures are discussed in 
Section\,\ref{mdlr}, in which general compatibility conditions are specialised to unimodular Poisson structures. The central point here is an almost canonical 
disassembling of a Poisson structure $P$ into a \emph{unimodular} part and a \emph{completely 
non-unimodular} part which we call the \emph{modular disassembling}. 
The Poisson bi-vector corresponding to the completely non-unimodular part of $P$ is of the
form $P_{\nu}\wedge\Xi=\ls P,\nu\Xi\rs$ where $\Xi$ is the modular vector field 
of $P$ and $\nu$ is a smooth function on the manifold such that $\Xi(\nu)=1$. It is completely characterised 
by the property that its unimodular part is trivial. 
The second part of  Section\,\ref{mdlr} is dedicated to the \emph{matching problem}: 
what are the different compatible configurations of two given Poisson structures?  In full 
generality, this problem seems to be very difficult. For this reason, we restrict 
our study to the particular case of two completely non-unimodular structures.  It turns 
out that, even in this case, the equivalence classes of matchings are labeled by functional parameters (Proposition\,\ref{mathc invariants}).

These results on the modular disassembling of general Poisson structures are then
adapted to Lie algebras in Section\,\ref{Lie-modlarity}. In particular, we call \emph{modular} those Lie algebras whose Poisson bi-vector is completely non-unimodular 
and we show how to subtract from a Lie algebra a suitable modular algebra in order
to get a unimodular algebra. The structure of modular Lie algebras is very simple. 
So, in this sense, the study of general Lie algebras is reduced to the unimodular case. 
To make this point more precise, we discuss compatibility  conditions of modular and unimodular Lie algebras. One result in this direction is that semi-simplicity and 
modularity of Lie algebras are incompatible. In the second part of this section 
the matching problem for modular Lie algebras is solved. The result is basically
that matchings of modular Lie algebras are labeled by the representations of the 2-dimensional 
Lie algebras (Theorem\,\ref{LieMathing}).

The main result of this paper is proven in Section\,\ref{dis-probl}: finite-dimensional Lie algebras over an algebraically closed field of characteristic zero or over $\R$, can be assembled from lieons (Theorems\,\ref{C-dis} 
and \ref{R-dis}).  The proof  splits naturally  into a ``solvable" part and a ``semi-simple" 
part. Accordingly, we first show that any solvable Lie algebra over an arbitrary ground 
field of characteristic zero can be assembled from lieons (Proposition\,\ref{dis-solv}). 
This part of the proof is very simple. The ``semi-simple" part is more delicate and as a preliminary step we reduce it to the disassembling problem for
abelian extensions of simple algebras.
This last problem is essentially a question about representations of simple Lie algebras.
The idea of using results from this rich and well elaborated theory seems most attractive.  However, a deeper analysis shows that such an approach, even if possible, would be 
too cumbersome and hardly instructive (if not ``amoral").  Also,
the fact that Lie algebras are compounds is, in our opinion, fundamental and as such 
should be established independently of the classification of simple Lie algebras and 
their representations. 
These are reasons for following an independent approach, which is based on the 
\emph{stripping procedure} (see Subsection\,\ref{striptease}) and reduces the 
problem to the representations of \emph{hyper-simple} Lie algebras, which are simple 
algebras without proper non-abelian subalgebras. Hyper-simple algebras do not exist over algebraically closed fields of characteristic zero, while the only hyper-simple Lie algebra 
over $\R$ is $\gs\go(3)$. This is the reason why the assembling theorem is proven in these two cases.

Besides the proof of the main theorem,  various assembling-disassembling techniques 
for Lie algebras over arbitrary ground fields are developed throughout Section\,
\ref{dis-probl}. They not only indicate possible approaches to the  assembling problem over arbitrary fields, but also they are  useful in the analysis of the compound structure 
of real and complex Lie algebras.

Section\,\ref{1st-level} is dedicated to constructions in first-level Lie algebras. 
These are Lie algebras that can be assembled from lieons in one step. To this end, we analyse the compatibility conditions for two lieons and we show that they depend only on the relative positions of the
subspaces carrying the centers and derived algebras of the lieons under consideration. One of the consequences of this fact is that the design of 
first-level Lie algebras does not depend on the ground field. We give various examples of first-level Lie algebras.

In Section\,\ref{classical} we describe constructions of
classical Lie algebras from triadons which are, in a sense, canonical. An interesting fact is that this can be done 
in no more than four steps. For instance, two steps are sufficient for the orthogonal Lie
algebras $\gs\go(g)$. This section may also be viewed as a demonstration of the 
general assembling-disassembling techniques at work.
More precisely, all simple 3-dimensional Lie algebras 
can be simply assembled from three triadons. Simple Lie algebras in higher dimensions
require more than one step.

Many natural questions arise in connection with the topics of this paper, and a selection of related problems is briefly discussed in the concluding Section\,\ref{problems}.

%\begin{center}
\section{Preliminaries}\label{preliminaries}
%\end{center}
In this section we collect the necessary data and facts on the calculus of multi-vectors and differential forms, Poisson geometry, and compatibility of Poisson and Lie algebra structures, that will be used throughout the paper, and we fix the notation. The reader will find more details about the material in this section in \cite{CabVin, V90}.  In this paper, all manifolds and maps are assumed to be smooth.

%%%%%%%%%%%%%%%%%%%%%%%%%%%%%%%
\subsection{Multi-vectors and differential forms} We use $M$ for an $n$--dimensional manifold and
%%%%%%%%%%%%%%%%%%%%%%%%%%%%%%

\begin{enumerate}
\item $D_{*}(M)=\bigoplus_{k\geq 0}D_k(M)$ for the exterior algebra of multi-vectors on $M$, 
$D(M)=D_1(M)$ for the $C^{\infty}(M)$--module of vector fields on $M$, and $``\wedge"$ for the wedge product of multi-vectors;
\item $\ldb\cdot,\cdot\rdb$ for the Schouten bracket in $D_{*}(M)$;
\item $\Lambda^{*}(M)=\bigoplus_{k\geq 0}\Lambda^{k}(M)$ for the exterior algebra of 
differential forms on $M$ and $``\wedge"$ for the wedge product of differential forms.
\end{enumerate}

If $S$ is a $\Z$-graded object, say, a multi-vector, then we use $(-1)^{\dots S\dots}$ 
and $(-1)^{\dots \bar{S}\dots}$ for $(-1)^{\dots \mathrm{deg}S\dots}$ and $(-1)^{\dots (\mathrm{deg}S-1)\dots}$, respectively.
For instance, if $P\in D_k(M), \,Q\in D_l(M)$, then in this notation
$(-1)^{P\bar{Q}}=(-1)^{k(l-1)},  \,(-1)^{P+\bar{Q}}=(-1)^{k+l-1}$, and the graded anti-commutativity and the Jacobi 
identity for the Schouten bracket are as follows:
\begin{equation}\label{Schouten-skew}
\ldb P,Q\rdb=-(-1)^{\bar{P}\bar{Q}}\ldb Q,P \rdb
\end{equation}
\begin{equation}\label{Schouten-Jacobi}
(-1)^{\bar{P}\bar{R}}\ldb P, \ldb Q,R \rdb \rdb+(-1)^{\bar{R}\bar{Q}}\ldb R, \ldb P,Q \rdb \rdb+
(-1)^{\bar{Q}\bar{P}}\ldb Q, \ldb R,P \rdb \rdb=0
\end{equation}

We denote by $\mathrm{Hgr}\;\Lambda^*(M)$ the totality of graded $\R$--linear operators acting on the graded space $\Lambda^*(M)$, and by $[\cdot,\cdot]^{gr}$, the graded commutator of such operators. An operator $\Delta\in \mathrm{Hgr}\;\Lambda^*(M)$ is 
a \emph{(graded) differential operator} of degree $k$ over $\Lambda^*(M)$ if
$$
[\omega_0,[\omega_1,\dots,[\omega_k,\Delta]^{gr},\dots,]^{gr}]^{gr}=0, \quad 
\forall \;\omega_o,\omega_1,\dots,\omega_k\in \Lambda^*(M),
$$
where the $\omega_i$'s are understood to be left multiplication operators.

Insertion of a multi-vector $Q\in D_k(M)$ into a differential form $\omega\in\Lambda^l$ is denoted by $Q\rfloor\omega\in\Lambda^{l-k}$, and $i_Q : \Lambda^{*}(M)\rightarrow\Lambda^{*}(M)$ refers to the operator
$\omega\mapsto Q\rfloor\omega$. Obviously, 
$$
i_P\circ i_Q=i_{P\wedge Q} \quad \mbox{and} \quad [i_P,i_Q]^{gr}=0,  \;P,Q\in D_{*}(M).
$$

The correspondence  $Q\leftrightarrow i_Q$ identifies the algebra $D_{*}(M)$ with the algebra of $C^{\infty}(M)$--linear differential operators  acting on $\Lambda^{*}(M)$, in such a way that $k$--vectors  correspond to operators of order $k$. Using this identification the Schouten bracket may be defined by the formula
\begin{equation}\label{Schouten}
i_{\ldb P,Q\rdb}=[[i_P,d]^{gr},i_Q]^{gr}=-(-1)^Q[i_P,[i_Q,d]^{gr}]^{gr}, \quad P,Q\in D_{*}(M).
\end{equation}

The \emph{Lie derivative operator} along a multivector $Q$ is defined as
\begin{equation}\label{Lie derivative}
L_Q=[i_Q,d]^{gr}:\Lambda^*(M)\rightarrow\Lambda^*(M),
\end{equation}
and Definition (\ref{Schouten}) may be rewritten in the form
\begin{equation}\label{Schouten-Lie}
i_{\ldb P,Q\rdb}=[L_P,i_Q]^{gr}=-(-1)^Q[i_P,L_Q]^{gr}.
\end{equation}

The following very useful formula  relates the insertion and the Lie derivative operators:
\begin{equation}\label{[L_X,i_P]}
[i_Q, L_X]^{gr}=i_{L_X(Q)},\quad X\in D(M), \;Q\in D_{*}(M),
\end{equation}
where $L_X(Q)=\ldb Q,X\rdb$.

The \emph{liesation} operator $L:Q\mapsto L_Q$ is a (right graded) derivation of 
the algebra $D_{*}(M)$:
\begin{equation}\label{L as derivation}
L_{P\wedge Q}=i_P\circ L_Q+(-1)^Q L_P\circ i_Q.
\end{equation}
Another useful interpretation of the Schouten bracket is easily derived from (\ref{Schouten-Lie}) and (\ref{L as derivation}):
\begin{equation}\label{Schouten via L}
i_{\ldb P,Q\rdb}=(-1)^Q L_{P\wedge Q}-(-1)^Q i_P\circ L_Q-
(-1)^{\bar{P}Q} i_Q\circ L_P.
\end{equation}

A convenient coordinate-wise description of the above operations is as follows. Let $x=(x_1,\dots,x_n)$ be a local chart
on $M$. Instead of the standard local expression 
$$
Q=\sum_{1\leq i_1<\dots<i_k\leq n}a_{ i_1,\dots,i_k}(x)\frac{\partial}{\partial x_{i_1}}\wedge\dots\wedge\frac{\partial}{\partial x_{i_n}}, \quad Q\in D_k(M),
$$
we shall use
\begin{equation}\label{antipolinomials}
Q=\sum_{1\leq i_1<\dots<i_k\leq n}a_{ i_1,\dots,i_k}(x)\xi_{i_1}\dots\xi_{i_n}
\end{equation}
assuming that the variables $\xi_i$'s anti-commute, i.e., $\xi_i\xi_j=-\xi_j\xi_i$. This allows 
one to introduce ``partial derivatives" $\frac{\partial}{\partial x_i}$ and $\frac{\partial}{\partial \xi_i}$ 
acting on the skew-commutative polynomials  in the $\xi_i$'s (\ref{antipolinomials}). In fact,
the first acts on the coefficients $a_{ i_1,\dots,i_k}(x)$, while the second is $C^{\infty}(M)$-linear and obeys the  ``graded" commutation relation
$\frac{\partial}{\partial \xi_i}\circ \xi_j+\xi_j\circ\frac{\partial}{\partial \xi_i}=\delta_{ij}$
in which $\xi_j$ refers to the operator of multiplication by $\xi_j$. In these terms
the Schouten bracket may be written
\begin{equation}\label{Schouten in coordinates}
\ldb P,Q \rdb=-\sum_{i}\left(\frac{\partial P}{\partial x_i}\frac{\partial Q}{\partial \xi_i}+
(-1)^P \frac{\partial P}{\partial \xi_i}\frac{\partial Q}{\partial x_i}\right).
\end{equation}
In particular, introducing the operator $X_P:D_{*}(M)\rightarrow D_{*}(M), \;X_P(Q)=\ldb P,Q \rdb$,
we obtain
\begin{equation}\label{Schouten-Hamilton field}
X_P=-\sum_{i}\left(\frac{\partial P}{\partial x_i}\frac{\partial}{\partial \xi_i}+
(-1)^P\frac{\partial P}{\partial \xi_i}\frac{\partial }{\partial x_i}\right).
\end{equation}

%%%%%%%%%
\subsection{Poisson manifolds}
%%%%%%%%%%%%%

 Recall that a \emph{Poisson structure} on a manifold $M$ is a Lie algebra structure on the $\R$-- vector space $C^{\infty}(M)$
 $$
 (f,g)\mapsto \{f,g\}\in C^{\infty}(M), \quad  f,g\in C^{\infty}(M),
 $$
which is  additionally a \emph{bi-derivation}, i.e.,
$$
\{fg,h\}=f\{g,h\}+g\{f,h\} \quad \mbox{and}  \quad\{f,gh\}=g\{f,h\}+h\{f,g\}.
$$
$P\in D_2(M)$ is a \emph{Poisson bi-vector} if $\ldb P,P\rdb=0$.  The formula
$$
\{f,g\}=P(df,dg), \quad f,g\in C^{\infty}(M).
$$
establishes a one-to-one correspondence
between Poisson  bi-vectors and Poisson structures on $M$. The Poisson bracket associated in this 
way with the Poisson bi-vector $P$ will be denoted by $\{\cdot,\cdot\}_P$.

A Poisson manifold is \emph{non-degenerate} if the corresponding Poisson bi-vector is 
non-degenerate, i.e., if the correspondence
$$
\gamma_P:\Lambda^1(M)\rightarrow D(M), \quad\omega\mapsto P(\omega,\cdot),
$$
is an isomorphism of $C^{\infty}(M)$--modules.  The homomorphism $\gamma_P$ naturally extends to a homomorphism of exterior algebras, which will  also be denoted by
$$
\gamma_P:\Lambda^{*}(M)\rightarrow D_{*}(M).
$$
This extension is an isomorphism if and only if $P$ is non-degenerate. In this case 
$\gamma_P^{-1}(P)\in \Lambda^2(M)$ is a symplectic form on $M$. In this way non-degenerate Poisson manifolds are identified with symplectic manifolds.

The Poisson differential 
$$
\partial_P:D_{*}(M)\rightarrow D_{*+1}(M), \quad \partial_P(Q)=\ldb P,Q\rdb,
$$ 
associated with a Poisson bi-vector $P$ supplies $D_{*}(M)$ with a co-chain complex structure. With any function $f\in C^{\infty}(M)$ on a Poisson manifold $M$ is associated  the vector 
field
\begin{equation}\label{P-Ham field}
P_f\DEF\partial_P(f)=\ldb P,f\rdb=-\gamma_P(df)=-df\rfloor P,
\end{equation}
called the \emph{Hamiltonian vector field}\footnote{$P_f$ is, in fact, the opposite of the usual Hamiltonian vector field but more convenient in the context of this paper.} corresponding to the \emph{Hamiltonian function} $f$.

The following definition is central for this paper.
\begin{defi}\label{Poisson compatibility}
Poisson structures $P_1$ and $P_2$ on a manifold $M$ are called \emph{compatible} if $P_1+P_2$ is a Poisson structure as well.
\end{defi}
\begin{proc}
Poisson structures $P_1$ and $P_2$ are compatible if one of the following equivalent conditions holds:
\begin{enumerate}
\item $\ldb P_1,P_2\rdb=0$;
\item $sP_1+tP_2$ is a Poisson structure, $\forall s,t\in\R$;
\item the bracket $\{\cdot,\cdot\}=s\{\cdot,\cdot\}_{P_1}+t\{\cdot,\cdot\}_{P_2},$, is a Lie algebra structure
on $C^{\infty}(M)$, $\forall s,t\in\R$;
\item $\partial_{P_1}+\partial_{P_2}$ is a differential in $D_{*}(M)$, or, equivalently, $\partial_{P_1}\partial_{P_2}+\partial_{P_2}\partial_{P_1}=0$.
\end{enumerate} 
\end{proc}
\begin{proof}
The first assertion follows directly from the equality, $\ldb P_1+P_2,P_1+P_2\rdb=2\ldb P_1,P_2\rdb$, while (2) - (4) are obvious consequences of it.
\end{proof}

%%%%%%%%%%%%%%%%
\subsection{Lie algebras} 
%%%%%%%%%%%%%%%%%

The term ``Lie algebra" in the current literature is used with two different meanings: as a concrete \emph{Lie algebra structure} on a concrete vector space and also as an isomorphism class of such structures. Since this distinction is essential in the context
of this paper, we shall use ``Lie algebra structure" instead of the usual ``Lie algebra" to
avoid a possible ambiguity.

Below Lie algebra structures will be denoted by bold Fraktur characters $\gG, \gH$, etc. 
The symbol $|\gG|$ refers to the vector space supporting $\gG$. We use square brackets,  
if necessary with additional indexes, for Lie product operations.

Let $\gG$ be a Lie algebra structure over a ground field $\gk$ and  $V=|\gG|$. This 
structure  naturally extends to a Lie algebra structure on the algebra $\gk[V^*]$ of 
polynomials on $V^*=\Hom_{\gk}(V,\gk)$. In fact, 
denoting by $f_v$ the linear function on $V^*$ corresponding to $v\in V$, we define a 
``Poisson bracket" $\{\cdot,\cdot\}$ on linear functions by putting
$$
\{f_v,f_w\}\DEF f_{[v,w]}, \quad v,w,\in V,
$$ 
and then extend it as a bi-derivation to the whole polynomial algebra. This construction 
remains valid for any larger algebra $A\supset \gk[V^*]$ assuming that any derivation 
of  $\gk[V^*]$ uniquely extends to $A$. For instance, $C^{\infty}(V^*)$ is such an algebra 
for $\gk=\R$. We shall refer 
to the so-defined Lie algebra as the \emph{Poisson structure on the dual to the Lie algebra} 
$\gG$. The corresponding Poisson bi-vector on $V^*$ will be denoted $P_{\gG}$.

In coordinates this Poisson structure is described as follows. Let $\{e_i\}$ be a 
basis in $V$ and $x_i=f_{e_i}$. Then
\begin{equation}\label{? coord}
\{f,g\}=\sum_{i,j,k}c_{ij}^kx_k\frac{\partial f}{\partial x_i}\frac{\partial g}{\partial x_j}
\end{equation}
where $c_{ij}^k$ are the structure constants of $\gG$, and
\begin{equation}\label{Poisson dual}
P_{\gG}=\sum_{i,j,k}c_{ij}^kx_k\xi_i\xi_j.
\end{equation}
Poisson bi-vectors $P_{\gG}$'s  are characterised by the linearity of their components in any cartesian chart on $V^*$. For this reason they are also called \emph{linear}. More generally 
a multi-vector field $Q\in D_{*}(W)$ on an $\R$-vector space $W$ is said to be \emph{linear} 
if its coefficients in a cartesian chart on $W$ are linear. In such a case
$$
Q_{\theta}=\ldb X_{\theta},Q\rdb, \quad \theta\in W,
$$
where $Q_{\theta}$ is the constant vector field that has the same value as $Q$ at $\theta$ 
and $X_{\theta}$ is the constant vector field corresponding to $\theta$. This observation is useful when dealing with linear Poisson structures.

Let $\gG_1$ and $\gG_2$ be Lie algebra structures on a $\gk$-vector space $V$ and $[\cdot,\cdot]_1, [\cdot,\cdot]_2$ the corresponding Lie products. The following is the analogue of Definition\,\ref{Poisson compatibility} for Lie algebras.

\begin{defi}
 Lie algebra structures $\gG_1$ and $\gG_2$ on a $\gk$-vector space $V$ are called \emph{compatible} if $[\cdot,\cdot]\DEF[\cdot,\cdot]_1,+[\cdot,\cdot]_2$
 is a Lie product in $V$.
\end{defi}
The  Lie algebra structure on $V$ defined by the Lie product $[\cdot,\cdot]_1,+[\cdot,\cdot]_2$ will be denoted by 
 $\gG_1+\gG_2$ . Obviously, we have
 \begin{proc}\label{Poisson-lie dual}
Lie algebra structures $\gG_1$ and $\gG_2$ on $V$ are compatible if and only if the corresponding Poisson structures on $V^*$ are 
compatible. $\qquad\qquad\qquad\Box$
\end{proc}
\begin{rmk}
If Poisson structures $P_1,\dots,P_m$ are mutually compatible, then structures
$\lambda_1P_1,\dots,\lambda_mP_m$ with $\lambda_1,\dots\lambda_m\in\gk$
are mutually compatible as well. So, $P=\lambda_1P_1+\dots+\lambda_mP_m$
is a Poisson structure called a \emph{linear combination of structures} $P_i$'s.
 The expression ``\emph{Linear combination of some mutually compatible Lie algebra structures}" has the same meaning.
\end{rmk}
%%%%%%%%%%
\subsection{The Lie rank of Poisson manifolds and Lie algebras}
%%%%%%%%%%%%

Recall that a bi-vector field $Q\in D_2(M)$ generates  a distribution  (with singularities) on 
$M$. This distribution is defined as the $C^{\infty}(M)$--submodule $D_Q(M)$ of 
$D(M)$ generated by vector fields $Q_f=df\rfloor Q, \;f\in C^{\infty}(M)$.

Geometrically,  $D_Q(M)$ may be viewed as a family of vector spaces 
$M\ni x\mapsto \triangle_Q(x)\subset T_{x}M$ on $M$ where the subspace 
$\triangle_Q(x)\subset T_{x}M$ is generated by vectors of the form 
$Q_{f,x}\in T_{x}M,\quad \forall f \in C^\infty(M).$ The function
$$
M\ni x\mapsto {\rm rank}_Q(x) \DEF\dim\triangle_Q(x)
$$
is lower semi-continous with values in even integers. In 
particular, ${\rm rank}_Q(x)$ is locally constant except on a closed subset in $M$ without
internal points and reaches its maximum value, say $2k$, in an open domain of $M$.  
The number $k$ is uniquely characterised by the property that $Q^k\neq 0$, while
$Q^{k+1}=0$ (here $Q^{i}$ stands for the $i$-th wedge power of $Q$). Alternatively, 
$2k$ is equal to the maximal number of independent vector fields of the form $Q_f$. 
The above-said remains valid for bi-vectors with polynomial coefficients on a vector
space over a ground field of characteristic  zero. Such bi-vectors have the same 
rank everywhere except on an algebraic subvariety of $\bf V.$
\begin{defi}\label{Oppa}
\emph{1)} A bi-vector field $Q$ is said to be of \emph{rank} $2k$ if 
$$
Q^k\neq 0,\quad Q^{k+1}=0
$$
\emph{2)} A Lie algebra is said to be of  \emph{Lie rank} $2k$ if the associated linear Poisson 
bi-vector on its dual is of rank $2k$.
\end{defi}

\begin{ex}
If  $2k\leq n$,  then the Lie rank of the Lie algebra
$$
\gG_{n,k}=\underbrace{\gG_2\oplus\dots\oplus\gG_2}_{\mbox{$k$ times}}\oplus\gamma_{n-2k}
$$
where $\gG_2$ is the non-abelian $2$-dimensional Lie algebra and $\gamma_{n-2k}$ the 
$(n-2k)$-dimensional abelian Lie algebra is $2k$.  If $n=2k+\epsilon, \quad\epsilon=0, 1$,
then $\gG_{n,k}$  is an n-dimensional Lie algebra of maximal  Lie rank. 
\end{ex}

If $Q$ is a Poisson bi-vector, then $[Q_f,Q_g]=Q_{\{f,g\}}$, i.e.,  the distribution $D_Q(M)$ 
is a Frobenius distribution. The maximal integral submanifolds of $D_Q(M)$ constitute the 
canonical \emph{symplectic (generalised) foliation} of $M$ associated with $Q$ (see \cite{VKra, Wein}). 
If $Q=P_{\gG}$, then the leaves of this (generalised) foliation are the \emph{orbits of the coadjoint 
representation of $\gG$}.

%\begin{center}
\section{Modularity of Poisson and Lie algebra structures}\label{mdlr}
%\end{center}

In this section, we introduce and study the splitting of Poisson and Lie algebra structures 
into {\it unimodular} and {\it non-unimodular} parts. This splitting is canonical up to a 
gauge transformation and, in a sense, it  reduces the study of general Poisson or Lie algebras 
structures to that of the unimodular ones. The central notions in this construction are the modular 
vector field and the modular class introduced by J.-L. Koszul \cite{Kz}. In the survey 
\cite{Ksmn}, the reader will find an extensive bibliography on this topic. 
 
%%%%%%%%%%%%%%
\subsection{Unimodular Poisson structures}
%%%%%%%%%%%%%%%%%%%%%

If $\omega\in\Lambda^n(M)$ is a volume form, then the map
$$
Q\mapsto Q\mathop{\rfloor} \omega,\qquad Q\in  D_*(M),
$$
which we shall call \emph{$\omega$-duality},
is an isomorphism of $C^{\infty}(M)$--modules from $D_*(M)$ onto 
$\Lambda^*(M).$ In particular, the $(n-2)$--form 
$\alpha=\alpha_{P,\omega}=P\rfloor\omega$, $\omega$-dual to the Poisson bi-vector 
$P$, completely characterises this bi-vector. 
\begin{proc}\label{Dual}
$P \in D_2(M)$ is a Poisson bi-vector on $M$ if and only if
\begin{equation}\label {DUAL}
d(P\rfloor\alpha)=2P\rfloor d\alpha \quad\mathrm{with}\quad  \alpha=P\rfloor\omega.
\end{equation}
\end{proc}
\begin{proof}
For $P=Q$ it easily follows from (\ref{Schouten-Lie}) and (\ref{L as derivation})  that
$$
L_{P\wedge P}-2i_P\circ L_{P}= i_{\ldb P,P\rdb}
$$
and, as a consequence, that
\begin{equation}\label{PPP}
L_{P\wedge P}(\omega) -2P\mathop{\rfloor}L_{P}(\omega) = \ldb 
P,P\rdb\mathop{\rfloor}(\omega)
\end{equation}

On the other hand,  by Definition (\ref{Lie derivative}) for $Q=P$ and $P\wedge P$,
we have
\begin{equation}\label{Pdd}
L_{P}(\omega)=-d(P\mathop{\rfloor}\omega)=-d\alpha
\end{equation}
and
$$
 L_{P\wedge P}(\omega)=-d(i_{P\wedge P}(\omega))=
 -d(P\mathop{\rfloor}(P\mathop{\rfloor}\omega))=-d(P\mathop{\rfloor}\alpha)
$$
 With these substitutions (\ref{PPP}) becomes
$$
 d(P\mathop{\rfloor}\alpha)-2P\mathop{\rfloor} d\alpha=
-\ldb P,P\rdb\mathop{\rfloor}\omega
$$
and the desired result follows from the fact that a multi-vector $Q\in D_{*}(M)$ vanishes
iff $Q\mathop{\rfloor}\omega=0$. 
\end{proof}

\begin{defi}\label{UUU}
A Poisson structure on $M$ is called \emph{unimodular} with respect to $\omega$, or $\emph{$\omega$--unimodular}$, if $ L_{P}(\omega)=0.$
\end{defi}

\begin{proc}\label{Uni}
A Poisson structure $P$ is unimodular if and only if one of the following relations holds
\begin{enumerate}
\item $
d\alpha=0\quad\mbox{with}\quad\alpha=P\mathop{\rfloor}\omega$;
\item $L_{P_{f}}(\omega)=0,\quad\forall\, f\in{\rm C}^\infty(M),$
i.e., the Lie derivatives of $\omega$ along all $P$--Hamiltonian vector fields vanish.
\end{enumerate}
\end{proc}

\begin{proof} In view of (\ref{Pdd})  the first assertion is obvious. Then, by 
Definition (\ref{Lie derivative}), we have
$$
\begin{array}{l}
L_{P_{f}}(\omega)=d(P_{f}\mathop{\rfloor}\omega)=
-d((df\mathop{\rfloor}P)\mathop{\rfloor}\omega)=
-d(df\wedge(P\mathop{\rfloor}\omega))=\\df\wedge d(P\mathop{\rfloor}\omega) 
=-df\wedge L_{P}(\omega),
\end{array}
$$
i.e.,
\begin{eqnarray}\label{LPf}
L_{P_{f}}(\omega)=-df\wedge L_{ P}(\omega)
\end{eqnarray}
If  $\rho\in \Lambda^k( M),\, k<n,$ and $df\wedge\rho=0$ holds for all $f\in C^\infty( M),$ 
then, obviously, $\rho=0$. Now this observation applied to (\ref{LPf}) proves the second assertion.
\end{proof}
\begin{rmk}\label{COO}
A direct consequence of \emph{(\ref{LPf})} is that $P$ is unimodular if and only if
$L_{P_{x_i}}(\omega)=0,\,i=1,\dots,n,$ for any local chart 
$(x_1,\dots,x_n)$ on $M.$
\end{rmk}

Compatibility conditions  for unimodular Poisson structures simplify as follows.
\begin{proc}\label{PPO}
Let $P_1,P_2,$ be $\omega$-unimodular Poisson structures. They are 
compatible if and only if $L_{P_1\wedge P_2}(\omega)=0$.
\end{proc}

\begin{proof} Formula (\ref{Schouten via L}) for $P=P_1, \,Q=P_2$ applied to $\omega$ gives:
$$
L_{P_1\wedge P_2}(\omega)=P_1\rfloor L_{P_2}(\omega)+P_2\rfloor L_{P_1}(\omega) + 
\ldb P_1,P_2\rdb\mathop{\rfloor}\omega.
$$
Due to the unimodularity of $P_1$ and $P_2$, it reduces to 
$$
L_{P_1\wedge P_2}(\omega)=\ldb P_1,P_2\rdb\mathop{\rfloor}\omega.
$$
and it remains to note that $\ldb P_1,P_2\rdb =0$ is equivalent to 
$\ldb P_1,P_2\rdb\mathop{\rfloor}\omega=0$.
\end{proof}
\begin{cor}\label{PROD}
Two $\omega$--unimodular Poisson structures $P_1$ and $P_2$ are compatible if 
$P_1\wedge P_2=0.\qquad\Box$
\end{cor}

\begin{cor}\label{prod}
Any two $\omega$--unimodular Poisson structures on a 3-dimensional manifold are 
compatible.
\end{cor}
\begin{rmk}\label{LOB}
The condition $L_{P_1\wedge P_2} =0$, which, obviously, guarantees the
compatibility of $P_1$ and $P_2$ is not, in fact, weaker 
than $P_1\wedge P_2=0$, since $L_Q=0$ is equivalent to $Q=0$ 
for any $Q\in D_*(M)$.
\end{rmk}

%%%%%%%%%%%%%%
\subsection{The Lie rank of unimodular Lie algebras}
%%%%%%%%%%%%%%%%%

The following example shows the existence of unimodular odd-dimensional Lie algebras of
maximal Lie rank, i.e., of rank $2k$ if $n=2k+1$.
 \begin{ex}
Let $V$ be a vector space of dimension $2k+1$ and $(x_1,x_2,\dots,x_{2k+1})$ a cartesian chart on $V^*$, and $(\xi_1,\xi_2,\dots,\xi_{2k+1})$ the dual basis. Then
$$
P=x_{2k+1}(\xi_1\wedge\xi_2+\dots +\xi_{2k-1}\wedge\xi_{2k})
$$
is a Poisson bi-vector, as well as bi-vectors $P_s=x_{2k+1}\xi_{2s-1}\wedge\xi_{2s}, \;s=1,\dots,k$, of  rank two.
Obviously, the rank of $P$ is $2k$, the $P_s$'s are mutually compatible and $P=P_1+\dots+P_k$.
So, the Lie algebra structures $\gG, \gG_1,\dots, \gG_k$ on $V$ corresponding $P, P_1,\dots, P_k$, respectively,
are mutually compatible, $\gG=\gG_1+\dots+\gG_k$ and the Lie rank of $\gG$ is $2k$. 
\end {ex}

Now we shall show that the Lie rank of an unimodular Lie algebra is less than 
its dimension.

\begin{proc}\label{XUP}
If $Q$ is an $\omega$-unimodular Poisson structure of rank $2k=n$ on a manifold $M$, of dimension $2k$ 
then $Q^k\rfloor\omega$ is a constant function, different from zero.
\end{proc}
\begin{proof} Since $Q$  is a Poisson structure, then, according to (\ref{Schouten via L}), 
$L_{Q^2}=2i_Q\circ L_Q$. Inductively, one easily finds that
$$
L_{Q^s}=si_{Q^{s-1}}\circ L_Q,\quad s\ge 2.
$$
If $Q^k\ne 0$, then the  function $f=Q^k\rfloor\omega$ is 
different from zero.

But  $L_Q(\omega)=0$, since $Q$ is $\omega$-unimodular ,  and 
hence
$$
df=d(Q^k\rfloor\omega)=L_{Q^k}(\omega)=kQ^{k-1}\rfloor L_Q(\omega)=0. 
$$
\end{proof}
\begin{cor}\label{6.1f}
The Lie rank of an unimodular Lie algebra $\gG$ is less than $\dim\gG$.
\end{cor}
\begin{proof} Assume that the rank  of the associated Poisson bi-vector $P_{\gG}$ 
on $\vert\gG\vert^*$ is equal to $\dim \gG=2k$. If $\omega$ is a cartesian (``constant") volume form on $\vert\gG\vert^*$, then $f=P^k\rfloor \omega$ is a homogeneous polynomial of order $k$ on $\vert\gG\vert^*$. This contradicts the fact that, according to Proposition
\ref{XUP}, $f$ is a non-zero constant.
\end{proof}

%%%%%%%%%%%%%%%%%%%%%%
\subsection{Modular vector fields}
%%%%%%%%%%%%%%%%%%%%%%

Let $\omega$ be a volume form and $P$ a Poisson bi-vector on a manifold $M$.
\begin{defi}
The vector field $\Xi=\Xi_{P,\omega}$ defined by the relation
\begin{equation}\label{POL}
\Xi\mathop{\rfloor}\omega=d(P\rfloor\omega)
\end{equation}
is called the \emph{modular vector field} of $P$  with respect to $\omega$, or the  
\emph{$\omega$--modular vector field}.
\end{defi}

It follows from (\ref{Pdd}) and (\ref{LPf})  that
$$
L_{P_f}(\omega)=df\wedge (\Xi\mathop{\rfloor}\omega)
$$
On the other hand, $0=\Xi\mathop{\rfloor}(df\wedge\omega)= 
\Xi(f)\omega-df\wedge(\Xi\mathop{\rfloor}\omega).$ So,
\begin{equation}\label{DIV}
L_{P_f}(\omega)=\Xi(f)\omega. 
\end{equation}
This shows that $\Xi(f)={\rm div}_\omega (P_f)$ so that  
$\Xi=\Xi_{P,\omega}$ measures the divergence of $P$-Ha\-mil\-to\-ni\-an vec\-tor 
fields with respect to the volume form $\omega.$ This property was the 
original definition of modular vector fields.

\begin{proc}\label {LIST} 
Let $\Xi=\Xi_{P,\omega}$ be the modular field of a Poisson structure $P$ and 
$\alpha=\alpha_{P,\omega}=P\rfloor\omega$. Then the following relations hold:
$$
\begin{array}{lcc}
({\rm i})\quad L_{\Xi}(\alpha)=0;\\
({\rm ii})\quad L_P(d\alpha)=0;\\
({\rm iii})\quad  L_P(\alpha)=-\Xi\mathop{\rfloor}\alpha;\\
({\rm iv})\quad L_\Xi(\omega)=0;\\
({\rm v})\quad L_\Xi(P)=\ldb P,\Xi\rdb=\partial_P(\Xi)=0;\\
({\rm vi})\quad P_f\mathop{\rfloor}\omega=-df\wedge\alpha, \forall f \in C^\infty(M);\\
({\rm vii})\quad [\Xi, P_{f}]=P_{\Xi(f)}, \forall f \in C^\infty(M).
\end{array}
$$
\end{proc}

\begin{proof}
First, by definition,
$
 L_{\Xi}(\alpha)=d(\Xi\mathop{\rfloor}\alpha)+\Xi\mathop{\rfloor}d\alpha.
$
Also
\begin{equation}\label{pre} 
\Xi\mathop{\rfloor}\alpha= \Xi\mathop{\rfloor}(P\mathop{\rfloor}\omega)= 
P\mathop{\rfloor}(\Xi\mathop{\rfloor}\omega)= P\mathop{\rfloor}d\alpha
\end{equation}
and hence, in view of (\ref{DUAL}), we have
$$
d(\Xi\mathop{\rfloor}\alpha)=d(P\mathop{\rfloor}d\alpha)= 
\frac{1}{2}d^2(P\mathop{\rfloor}\alpha)=0
$$
On the other hand,
 $\Xi\mathop{\rfloor}d\alpha=
\Xi\mathop{\rfloor}(\Xi\mathop{\rfloor}\omega)=0$. This proves (i). 

Similarly,
$$
L_P(d\alpha)=[i_P,d]^{gr}(d\alpha)=-d(P\mathop{\rfloor}d\alpha)=0
$$
This proves (ii).

In turn, (iii) directly follows  from (\ref{DUAL}) and (\ref{pre}):
$$
L_P(\alpha)=[i_P,d](\alpha)=P\mathop{\rfloor}d\alpha- 
d(P\mathop{\rfloor}\alpha)=-P\mathop{\rfloor}d\alpha= 
-\Xi\mathop{\rfloor}\alpha.
$$
The simple computation
$$
L_\Xi(\omega)=d(\Xi\mathop{\rfloor}\omega)=d^2(P\rfloor\omega)=0
$$
proves (iv). 
To prove (v) it suffices to show that
  $L_\Xi(P)\mathop{\rfloor}\omega=0.$ But according to (i) and (\ref{[L_X,i_P]}) we have
 $$0=L_\Xi(\a)=L_\Xi(P\mathop{\rfloor}\o)=
-L_\Xi(P)\mathop{\rfloor}\omega+P\mathop{\rfloor}L_\Xi(\omega)= 
-L_\Xi(P)\mathop{\rfloor}\omega.
$$

Now, a particular case of (\ref{Schouten}) is $i_{\ldb f, P\rdb}=[[f,d]^{gr},i_P]^{gr}=-[df,i_P]$, and one gets (vi) 
by applying this relation to $\omega$. 

Finally, recall the general formula
$$
L_X(\rho\mathop{\rfloor}Q)=-L_X(\rho)\mathop{\rfloor}Q+\rho 
\mathop{\rfloor}L_X(Q)
$$
with $X\in D(M),\quad Q\in D_i(M),\quad \rho\in\Lambda^j(M).$ By applying 
this formula to $X=\Xi,\quad Q=P$ and $\rho=df$ one finds that
$$
[ \Xi, P_f]=-L_\Xi(P_f)=L_\Xi(df\rfloor P)=-L_{\Xi}(df)\rfloor P+df\rfloor L_{\Xi}(P).
$$
Now (vii) directly follows from (v) by observing that $L_\Xi(df)=d\,\Xi(f)$. 
\end{proof}

Additivity is an important property of $\omega$--modular fields.
\begin{proc}\label{PFA}
If Poisson structures $P_1$ and $P_2$ are compatible and $\Xi_1,\Xi_2$
 are $\omega$--modular fields for $P_1$ and $P_2$, respectively, then
$\Xi_1+\Xi_2$ is the $\omega$--modular field of $P_1+P_2.$
\end{proc}

{\it Proof.} The proof is straightforward from 
(\ref{DIV}).$\qquad\Box$
\begin{cor}\label{BUU} If $P_i,  \,\Xi_i, \,i=1,2$, are as above, then
\begin{equation}\label{buu}
L_{\Xi_1}(P_2)+L_{\Xi_2}(P_1)=0
\end{equation}
\end{cor}
{\it Proof.} The proof follows from $L_{\Xi_1}(P_1)=L_{\Xi_2}(P_2)=
L_{\Xi_1+\Xi_2}(P_1+P_2)=0$ (Proposition\;\ref{LIST} (v)). $\qquad\Box$

 %%%%%%%%%%%%%%%%%
 \subsection{The $\omega$--modular class}
 %%%%%%%%%%%%%%%%%%%%%
 
The following known fact describes the dependence of the $\omega$--modular field on $\omega$. 
For completeness we give a short proof of it.
\begin{proc}\label{GAU}
If $\omega^\prime=f\o$ is another volume form on $M$, then
\begin{equation}\label{gau}
\Xi_{P,f\omega}=\Xi_{P,\omega}-P_{\ln|f|}
\end{equation}
\end{proc}
{\it Proof.} By definition the vector field $\Xi_{P,f\omega}$ is the unique solution of
\begin{equation}\label{gag}
\Xi_{P,f\omega}\mathop{\rfloor}(f\omega)=d\alpha_{P,f\omega}
\end{equation}
and $\alpha_{P,f\omega}=P\mathop{\rfloor}(f\omega)=f\alpha$ where 
$\alpha=\alpha_{P,\omega}$. If $\Xi_{P,f\omega}=\Xi_{P,\omega}+Y$, then  (\ref{gag})
may be rewritten as
$$
fY\mathop{\rfloor}\omega=df\wedge\alpha\Longleftrightarrow 
Y\mathop{\rfloor}\omega=d(\ln|f|)\wedge\alpha
$$
($f$ is nowhere zero, since $f\o$ is a volume form). Now Proposition \ref{LIST}, 
(vi), shows that $Y=-P_{\ln|f|}.\qquad\Box$

This result has the following cohomological interpretation. By
Proposition \ref{LIST} (v), $\partial_P(\Xi)=0,$
i.e., $\Xi$ is a 1-cocycle of the complex $\{ D_*(M), \partial_P\}$. On the other hand, 
$P$--Hamiltonian fields are co-boundaries of this complex:
$P_g=\partial_P(g)$, for $g \in C^\infty(M)$. Hence Proposition \ref{GAU} yields:

\begin{cor}\label{COH}
The cohomology class of the $\omega$--modular field $\Xi_{P,\omega}$ in 
$\{D_*(M),\partial_P\}$ does not depend on $\omega$ and, therefore, is 
well-defined by $P$. $\qquad\Box$
\end{cor}
\begin{defi}
The $\partial_P$-cohomology class of the $\omega$-modular field of $P$ is called 
the \emph{modular} class of $P$.
\end{defi}
\begin{cor}\label{HMM}
A Poisson structure $P$  is $\omega$--unimodular with respect to a volume form 
$\omega$ if and only if its modular class vanishes.$\qquad\Box$
\end{cor}
If $P$ is non-degenerate, then $(M, \gamma_P(P))$ is a symplectic
manifold. In this case the isomorphism $\gamma_P:\Lambda^*(M) \rightarrow D_*(M)$  
is also an isomorphism of
complexes  $\{\Lambda^*(M),d\}$ and $\{D_*(M),d_P\}$.
Therefore, if $H^1(M)=0,$ then any non-degenerate Poisson structure on $M$ is $\o$--unimodular
with respect to a suitable volume form $\o.$

%%%%%%%%%%%%%%%
\subsection{The modular disassembling of a Poisson structure}
%%%%%%%%%%%%%%%%%

Now we shall show that the $\omega$--modular vector field of a Poisson structure $P$ 
allows one to disassemble this structure, at 
least locally,  into two parts, one of which is 
$\omega$--unimodular, while  all ``$\omega$--non--unimodularity" 
of $P$ is concentrated in the second part.

\begin {proc}\label{SPLIT}
Let $\Xi$ be the $\omega$--modular vector field of a Poisson structure $P$ and 
$\nu\in C^\infty(M)$ be such 
that $\Xi(\nu)=1$. Then
\begin{enumerate}
\item $\Xi\wedge P_\nu$ is a Poisson structure compatible with $P$;
\item $L_{\Xi\wedge P_\nu}(\omega)=-L_P(\omega)$;
\item $P+{\Xi\wedge P_\nu}$ is an  $\omega$-unimodular Poisson structure
compatible with $P$, and $\nu$  is a Casimir function of $P+{\Xi\wedge P_\nu}$;
\item $P_\nu\wedge \Xi=\ldb P, \nu\Xi\rdb=\partial_P(\nu\Xi)$.
\end{enumerate}
\end{proc}
\begin{proof} First, from Proposition \ref{LIST} (vii), we see that 
$$
\ldb \Xi, P_{\nu}\rdb=[\Xi, P_{\nu}]=P_{\Xi(\nu)}=P_1=0.
$$
Since the Schouten bracket is a graded bi-derivation of the exterior algebra $D_{*}(M)$, this implies
$$
\ldb\Xi\wedge P_\nu, \Xi\wedge P_\nu\rdb=0.
$$
Hence $\Xi\wedge P_\nu$ is a Poisson structure. By the same reason,
$$
\ldb P,\Xi\wedge P_\nu\rdb=\ldb P,\Xi\rdb\wedge P_\nu- \Xi\wedge\ldb P,P_\nu\rdb
$$
But $\ldb P,P_\nu\rdb=0$ since $P_\nu$ is a $P$--Hamiltonian field.
So, the right-hand side of the above equality vanishes by Proposition 
\ref{LIST} (v). Therefore,
$\Xi\wedge P_\nu$ is compatible with $P.$

To prove the second assertion, we let $P=\Xi, \;Q=P_\nu$  in  (\ref{L as derivation})  
and then apply the result to $\omega$:
$$
L_{\Xi\wedge P_\nu}(\o)=\Xi\mathop{\rfloor}L_{P_\nu}(\omega)- 
L_\Xi(P_{\nu}\mathop{\rfloor}\omega)
$$
Similarly, when Formula (\ref{[L_X,i_P]}), with $X=\Xi, \;Q=P_\nu$, is applied to 
$\omega$, we obtain
$$
L_\Xi(P_\nu\mathop{\rfloor}\omega)=-L_\Xi(P_\nu)\mathop{\rfloor}\omega+ 
P_\nu\mathop{\rfloor}L_{\Xi}(\omega)
$$
Since, by Proposition \ref{LIST} (v), (vii), $L_{\Xi}(\omega)=0$  and $L_\Xi(P_\nu)=0$, 
this shows that
$$
L_{\Xi\wedge P_\nu}(\o)=\Xi\mathop{\rfloor}L_{P_\nu}(\omega)
$$
On the other hand, according to (\ref{DIV}), $L_{P_\nu}(\omega)=\Xi(\nu)\omega=\omega$ and hence
$$
L_{\Xi\wedge P_\nu}(\o)=\Xi\rfloor\omega=d\alpha=-L_P(\omega).
$$

In turn, the $\omega$--unimodularity of $P+\Xi\wedge P_\nu$
directly follows from Assertion (2), and
$$
d\nu\mathop{\rfloor}(P+\Xi\wedge  P_\nu)=d\nu\mathop{\rfloor}P+ 
\Xi(\nu){P_\nu}=0
$$
proves that $\nu$ is a Casimir function of this structure.

Finally, the fact that $\partial_P$ is a graded derivation of the exterior algebra 
$D_{*}(M)$
together with $P_\nu=\ldb P, \nu\rdb, \;\ldb P, \Xi\rdb=0$ (Proposition\;\ref{LIST} (v)) proves the last assertion.
\end{proof}
\begin{rmk}\label{loc-nu}
The equation $\Xi(\nu)=1$ admits solutions only in the open domain
$U$ of $M$ in which the vector field $\Xi$ does not vanish and these solutions
are, generally, local.  Accordingly, the function $\nu$
in Proposition\,\ref{SPLIT} is defined only locally, in $U$.
\end{rmk}
Thus,  $P$ is the sum 
\begin{equation}\label{BRA}
P=(P-P_\nu\wedge\Xi)+P_\nu\wedge\Xi
\end{equation}
of two compatible Poisson structures, the first of which
is $\omega$-unimodular, while the second one is 
$\omega$-non-unimodular (if different from zero). Accordingly, $P-P_\nu\wedge\Xi$
is the \emph{$\omega$-unimodular part} of $P$ and $P_\nu\wedge\Xi$ is its
\emph{$\omega$-non-unimodular part}. Note that $P_\nu\wedge\Xi$ is of rank 2 
(if different from zero), and as such it is a $\omega$--non--unimodular Poisson structure 
of the smallest possible dimension, which is compatible with $P$.

Formula (\ref{BRA}) is called an \emph{$\omega$-modular disassembling} of $P$. It is,
obviously, unique only up to Poisson bi-vectors of the form $P_f\wedge\Xi$
with $\Xi(f)=0$. Nevertheless, the following interpretation of (\ref{BRA}) in terms 
of Poisson brackets reveals a part of it that is invariant, i.e., independent on the choice of
the normalising function $\nu$.

Denote by $\{\cdot,\cdot\}_{non}$ (resp., by $\{\cdot,\cdot\}_{uni}$) the Poisson 
bracket corresponding to the $\omega$-non-uni-modular
 (resp., $\omega$-unimodular) part of $P$. Then
$$
\{f,g\}_{non}=\{f,\nu\}\Xi(g)-\{g,\nu\}\Xi(f).
$$
If $P$--Hamiltonian fields $P_f$ and $P_g$ are $\omega$--divergenceless i.e., 
$\Xi(f)=\Xi(g)=0$ (see (\ref{DIV})), then
$$
\{f,g\}_{non}=0 \qquad \mbox{is equivalent to}\qquad \{f,g\}_{uni}=\{f,g\}.
$$
This shows that the restriction of the bracket $\{\cdot,\cdot\}_{uni}$ to the
$\omega$--divergenceless part of the original Poisson structure does not depend 
on the choice of the normalising function $\nu$ (see (\ref{BRA})).

The Poisson bi-vector $P_\nu\wedge\Xi$ is a
\emph{$\omega$-modular bi-vector} associated with $P$ in the sense of Definition\,\ref{PNU} below.
All $\omega$-modular bi-vectors associated with $P$ are compatible each other. 
Indeed, by Proposition\;\ref{LIST} (vii), $[P_{f}, \Xi]=0$ 
if $\Xi(f)=\const$. Since the Schouten bracket is a graded bi-derivation of $D_{*}(M)$, 
this shows
$$
\ldb P_\nu\wedge\Xi, P_{\mu}\wedge\Xi\rdb=0 \quad \mathrm{if} \quad \Xi(\nu)=\Xi(\mu)=1.
$$

The following proposition shows that an $\omega$-modular bi-vector coincides with its
$\omega$-non-unimodular part.
\begin{proc}\label{IDE} If $P,\omega,\Xi$ and $\nu$ are as above, then
\begin{enumerate}
\item the $\omega$--modular field of the Poisson structure 
 $P_\nu\wedge \Xi$  coincides with  $\Xi$;
\item if $\Xi(\mu)=1$, then  $(P_\nu\wedge\Xi)_\mu\wedge\Xi=P_\nu\wedge\Xi$.
\end{enumerate}
\end{proc}
\begin{proof}
Since $L_Q=d(Q\rfloor\omega)$ if $Q\in D_2(M)$, the first assertion is just an interpretation of Proposition \ref{SPLIT}(2). The second one follows from:
$$
(P_{\nu}\wedge\Xi)_\mu=-d\mu\mathop{\rfloor}(P_\nu\wedge\Xi)= 
\Xi(\mu)P_\nu-\{\mu,\nu\}\Xi.
$$
\end{proof}

Now we are ready to define general $\omega$--modular bi-vectors. To this end we 
need the following generalisation of Proposition\,\ref{IDE}.
\begin{proc}\label{UAB}
Let $X,\Xi\in D(M)$ and the volume form $\omega\in\Lambda^n(M)$ be such that
\begin{equation}\label{uab}
[X,\Xi]=0, \quad L_\Xi(\omega)=0,\quad L_X(\omega)=\omega .
\end{equation}
Then $P=X\wedge\Xi$ is a Poisson structure, which coincides with its 
$\omega$--non--unimodular part, and $\Xi$ is the  $\omega$--modular field of $P$.
\end{proc}

\begin{proof}
First, observe that conditions (\ref{uab}) assures independence of vector fields
$X$ and $\Xi$. Since $\ldb X\wedge\Xi,X\wedge\Xi\rdb 
=2[X,\Xi]\wedge X\wedge\Xi$, the condition $[X,\Xi]=0$ implies that $P$ is a 
Poisson structure. On the other hand, in view of (\ref{[L_X,i_P]}), (\ref{L as derivation}) 
and (\ref{uab}) we have
$$
\begin{array}{l}
L_{P}(\omega)=X\mathop{\rfloor}L_\Xi(\omega)-L_X(\Xi\mathop{\rfloor}\omega)=
L_X(\Xi)\mathop{\rfloor}\omega-\Xi\mathop{\rfloor}L_X(\omega)= 
-\Xi\mathop{\rfloor}\omega .
\end{array}
$$
This shows (see (\ref{Pdd}) and (\ref{POL})) that $\Xi$ is the 
$\omega$--modular vector field of $P$.

Since $X$ and $\Xi$  commute, there exists, at least locally,  a function $\nu$ such that 
$\Xi(\nu)=1,\quad X(\nu)=0$. For such a function $\nu$, $P_\nu=-d\nu\rfloor(X\wedge\Xi)=X,$ 
i.e., locally, $P=P_\nu\wedge\Xi$ and, therefore, by Proposition\,\ref{IDE} (2),
$P$ coincides with its own $\omega$--non--unimodular part. 
\end{proof}

\begin{defi}\label{PNU} A Poisson bi-vector described in Proposition \emph{\ref{UAB}}
 is called an \emph{$\omega$--modular bi-vector}.
\end{defi}
As a direct consequence of the $\omega$--modular disassembling (\ref{BRA}) and of Proposition \, \ref{UAB}, we obtain
\begin{cor}
A Poisson structure, which coincides with its $\omega$-non-unimodular part is locally
an $\omega$--modular bi-vector.
\end{cor}
In what follows, we shall assume, when 
referring to an $\omega$--modular bi-vector of the form $X\wedge\Xi$, that the conditions of Proposition\,\ref{UAB} are satisfied.

\begin{rmk}
If $X,Y\in D(M)$, then, as it is easy to see, 
$\ldb X\wedge Y, X\wedge Y\rdb=2[X,Y]\wedge X\wedge Y$. Hence the bi-vector $X\wedge Y$ 
is Poisson if and only if the distribution generated by $X$ and $Y$ is integrable. In particular, 
the distribution associated with an $\omega$--modular bi-vector is integrable.
\end{rmk}

%%%%%%%%%%%%%%%%%%%%%%%%
\subsection{Compatibility of $\omega$--modular bivectors}
%%%%%%%%%%%%%%%%%%%%%%%%%

Now we shall discuss compatibility conditions involving  $\omega$--modular bivectors.  First, we consider
the inverse of the $\omega$--modular splitting procedure.
\begin {proc}\label{VUB}
An  $\omega$--modular bi-vector $X\wedge\Xi$ and an $\omega$-unimodular  structure $Q$ are compatible if and only if
$$
L_\Xi(Q)=0,\quad\Xi\wedge L_X(Q)=0.
$$
\end{proc}
\begin{proof}
First, observe that $\Xi$ is the  $\omega$--modular field of
the Poisson structure $X\wedge\Xi+Q$ (Proposition\,\ref{PFA}). 
Therefore, in view of Proposition\,\ref{LIST}, (v), 
$$
0=L_\Xi(X\wedge\Xi+Q)=L_\Xi(Q).
$$ 
On the other hand, the compatibility 
condition of $X\wedge\Xi$ and  $Q$ is
$$
\begin{array}{l}
0=\ldb X\wedge\Xi,Q\rdb=X\wedge\ldb\Xi,Q\rdb-\ldb X,Q\rdb
\wedge\Xi=
L_\Xi(Q)\wedge X+ L_X(Q)\wedge \Xi .
\end{array}
$$
Since $L_\Xi(Q)=0$, this gives  the desired result.
\end{proof} 

\begin{rmk}
Generally, $X\wedge\Xi$ is not an $\omega$--modular bi-vector associated with $P=X\wedge\Xi+Q$.
For example, if $M=\R^2,  \,\omega=dx_1\wedge dx_2, X=x_1\xi_1, \,\Xi=\xi_2$ and $Q=\xi_1\xi_2$, then
$X\wedge\Xi$ is an  $\omega$--modular bi-vector compatible with the unimodular Poisson bi-vector $Q$. But $P$ itself is an $\omega$--modular bi-vector and as such coincides 
with its $\omega$-non-unimodular part. 
\end{rmk}

Now we shall discuss the compatibility of two $\omega$-modular bi-vectors.
Assume $X_i,\Xi_i\in D(M), i=1,2,$ to be as in Proposition \ref{UAB}. 
By developing the compatibility condition $\ldb P_1,P_2\rdb=0$ of $\omega$-modular 
bi-vectors $P_1=X_1\wedge \Xi_1$ and $P_2=X_2\wedge{\Xi}_2$,
we obtain
\begin{eqnarray}\label{kom}
 \ldb X_1,X_2\rdb\wedge\Xi_1\wedge\Xi_2 +
\ldb\Xi_1,\Xi_2 \rdb\wedge X_1\wedge X_2=\nonumber\\
 \ldb X_1,\Xi_2 \rdb\wedge \Xi_1\wedge X_2+
\ldb \Xi_1,X_2 \rdb\wedge X_1\wedge\Xi_2.
\end{eqnarray}

Note that (\ref{kom}) only takes into account the product structure of $P_1$ and $P_2$. So,
their modularity properties are to be additionally taken into consideration. Since (\ref{kom}) guarantees that $P=P_1+P_2$ is a Poisson bivector,
its modular vector field is $\Xi=\Xi_1+\Xi_2$ and Proposition\,\ref{LIST}  and its consequences are valid for these $P$ and $\Xi$. 

In this case, Formula (\ref{buu}) becomes 
\begin{equation}\label{lko}
\ldb X_2,\Xi_1\rdb\wedge\Xi_2+\ldb X_1,\Xi_2 \rdb\wedge \Xi_1
+(X_1-X_2)\wedge\ldb\Xi_1,\Xi_2 \rdb=0
\end{equation}
Note that (\ref{lko}) is a formal consequence of (\ref{kom}) and the modularity property of the
$P_i$'s. Moreover, taking into account relation (\ref{lko})
multiplied by $X_2$, one can bring Formula (\ref{kom}) to the form
\begin{equation}\label{rdc}
\ldb X_1,X_2\rdb\wedge\Xi_1\wedge\Xi_2= \ldb \Xi_1,X_2 \rdb \wedge (X_1-X_2) 
\wedge\Xi_2.
\end{equation}
Similarly, multiplication of (\ref{lko}) by $X_1$ leads to
\begin{equation}\label{rds}
\ldb X_1,X_2\rdb\wedge\Xi_1\wedge\Xi_2= \ldb X_1,\Xi_2 \rdb \wedge (X_1-X_2) 
\wedge\Xi_1.
\end{equation}
This proves
\begin{proc}\label{OMP}
$\omega$-modular bi-vectors $X_1\wedge\Xi_1$ and  $X_2\wedge\Xi_2$ are 
compatible if and only if \emph{(\ref{lko})} and one of Formulae \emph{(\ref{rdc})} or \emph{(\ref{rds})} 
holds. $\quad \Box$
\end {proc}
\begin{rmk}\label{TCN}
Condition \emph{(\ref{buu})} is manifestly satisfied if $\Xi_1=\lambda\Xi_2, \quad 0\ne \lambda\in C^{\infty}(M)$,
and hence two $\omega$-modular bi-vectors are compatible if their $\omega$-modular 
fields are proportional.
\end{rmk}

%%%%%%%%%%%%%%%%%%%%%%%%%%%
\subsection{On the complexity of the matching problem}\label{comp-match}
%%%%%%%%%%%%%%%%%%%%%%%%%%%%

Let $\mathcal{P}_i$  (resp.,  $\mathcal{G}_i$), $i=1,\dots, m$, be diffeomorphism 
(resp., isomorphism) classes of Poisson (resp., Lie algebras) structures on a manifold 
(resp., vector space). The \emph{matching problem} consists 
of classifying various realisations 
$P_i$'s (resp., $\gG_i$'s) of these structures in such a way that all
$P_i$'s (resp., all $\gG_i$'s) are mutually compatible. 
Such a realisation will be called a \emph{matching}. The equivalence of two matchings 
is  defined in an obvious manner, and the matching problem consists of describing the
equivalence classes of matchings of given Poisson (resp., Lie algebra) structures.
In full generality, this problem seems to be very difficult. Below, we shall discuss it for two $\omega$--modular bi-vectors. One of our goals here is to give an idea of complexity
of the matching problem and to test the available techniques. Our approach to the problem
is to solve the compatibility conditions for $\omega$--modular bi-vectors $P_1=X_1\wedge \Xi_1$ and $P_2=X_2\wedge \Xi_2$ on the simplifying assumption that $\Xi_1$ and $\Xi_2$ 
are independent and that $[\Xi_1,\Xi_2]=0$. The second condition is not restrictive for Lie algebras, while the first one is not essential in view of
 Remark\,\ref{TCN}. 

With these assumptions, Formula (\ref{lko}) becomes
$$
\ldb X_2,\Xi_1\rdb\wedge\Xi_2+\ldb X_1,\Xi_2\rdb\wedge\Xi_1=0,
$$
or, equivalently,
\begin{equation}\label{yhi}
\ldb X_1,\Xi_2\rdb=f_1\Xi_1+\lambda\Xi_2, \qquad
\ldb X_2,\Xi_1\rdb=\lambda\Xi_1+f_2\Xi_2
\end{equation}
for some functions $f_1, f_2, \lambda\in C^\infty(M)$. Now each of Formulae
(\ref{rdc}) and (\ref{rds}) can be brought to the form
$$
(\ldb X_1,X_2\rdb- \lambda( X_1-X_2))\wedge\Xi_1\wedge\Xi_2=0 .
$$
The last relation is equivalent to
\begin{equation}\label{ihy}
\ldb X_1,X_2\rdb=\lambda( X_1-X_2)+\mu_1\Xi_1-\mu_2\Xi_2
\end{equation}
for some functions $\mu_1,\mu_2\in C^\infty(M)$.

\begin{lem}\label{UNI}
 If $X_1,X_2,\Xi_1, \Xi_2$
are as above, then functions $f_1,f_2,\lambda,\mu_1,\mu_2\,$ in 
\emph{(\ref{yhi})} and \emph{(\ref{ihy})} are subject to the following restrictions:
\begin{eqnarray}\label{erl}
\Xi_1(\lambda)=\Xi_2(\lambda)=
\Xi_1(f_1)=\Xi_2(f_2)=0\nonumber\\
\Xi_1(\mu_2)=2f_2\lambda+X_1(f_2)\nonumber\\
\Xi_2(\mu_1)=2f_1\lambda+X_2(f_1)\nonumber\\
\lambda^2+f_1f_2=-\Xi_1(\mu_1)-X_1(\lambda)= -\Xi_2(\mu_2)-X_2(\lambda).
\end{eqnarray}
\end{lem}
\begin{proof}
These relations are consequences of  Relations (\ref{yhi}), (\ref{ihy}) and of the Jacobi identities involving the vector fields $X_1,X_2, \Xi_1,\Xi_2$. For instance, relations 
$\Xi_1(f_1)=\Xi_1(\lambda)=0$ come from the Jacobi identity for $X_1, \Xi_1,\Xi_2$. 
\end{proof} 

If functions $g_1$ and $g_2$ are such that $\Xi_1(g_2)=-f_2,  \,\Xi_2(g_1)=-f_1$ and
$\Xi_i(g_i)=0, \,i=1,2$, then the gauge substitution $X_i+g_i\Xi_i\mapsto X_i$ 
annihilates $f_1$ and $f_2$. Since the vector fields $\Xi_i$'s commute, functions 
$g_i$'s with these properties exist, at least locally, due to the relations $\Xi_i(f_i)=0$ 
from the above lemma. 
In turn, Relations (\ref{erl}) simplify:
\begin{eqnarray}\label{erlm}
\Xi_1(\lambda)=\Xi_2(\lambda)=\Xi_1(\mu_2)=\Xi_2(\mu_1)=0,\nonumber\\
\Xi_1(\mu_1)+X_1(\lambda)= \Xi_2(\mu_2)+X_2(\lambda)=-\lambda^2.
\end{eqnarray}

In the subsequent analysis of the relations thus obtained, it is convenient to pass to 
vector fields $Z=X_1-X_2$ and $W=X_1+X_2$, assuming that
$X_1$ and $X_2$ are normalised as above. In these terms, Relations 
(\ref{yhi}) and (\ref{ihy}) read as follows:
\begin{eqnarray}\label{newihy}
\ldb Z,\Xi_i\rdb=-\lambda\Xi_i,  \quad\ldb W,\Xi_i\rdb=\lambda\Xi_i, \quad i=1,2,\nonumber \\
\ldb Z,W\rdb=2\lambda\,Z+2\,\mu_1\Xi_1-2\,\mu_2\Xi_2
\end{eqnarray}
and (see (\ref{erl}))
\begin{equation}\label{newerlm}
Z(\lambda)=\Xi_2(\mu_2)-\Xi_1(\mu_1), \quad W(\lambda)=-(2\lambda^2+\Xi_1(\mu_1)+\Xi_2(\mu_2)).
\end{equation}

Looking for local solutions of Relations (\ref{newihy}) and (\ref{newerlm}), we may 
suppose that the bi-dimensional foliation generated by $\Xi_1$ and $\Xi_2$ is a 
fibration $\pi:M\rightarrow N$. Moreover, it follows from these formulas that
vector fields $Z$ and $W$ are $\pi$-projectable and that 
$\Xi_i^2(\mu_i)=0, \;i=1,2$.
Therefore,
$$
\lambda=\pi^*(z), \quad \Xi_i(\mu_i)=\pi^*(\nu_i), \;i=1,2,  \quad \mbox{for some} \quad z, \,\nu_1, \,\nu_2\in\,C^{\infty}(N). 
$$

Let $\bar{Z}=\pi(Z), \; \bar{W}=\pi(W), \;u=\nu_2-\nu_1$ and $v=\nu_1+\nu_2$. Then
\begin{equation}\label{projrel}
\bar{Z}(z)=u, \quad \bar{W}(z)=-(2z^2+v), \quad [\bar{Z},\bar{W}]=2z\bar{Z}.
\end{equation}
\begin{proc}\label{mathc invariants}
Vector fields $\bar{Z}$ and $\bar{W}$ and functions $z,u,v\in C^{\infty}(N)$ are differential invariants of matchings of $\omega$-modular bi-vectors with respect to
the group of  diffeomorphisms of $M$ preserving the volume form $\omega$.
\end{proc}
\begin{proof}
Note that the normalised vector fields $X_i$'s are uniquely defined up to gauge transformations 
of the form $X_i\mapsto X_i+\phi_i\Xi_i$ with $\Xi_j(\phi_i)=0, \;i,j,=1,2$. It follows 
that their projections on $N$ and, consequently, that the projections of 
$Z$ and $W$ do not change. 
This proves the invariance of $\bar{Z}$ and $\bar{W}$.

In addition, the transformations under consideration 
do not change defining $\lambda$ Relations 
(\ref{yhi}). This shows that $\lambda$ and hence $z$ are invariant. In turn, the  
invariance of $\bar{Z},\bar{W}$ and of $\lambda$ and Relations (\ref{projrel}) yield the invariance of $u$ and $v$.
\end{proof}

Now we shall give a geometric interpretation of the data we obtained in order to evaluate the
complexity of the matching problem. Let $\mathcal{A}$ be the subalgebra of 
$C^{\infty}(N)$ composed of all smooth functions
of $u,v$ and $z$. This subalgebra is, obviously, an invariant of matchings, and so is
its real spectrum $\mathrm{Spec}_{\R}\mathcal{A}$. Locally, this spectrum is a manifold
with singularities (see \cite{N}) whose dimension is the number of functionally independent functions $u,v$ and $z$, which is 3 in general.
The inclusion map 
$\mathcal{A}\to C^{\infty}(N)$ induces a map of spectra
$$
\Upsilon\colon N=\mathrm{Spec}_{\R}C^{\infty}(N)\to \mathrm{Spec}_{\R}\mathcal{A} .
$$

Since $[\bar{Z},\bar{W}]=2z\bar{Z}$, the 
$C^{\infty}(N)$-submodule of $D(N)$ generated by $\bar{Z}$ and $\bar{W}$ represents an integrable 2-dimensional distribution
on $N$. Denote by $\Phi$ the associated 2-dimensional foliation. By construction, the equivalence class of the foliated manifold $(N,\Phi)$ is an invariant of matchings.
We shall call two maps,
$\Upsilon_1, \Upsilon_2\colon N\to \mathrm{Spec}_{\R}\mathcal{A}$,
$\Phi$-equivalent 
if $\Upsilon_1=\Upsilon_2\circ F$ where $F\colon N\to N$ is a diffeomorphism preserving $\Phi$. We see that the $\Phi$-equivalence class of $\Upsilon$ is an
invariant of matchings as well. This shows that the matching problem for two 
$\omega$-modular bi- vectors includes the classification of maps of
2-dimensional foliations to manifolds of dimensions not greater than 3. Other
hardly controllable complications come from Relations (\ref{projrel}).  All these observations show that the matching problem we considered does not allow an exact solution in
reasonable terms. 
In contrast, the similar problem for Lie algebras is reasonable
and its solution is given in Subsection \ref{mamoli}.

%%%%%%%%%%%%%%%%%%%%%%
\section{Modular structure of Lie algebras}\label{Lie-modlarity}
%%%%%%%%%%%%%%%%%%%%%

Now we shall apply the results of the preceding section to the case of linear Poisson structures and 
hence to Lie algebras.  In this case,  $M$ is replaced by the dual $V^*$ of an $n$-dimensional vector space $V$ over a ground field $\gk$. Since  
the results of the preceding section are algebraically formal, they remain valid in the differential calculus over 
the algebra $\gk[V^*]$ of polynomials on $V^*$.

%%%%%%%%%%%%%%%%%%%%%%%%%%%
\subsection{The modular disassembling of Lie algebras}
%%%%%%%%%%%%%%%%%%%%%%%%%%%

The \emph{cartesian volume form} 
$\omega=dx_1\wedge\dots\wedge dx_n$ associated with a standard cartesian chart 
$(x_1,\dots,x_n)$ on $V^*$ is well-defined up to a scalar factor. Obviously, the concept 
of $\omega$--modularity does not change when passing from $\omega$ to $\lambda\omega, \,0\neq\lambda\in\gk$.
So, the {\it cartesian modularity}, i.e.,  $\omega$--modularity with respect
to a cartesian volume form $\omega$, is well--defined on $V^*$ and will be 
referred to simply as \emph{modularity}. In this section we shall  only deal with polynomial 
tensor fields on $V^*$ and, accordingly, adjectives \emph{constant}, \emph{linear}, etc, 
will refer to the fields with constant, linear, etc, coefficients, respectively.

In what follow, $P$ denotes a linear Poisson structure on $V^*$ and it is identified with a 
Lie algebra structure on $V$ (see Subsection\,2.3). The differential form $\alpha=\alpha_P=
P\mathop{\rfloor}\omega$ is linear, while $d\alpha$, as well as the modular 
vector field $\Xi=\Xi_P$, are constant. It is easy to see that $\Xi$ does not depend on the
choice of a cartesian volume form. Since it is constant, the field $\Xi$ is identified with a 
vector $\theta=\theta_P\in V^*$ called the \emph{modular vector} of  $P$ or of the 
corresponding Lie algebra.

Since $\Xi$ is constant, a function $\;\nu\;$ such that $\Xi(\nu)=1$ can be chosen to be linear 
and, therefore, identified with a vector $v\in V$ such that $\theta(v)=1$. The Poisson 
bi-vector $P_\nu\wedge\Xi$ is  linear and hence corresponds to a Lie algebra structure 
on $V$. Obviously, it is well--defined by $P$. Therefore, the disassembling 
(\ref{BRA}) defines a disassembling of the Lie algebra associated with $P$ into 
unimodular and non-unimodular parts.  In fact,
\begin{equation}
\gG=\gG_{uni}+\gG_{non} \quad\mathrm{with} \quad P=P_{\gG}, \quad P_{\gG_{uni}}=
P-P_{\nu}\wedge\Xi, \quad
P_{\gG_{non}}=P_{\nu}\wedge\Xi. 
\end{equation}

A direct description of this disassembling in terms of the Lie algebra $\gG$ is as 
follows. First, recall that a linear operator $A:V\rightarrow V$ 
is naturally associated with a linear vector field $X$ on $V^*$. In fact, identifying 
vectors of $V$ with linear functions on $V^*$, this becomes a tautology: $A(u)=X(u)$. 
In particular, if $X=P_{\nu}$, then 
$$
A(u)\DEF\ldb u,\nu\rdb=P(du,d\nu),\quad u\in{\bf V^*}.
$$
Hence $A=-\ad_{\gG}\nu$ and the characteristic property (\ref{DIV}) of $\Xi$ is translated as
\begin{equation}\label{tra}
\theta(u)=-\tr(\ad_{\gG}u).
\end{equation}
This formula may be considered to be a direct definition of $\theta$. It also tells us that the
unimodular Lie algebras are those where the the adjoint representation
acts by operators of trace zero. In these terms, the Lie algebra structure
$\gG_{non}$ corresponding to $P_\nu\wedge\Xi$ reads as follows:
\begin{equation}\label{non}
[u,v]_{non}=\th(u)A(v)-\th(v)A(u), \quad u,v\in V, \quad A=\ad_{\gG}\nu,
\end{equation}
or, alternatively, $[u,v]_{non}=\th(u)[\nu,v]-\th(v)[\nu,u]$.
\begin{proc}\label{CAH} The operator $A=\ad_{\gG}\nu$
and  $\theta\in  V^*$ satisfy the relations: 
\be\label{usl} 
A^*(\th)=0, \quad \tr\,A=-1, \quad A(\nu)=0, \quad\th(\nu)=1
\ee 
Conversely, if $\theta\in V^*, \,\nu\in V$ and $A: V\rightarrow V$ satisfy the above relations, 
then Formula \emph{(\ref{non})} defines a Lie algebra, which coincides with its non--unimodular part.
\end{proc}
\begin{proof}
This proposition is nothing but Proposition \ref{UAB} in the particular case of Lie algebras. The relation $L_{\Xi}(\omega)=0$ from (\ref{uab}) is identically satisfied in this case, and so is the new relation, $A(\nu)=0$, which is equivalent to the obvious equality,
$P_{\nu}(\nu)=0$.   
\end{proof}

Since $\theta(\nu)=1$, the vector space $V$ splits into the direct sum of subspaces 
$ W_0=\ker\,\th$ and  $ W_1=\{\lambda\nu\, |\,\lambda\in \gk\}$ 
of dimensions $n-1$ and 1, respectively. They are invariant with respect to $A$ since $A^*(\theta)=0$ and $A(\nu)=0$ and also $\; W_1\subset\mbox{ker}\,A$.
This shows that Lie algebra (\ref{non})
 with $A$ and $\theta$ satisfying (\ref{usl}) is, up to an isomorphism, defined
 by the operator $A_0=A|_{W_0}:W_0\rightarrow W_0$. Indeed, 
 let $W_0$ and $W_1$ be vector spaces such that 
$\rm{dim}\,W_0=n-1, \rm{dim}\,W_1=1,$ and $0\neq e\in W_1.$ Then the following Lie algebra structure on 
$W=W_0\oplus W_1$ is defined by a linear 
operator $A_0:W_0\rightarrow W_0$ :
\begin{equation}\label{DKL}
[u,v]=0\quad{\rm{for}}\quad u,v\in W_0\quad{\rm{and}}\quad[u,e]=A_0(u) .
\end{equation}
This structure is isomorphic to that given by (\ref{non}) and (\ref{usl}), assuming 
that $\rm{tr}\,A_0\neq 0$. 

\begin{defi}\label{FRC}
Lie algebra \emph{(\ref{DKL})} with $\rm{tr}\,A_0\neq 0$ is called 
\emph{modular}.
\end{defi}

Note that the product in a modular Lie algebra is of the form (\ref{non}) with 
$\theta$ and $A$ satisfying relations (\ref{usl}) for some $\nu\in V$. This algebra 
will be denoted by $\gl_{A,\theta,\nu}$.

A summary of the above is:
\begin{proc}
Any finite-dimensional Lie algebra is the sum of a modular Lie algebra and a 
unimodular one compatible with it. $\quad \Box$
\end{proc} 

%%%%%%%%%%%%%%%%%%%%%%%%%%%%%%%%
\subsection{Compatibility of modular and unimodular Lie algebras}
%%%%%%%%%%%%%%%%%%%%%%%%%%%%%%%%

Now we shall discuss compatibility conditions of a modular Lie algebra and an unimodular one. 
A modular Lie algebra is of the form $\gG_{X\wedge\Xi}$ where $X$ and $\Xi$ 
are commuting linear and constant vector fields on $V^*$, respectively, which satisfy conditions (\ref{uab}). 
The product in this algebra is given by (\ref{non}), where $A:V\rightarrow V$ is the operator corresponding to $X$.
\begin{proc}
An unimodular Lie algebra algebra $\gG$ is compatible with the modular algebra 
$\gl_{A,\theta,\nu}$ if and only if
\begin{equation}\label{comp-mod-uni}
\theta([\gG,\gG])=0 \; \mathrm{and} \; \theta(u)([Av,w]+[v,Aw]-A[v,w])+\mathrm{cycle}=0, \;\forall  u,v,w\in V,
\end{equation}
where $[\cdot,\cdot]$ is the product in $\gG$.
\end{proc}
\begin{proof}
We have to express the two conditions of Proposition\;\ref{VUB} in terms of Lie algebras. To this end we need the following general formula
$$
L_Y(Q)(\omega,\rho)=Q(L_Y(\omega),\rho)+Q(\omega,L_Y(\rho))-Y(Q(\omega,\rho))
$$
with $Y\in D(M), \,Q\in D_2(M), \omega,\rho\in \Lambda^1(M)$. 
When $M=V^*, \,Y=\Xi, \,Q=P_{\gG}, \,\omega=du, \,\rho=dv$ this formula becomes
$$
L_{\Xi}(P_{\gG})(du,dv)=P_{\gG}(d(\theta(u)),dv)+P_{\gG}(du,d(\theta(v)))-\theta([u,v])=-\theta([u,v]),
$$
since $\theta(u)$ and $\theta(v)$ are constant. Condition $L_{\Xi}(P)=0$ of Proposition\,\ref{VUB} becomes $\theta([\gG,\gG])=0$. 

Similarly, for  $Y=X$ we obtain
$$
L_X(P_{\gG})(du,dv)=[Au,v]+[u,Av]-A([u,v]).
$$ 
Now it is easy to see that the second relation we have to prove is identical to 
the relation $\Xi\wedge L_X(P)=0$ of Proposition\,\ref{VUB}.
\end{proof}
\begin{cor}
Modular Lie algebras and unimodular Lie algebra $\gG$ such that $\gG=[\gG,\gG]$
are incompatible. In particular, a semi-simple Lie algebra cannot be the unimodular
part of a non-unimodular Lie algebra.
\end{cor}

%%%%%%%%%%%%%%%%%%%%%%%%%%%%%%%%%%%%%%%%%
\subsection{The matching problem for modular Lie algebras}\label{mamoli}
%%%%%%%%%%%%%%%%%%%%%%%%%%%%%%%%%%%%%

 Let $X_1\wedge\Xi_1, X_2\wedge\Xi_2$ be linear Poisson bi-vectors on 
$V^*$, with $X_i$ and $\Xi_i$ as in Section\,3.7. 
The modular field $\Xi$ of $X_i\wedge\Xi_i$ is a constant 
vector field on $V^*$ and, therefore, $[\Xi_1,\Xi_2]=0.$ For this reason, 
$X_1,X_2,\Xi_1,\Xi_2$ satisfy Relations (\ref{yhi}) and (\ref{ihy}) and, as 
a consequence, they satisfy the relations of Lemma\,\ref{UNI}. Since $X_1,X_2$ are linear 
vector fields on $V^*$, functions $f_1, f_2$ and $\lambda$ in (\ref{yhi}) 
are constant, while $\mu_1,\mu_2$ in (\ref{ihy}) are linear. As in 
Subsection\;\ref{comp-match}, after a suitable substitution 
$X_i+g_i\Xi_i\mapsto X_i, \,i=1,2,$ with linear $g_i$'s, functions $f_1$ and $f_2$ are eliminated. In this case relations in Lemma\,\ref{UNI} reduce to
\begin{equation}\label{ekl}
\Xi_1(\mu_2)= \Xi_2(\mu_1)=0, \quad\Xi_1(\mu_1)= 
\Xi_2(\mu_2)=-\lambda^2.
\end{equation}
Let $\bar V^*$ be the quotient of $V^*$  by the 2-dimensional subspace 
$\mathrm{span}(\Xi_1,\Xi_2)$. Then Relations (\ref{yhi}) and 
(\ref{ihy}) show that vector fields $X_1$ and 
$X_2$ project to some vector fields $\bar{X_1},\bar{X_2}$ on $\bar{V}^*$, 
respectively, and that
$[\bar{X_1},\bar{X_2}]=\lambda(\bar{X_1}-\bar{X_2}).$ So, $\bar{X_1}$ and $\bar{X_2}$ 
generate a 2-dimensional Lie algebra on $\bar{V}^*$. 

Let $V_1^*$ be a complement of $V_0^*=\rm{span}(\theta_1,\theta_2)$ in $V^*$ and let
$\pi_i:V^*\rightarrow V_i^*, \,i=0,1$, be the corresponding projections. 
A 
vector field $X\in D(V^*)$ which is projectable on $\bar{V}^*$ can be written as
$$
X=X_0+a_1\Xi_1+a_2\Xi_2, \quad a_1, a_2\in C^{\infty}(V^*)
$$ 
where $X_0$ is parallel to $V_1^*$. If $X$ is linear, then $X_0, a_1$ and $a_2$ 
are linear too.
In particular, vector fields $Z=X_1-X_2$ and $W=X_1+X_2$ can be written as 
\begin{equation}\label{fieldsDeco}
Z=Z_0+\alpha_1\Xi_1+\alpha_2\Xi_2, \quad W=W_0+\beta_1\Xi_1+\beta_2\Xi_2.
\end{equation}
In these terms, relations (\ref{newihy}) are equivalent to
\begin{eqnarray}\label{1-part}
-\Xi_1(\alpha_1)=\Xi_1(\beta_1)=-\Xi_2(\alpha_2)=\Xi_2(\beta_2)=-\lambda, \nonumber \\
\quad \Xi_1(\alpha_2)=\Xi_1(\beta_2)=\Xi_2(\alpha_1)= \Xi_2(\beta_1)=0.
\end{eqnarray} 
\begin{eqnarray}\label{0-part}
Z_0(\beta_1)-W_0(\alpha_1)-\lambda(3\alpha_1+\beta_1)=2\mu_1, \nonumber \\
Z_0(\beta_2)-W_0(\alpha_2)+\lambda(\beta_2-3\alpha_2)=-2\mu_2.
\end{eqnarray} 
Note that the linear functions $\alpha_i$'s and $\beta_i$'s depend on the choice of the 
complementary subspace $V_1^*$.

We also need the following lemma.
\begin{lem}\label{div X_0}
Let $\omega$ be a cartesian volume form on $V^*$ and 
$\omega=\pi_0^*(\omega_0)\wedge\pi_1^*(\omega_1)$
where $\omega_i$ is a cartesian volume form on $V_i^*$. Then
$$
L_X(\omega)=(\Div_{\omega_0}\bar{X}_0+\Xi_1(\alpha_1)+\Xi_2(\alpha_2))\omega
$$
where $\bar{X}_0$ is the restriction of $X_0$ to $V_1^*$. 
\end{lem}
\begin{proof}
Obviously, $X_0\rfloor\pi_0^*(\omega_0)=\Xi_0\rfloor\omega=\Xi_1\rfloor\omega=0$ so that
$$L_{X_0}(\pi_0^*(\omega_0))=0 \quad\mathrm{and} \quad L_{\alpha_i\Xi_i}(\omega)=
d\alpha_i\wedge(\Xi_i\rfloor\omega)=
\Xi_i(\alpha_i)\omega
$$
and
$$L_{X_0}(\omega)=\pi_0^*(\omega_0)\wedge L_{X_0}(\pi_1^*(\omega_1))=
\pi_0^*(\omega_0)\wedge\pi_1^*(L_{\bar{X}_0}\omega_1)=\Div_{\omega_0}\bar{X}_0\cdot\omega.
$$
\end{proof}

A linear function $\varphi$ on $V^*$ can be decomposed into the sum 
$\varphi=\varphi^0+\varphi^1$ where $\varphi^i$ 
is linear and vanishes on  $V_i^*, \,i=0,1$. Accordingly, Relations (\ref{1-part}) and 
(\ref{0-part}) split into two parts. The first explicitly describes functions $\alpha_i^1, \beta_i^1, \,i=1,2$, while the 
second one, in view of (\ref{ekl}), set no additional restrictions on these functions.      
This gives us the freedom to annihilate $\alpha_i^0$'s by means of the substitution
$$
(X_1+\alpha_1^0\Xi_1)\mapsto X_1, \quad(X_2-\alpha_2^0\Xi_2)\mapsto X_2,
$$
which is still at our disposal.  In this normalisation, Relations (\ref{0-part}) become
\begin{eqnarray}\label{reduced 0-part}
Z_0(\beta_1^0)-\lambda\beta_1^0=2\mu_1^0,  \quad\quad
Z_0(\beta_2^0)+\lambda\beta_2^0=-2\mu_2^0.
\end{eqnarray} 

Now the obtained data allow us to describe matchings of two modular Lie algebras. In fact, 
the matchings with non-zero invariant $\lambda$ 
are characterised by  ordered quadruples $(\mathcal{V},A,B,\lambda)$ where 
$\mathcal{V}$ is a vector space, $A,B:\mathcal{V}\rightarrow \mathcal{V}$
are linear operators such that $[A,B]=2\lambda A$ and $\tr\,B=2(1-\lambda)$.
The matchings with $\lambda=0$ are characterised by 
quintuples $(\mathcal{V},A,B,\nu_1,\nu_2)$ with commuting 
$A,B:\mathcal{V}\rightarrow \mathcal{V}$ such that $\tr\,A=0, \tr\,B=2$, and
$\nu_1,\nu_2\in\Ker\,A$. Quadruples $(\mathcal{V},A,B,\lambda)$ and 
$(\mathcal{V}',A',B',\lambda')$ are \emph{equivalent} if $\lambda=\lambda'$ and 
there exists an isomorphism $\Phi:\mathcal{V}\rightarrow \mathcal{V}'$
such that $ A'=\Phi A\Phi^{-1}, B'=\Phi B\Phi^{-1}$. Similarly, quintuples 
$(\mathcal{V},A,B,\nu_1,\nu_2)$ and $ (\mathcal{V}',A',B',\nu_1',\nu_2')$ are 
\emph{equivalent} if, additionally, $\nu_i'=\Phi(\nu_i)$. In other words, 
the group 
$\gG\gl({\mathcal{V}})$ acts naturally on quadruples and quintuples defined on 
${\mathcal{V}}$ and their equivalence classes are  labeled by the orbits of this action.

The quadruple (resp., quintuple) associated with a compatible pair 
$X_1\wedge \Xi_1, \;X_2\wedge \Xi_2$ is constructed on the subspace 
$$
\mathcal{V}=\Ann(\theta_1,\theta_2)\subset V, \quad
\Ann(\theta_1,\theta_2)=\{v\in V\mid \theta_1(v)=\theta_2(v)=0\}.
$$
We identify $\Ann(\theta_1,\theta_2)$ and, therefore, $\mathcal{V}$ with the dual
to the quotient space $\bar{V}=V^*/\mathrm{span}(\theta_1,\theta_2)$. In other words,
we identify vectors $v\in\mathcal{V}$ with linear functions $\varphi$ on $V^*$ such
that $\Xi_i(\varphi)=0, \,i=1,2$. It follows from (\ref{newihy}) that the subspace containing 
these functions is invariant with respect to the action of vector fields $Z$ and $W$, 
and we define the operators $A$ and $B$ to be the restrictions of $Z$ and $W$ to this subspace,
respectively. If $\lambda=0$, then we set $\nu_i=\beta_i^0$ for a suitable choice 
of $V_1^*$ (see below). 
\begin{thm}\label{LieMathing}
Assume that $\dim\mathcal V=n-2$ and $A,B, \nu_1, \nu_2$ are as above. Then the matchings 
of $n$-dimensional modular Lie algebras are classified by equivalence 
classes of quadruples $(\mathcal{V},A,B,\lambda)$ if $\lambda\neq 0$ and by quintuples 
$(\mathcal{V},A,B,\nu_1,\nu_2)$ if $\lambda=0$.
\end{thm}
\begin{proof}
Choose the complement $V_1^*$ so that $\mu_1^0=\mu_2^0=0$. It is not difficult to 
see that such $V_1^*$ exists and is unique. With this choice of $V^*_1$, functions 
$\mu_i=\mu_i^1$ are completely  determined by Relations (\ref{ekl}), 
and Relations (\ref{reduced 0-part}) simplify as follows:
\begin{eqnarray}\label{simplified 0-part}
A(\beta_1^0)-\lambda\beta_1^0=0,  \quad\quad
A(\beta_2^0)+\lambda\beta_2^0=0.
\end{eqnarray} 
The last of the commutation Relations (\ref{newihy}) implies that  $[A,B]=2\lambda A$. From this
it follows that $A$ is nilpotent if $\lambda\neq 0$ and, as a consequence, that 
$A\pm\lambda\id$ is non-degenerate. This shows that
the only solution of (\ref{simplified 0-part}) is $\beta_1^0=\beta_2^0=0$.  

Since $L_W(\omega)=L_{X_1}(\omega)+L_{X_2}(\omega)=2\omega$ (see (\ref{uab})) and 
$\Xi_1(\alpha_1)=\Xi_2(\alpha_2)=\lambda$ (see (\ref{1-part})), it follows from 
Lemma\,\ref{div X_0}
that $\Div_{\omega_0}W_0+2\lambda=2$ and hence $\tr\,B=\Div_{\omega_0}W_0=2(1-\lambda)$.
If $\lambda=0$, then (\ref{simplified 0-part}) just tells us that 
$\beta_1^0, \,\beta_2^0\in\Ker A$ and there are no other restrictions on these functions. 
This proves that the quadruples/quintuples associated with matchings of two modular Lie algebras possess the required properties.

Now we have to show that a given abstract quadruple/quintuple is associated with 
a pair of compatible modular Lie algebras. To this end, set
$V_1=\mathcal{V}, \;V=V_1\oplus V_0$ where $V_0$ is a 2-dimensional vector space,
so that $V^*=V^*_1\oplus V_0^*$. Then fix two independent vectors 
$\theta_1,\theta_2 $ in $V_0^*$, which will be understood as vectors in $V^*$, and
define $\Xi_i$ as the constant vector field on $V^*$ corresponding to 
$\theta_i, \,i=1,2$. Observe that, by construction, $V_1$ as a subset of $V$ is
$\Ann(\theta_1,\theta_2)$. Vector fields $Z_0$ and $W_0$ on $V^*$ are defined as
corresponding to operators $A\oplus 0_{V_0}, \,B\oplus 0_{V_0}\colon V\to V$.

Finally, define functions
$\alpha_i, \,\beta_i$ and $\mu_i$ by setting 
$$\alpha_i^0=\mu_i^0=0, \,\alpha_1^1=\lambda\varphi_1,  
\,\alpha_2^1=-\lambda\varphi_2, \,\beta_i^1=-\lambda\varphi_i, \,\mu_i^1=-\lambda^2\varphi_i
$$ 
where the $\varphi_i's$ are the linear function on $V^*$ vanishing on $V_1^*$ and such that 
$\varphi_i(\theta_j)=\delta_{ij}$, which are naturally identified with some vectors in 
$V_0$. Also, we set $\beta_i^0=0$ if $\lambda\neq 0$, and $\beta_i^0=\nu_i$  if $\lambda=0$.
In the last case, the vectors $\nu_i$'s are interpreted as linear functions on $V^*$.

Now, vector fields $Z$ and $W$ and, therefore, vector fields 
$X_1=\frac{1}{2}(Z+W), X_2=\frac{1}{2}(W-Z))$ are defined by
Formula (\ref{fieldsDeco}) with  the $\alpha_i$'s and $\beta_i$'s as above, and a direct computation shows that the linear bi-vectors $X_i\wedge\Xi_i$ thus constructed are 
modular and compatible. 
\end{proof}

\begin{rmk}
Since the representations of bi-dimensional Lie algebras are well-known, Theorem 
\emph{\ref{LieMathing}} gives an exhaustive description of the matchings of modular Lie 
algebras. In particular, given such a representation, the value of the invariant 
$\lambda$ can be found from the trace formula, $\tr\,B=2(1-\lambda)$. 
\end{rmk}
\begin{rmk}
It is easy to  see that matchings with proportional $\Xi_1$ and $\Xi_2$ are completely characterised by quadruples $(\mathcal{V}, A_1,A_2, \nu)$ such that 
$\dim\,\mathcal{V}=n-1, \tr\,A_1=\tr\,A_2=1$ and $\nu\neq 0\in\gk$. Namely, 
$\mathcal{V}=\Ann (\theta_1), \,\Xi_2=\nu\Xi_1$ and $A_i\colon \mathcal{V}\to\mathcal{V}$ 
is the linear operator corresponding to the vector field $X_i$
projected on the quotient space $V^*/ \mathrm{span}(\theta_1)$, which is identified with 
$\mathcal{V}^*$.  
\end{rmk}

%%%%%%%%%%%%%%%%%%%%%%
\section{The disassembling problem}\label{dis-probl}
%%%%%%%%%%%%%%%%%%%%%%%

This section contains the main results of our paper. Here, we discuss how a Lie algebra 
can be  
disassembled into pieces which are themselves Lie algebras. We
introduce some basic disassembling techniques which yield  various results on the assembling and disassembling of Lie algebras.  As our main result, we prove that  any 
finite-dimensional Lie algebra over an algebraically closed field of characteristic  
zero or over  $\R$ can be  assembled in a few steps from elementary constituents, called \emph{lieons}, which are of two 
types, called \emph{dyons} and \emph{triadons} (Theorems \ref{C-dis} and \ref{R-dis}).

In this section, ``Lie algebra" refers to a finite-dimensional Lie algebra 
over a ground field $\gk$ of characteristic zero.  

%%%%%%%%%%%%%%%%%%%%%%
\subsection{Statement of the problem}\label{statemant}
%%%%%%%%%%%%%%%%%%%%%%%%

%%%%%%%%%%%%%%%%%%%%%%%%%
\par\medskip
%\noindent
\textit{Simple disassemblings and lieons.} 
%%%%%%%%%%%%%%%%%%%%%%%%%%

We introduce the necessary terminology  in order to 
formulate the {\it disassembling  problem}. 

\begin{defi}\label{base-dis}
A \emph {simple disassembling} of a Lie algebra structure $\gG$ on a vector 
space $V$ is a representation of $\gG$ as a sum 
\begin{equation}\label{dec}
\gG = \gG_1 + \dots +\gG_k
\end{equation}
of mutually compatible Lie algebra structures $\gG_i$'s on $V$.
The Lie algebras 
$\gG_i$'s in \emph{(\ref{dec})} are called \emph{primary constituents} of 
$\gG$. 
\end{defi}

In this case,  slightly abusing the language,  we speak of a (simple) 
disassembling of the Lie algebra $\gG$ into Lie algebras $\gG_1,...,\gG_k$ or, 
alternatively, that  $\gG$ is assembled from $\gG_1,...,\gG_k$. Accordingly,  
we write 
\begin{equation}\label{2-dec}
 \gG_1 + \dots +\gG_k = \gH_1  + \dots + \gH_l,
\end{equation}
in order to express one of the  two following facts:
\begin{itemize}
\item a Lie algebra structure on a vector space $V$ admits two 
(different) disassemblings into Lie algebras structures $\gG_i$'s and 
$\gH_i$'s, respectively; 
\item the Lie algebras assembled from Lie algebras $\gG_i$'s and $\gH_i$'s, 
respectively, are isomorphic. 
\end{itemize} 

Having disassembled a Lie algebra $\gG$ into constituents $\gG_1,\dots , \gG_k$ 
it is natural to look for 
a further disassembling of the $\gG_i$'s and so on. This 
way one gets {\it secondary, ternary}, etc, constituents. The inverse to such 
a multi-step disassembling procedure will be called an \emph{assembling} 
procedure. 
A natural question, which we call the \emph{disassembling problem}is: 
\begin{quote}
{\it What are the \emph{``finest" (``simplest")} constituents of which any Lie 
algebra over a given ground field can be assembled?} 
\end{quote}

It is not difficult to imagine that the following Lie algebras 
must be in the list of the ``finest" algebras: 
\begin{itemize}
  \item the 1-dimensional Lie algebra, $\boldsymbol{\gamma}$,
  \item the unique non-abelian  2-dimensional Lie algebra, $\between$,
  \item the 3-dimensional Heisenberg Lie algebra  (over $\gk$), $\pitchfork$.
\end{itemize}
In terms of generators, the Lie algebras $\between$ and $\pitchfork$ are described 
as follows: 
\begin{align*}
 \between &= \{e_1, e_2 \mid [e_1,e_2]=e_2\} \\
  \pitchfork &= \{\e_1, \e_2, \e_3 \mid [\e_1,\e_2]=\e_3,\ [\e_1,\e_3]=[\e_2,\e_3]=0\}.
  \end{align*}
\noindent They are ``simplest" in any 
reasonable sense of the word. In particular, $\between$ is the ``simplest" 
non-unimodular algebra, while $\pitchfork$ is the ``simplest" non-abelian unimodular one. 

Denote by ${\between}_n, n\geqslant 2$, (resp., ${\pitchfork}_n, n\geqslant 3$) 
the direct sum of $\between$ (resp., $\pitchfork$) and the 
$(n-2)$-dimensional (resp., $(n-3)$-dimensional) abelian Lie algebra. We shall 
also use the symbol $\boldsymbol{\gamma}_n$ for the $n$-dimensional abelian Lie algebra. 
\begin{defi}\label{lieons}
A Lie algebra structure isomorphic to $\between_n$ $($resp., to $\pitchfork_n)$ is called 
an \emph{$n$-dyon} $($resp., an \emph{$n$-triadon}$)$. The collective name for both 
is \emph{$n$-lieon}.
\end{defi}
When the dimension $n$ is clear from the context we shall omit the prefix ``$n-$". 

Solving the disassembling problem for 3-dimensional Lie algebras is rather simple
(see \cite{MVV}):
\begin{ex}
Any unimodular 3-dimensional Lie algebra can be simply assembled from $l$ 
triadons, $l\leq 3$. Any non-unimodular 3-dimensional Lie algebra 
can be simply assembled from $l$ triadons, $l\leq 2$, and one 
dyon. In this sense, one can say 
that all 3-dimensional Lie algebras are simply assembled from $\between$'s, 
$\pitchfork$'s and $\boldsymbol{\gamma}$, since 
$\between_3=\between\oplus\boldsymbol{\gamma}$. 
\end{ex}

%%%%%%%%%%%%%%%%%%%%%%%%%%
\noindent See also \cite{Mor} for an explicit description of the 
algebraic variety $\mathrm{Lie}(3)$ of all Lie algebra structures on a 3-dimensional 
vector space.

It is sometimes more efficient to use ``chemical" formulas such as 
\begin{equation}\label{2h=2b}
 2\pitchfork = 2\between + 2\boldsymbol{\gamma}.
\end{equation}
This formula, which will be proven below, is synonymous to $\pitchfork + 
\pitchfork = \between_3 + \between_3.$ We stress that formulas such as 
(\ref{2h=2b}) only tell us that a given Lie algebra can be assembled either from the 
Lie algebras indicated in the left-hand side or from those in the right-hand side of the equality.\\
%\newline
%%%%%%%%%%%%%%%%%%%%%%%%%%%%
\par
%\bigskip
\noindent\textit{Assemblage schemes.}  
%%%%%%%%%%%%%%%%%%%%%%%%

\noindent Now we pass to a necessary bureaucracy. An {\it assembling scheme} 
(or an {\it a-scheme} for short) $\gS$ is a finite graph, whose set of vertices 
$vert\,\gS$ is a disjoint union of nonempty subsets $vert_s\gS, \,s=0,\dots ,m$, 
called {\it levels}, such that 
\begin{enumerate}
\item $vert_0\gS$ consists of only one vertex $o_{\gS}$, called
the {\it origin} of $\gS$;
\item edges of $\gS$ only connect vertices of consecutive levels. If 
$v_0\in vert_s\gS$ and \\ $v_1\in vert_{s+1}\gS$ are  ends of an edge, 
then they are called \\ the {\it origin} and the {\it end} of this edge, 
respectively; 
\item any vertex $v\in vert_s\gS , s>0$,  is the end of only one
edge; 
\item none of the vertices $v\in vert_s\gS , s<m$, is the origin of exactly one edge. A vertex which is not the origin of an edge is called an {\it end} of $\gS$.

\end{enumerate} 
%%%%%%%%%%%%%%%%%%%%%%%%%
\begin{wrapfigure}[13]{o}{4cm}
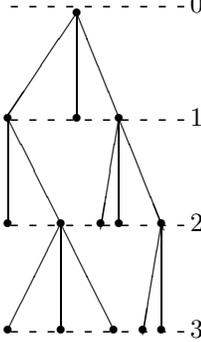
 
\figone
\caption{An a-sche\-me of length~3.}
\end{wrapfigure}

The number $m$, denoted also by $|\gS|$, is called the {\it length} of $\gS$. 
Obviously, $\gS$ is a connected graph and there exists at most one edge 
connecting two given vertices of it. All vertices in $vert_m\,\gS$ are ends. 
The a-schemes $\gS$ and $\gS'$ are \emph{equivalent} if they are equivalent
as graphs. 

Fig. 1 illustrates these formal definitions. The dashed lines in this figure indicate
the levels of the a-scheme considered.

%%%%%%%%%%%%%%%%%%%%%%%%%%%
\par\medskip
\noindent\textit{Multi-step disassemblings.} 
%%%%%%%%%%%%%%%%%%%%%%%%%%

%%%%%%%%%%%%%%%%%%%%%%%%%%%%
\begin{defi}
 Let $\gG$ be a Lie algebra structure on a vector space $V$ and
 $\gS$ be an a-scheme. A system $\{\gG_v\}, v\in vert\,\gS$,
 of Lie algebra structures on $V$ is called an $m$-step $(m=|\gS|)$
 \emph{disassembling} of $\gG$ if
 \begin{enumerate}
\item $\gG =\gG_{o_{\gS}}$;
\item If $v_1,\dots,v_p\in vert\,\gS$ are ends of edges having the common
origin v, then structures $\gG_{v_1},\dots ,\gG_{v_p}$ are mutually compatible 
and $\gG_v = \gG_{v_1}+\dots +\gG_{v_p}$. 
 \end{enumerate}
\end{defi}
%%%%%%%%%%%%%%%%%%%%%%%%%%%%%%%%%%%%%%

%%%%%%%%%%%%%%%%%%%%%%%%%%%%%%%%%%%%%%

A disassembling according an a-scheme $\gS$ will be called $\gS$-disassembling.
 The structure $\gG_v, v\in vert_s\gS,$ is an {\it $(s$-level$)$ term} of the 
$\gS$-disassembling and an {\it end term} if $v$ is an end point of $\gS$. 
If $(\gG_1,\dots,\gG_r)$ is the set of all end terms of a disassembling of $\gG$, 
then we say that $\gG$ \emph{is assembled from these Lie algebras}. 

We stress that if $v$ and $w$ are not the ends of two edges with a common 
origin, then  $\gG_v$ and $\gG_w$ are not, in general, compatible.

\begin{defi}\label{finest}
\begin{enumerate}
\item A disassembling of $\gG$ is called \emph{complete} if all its end terms are 
isomorphic either to $\between_n$ or to $\pitchfork_n$. 
\item Disassemblings of isomorphic Lie algebras $\gG$ and $\gH$ are called 
\emph{equivalent} if there exist an equivalence 
$\sigma:vert\,\gS\rightarrow vert\,\mathfrak{H}$ 
of the corresponding a-schemes and an isomorphism $\varphi:\gG\rightarrow\gH$ 
such that $\varphi$ is also an isomorphism of 
$\gG_v$ and $\gH_{\sigma(v)}$, for all $v$ in $vert\,\gS$.
\end{enumerate}
\end{defi}
 
Obviously, nonequivalent disassemblings can have equivalent a-schemes. 
The above-mentioned formula $\pitchfork+\pitchfork=\between_3+\between_3$
is a simple example. 

%%%%%%%%%%%%%%%%%%%%%%%%%
\par\medskip
\noindent\textit{The disassembling problem.}
%%%%%%%%%%%%%%%%%%%%%%%%

\noindent  The \emph{disassembling problem} is the question:
\begin{quote}
{\it Can a given Lie algebra be completely disassembled?}
\end{quote}
Below we shall develop some disassembling techniques and prove that any Lie 
algebra over an algebraically closed field or over $\R$ 
can be completely disassembled. This result confirms that lieons are 
elementary constituents of which all Lie algebras are made. 
Borrowing from terminology of theoretical physics, one may say that $\gamma$ creates 
the ``vacuum" which makes interactions possible between the constituents,  
$\between$ and $\pitchfork$, of the ``Lie matter".

We remark that the number of elementary 
constituents for Lie algebras cannot be reduced to one. Indeed, according 
to Proposition\,\ref{PFA}, by 
assembling unimodular Lie algebras, one only obtains unimodular algebras. 
So, only unimodular Lie algebras can be assembled from triadons. 
On the other hand, it is not difficult to show (see \cite{MVV})
that $\pitchfork$ cannot be assembled from 3-dyons only (compare 
with (\ref{2h=2b})\,!). Hence, neither $\pitchfork$ nor $\between$ can be 
excluded from the list of the ``finest" Lie algebras.

We shall now explain some basic techniques and constructions that will be used in our 
analysis of the disassembling problem. 

%%%%%%%%%%%%%%%%%%%%%%%%%%%%%%%% 
\par\bigskip
\subsection{Reduction to solvable and semi-simple algebras} \label{ReductionSolvable}
%%%%%%%%%%%%%%%%%%%%%%%%%%%%%%%%%

\noindent First, we shall show that the problem of disassembling splits naturally into a
``solvable" and a ``semi-simple" part. Given a Lie algebra $\gA$
and a vector space $V$ considered as an abelian Lie algebra, the semi-direct product defined  
by a representation $\rho:\gA\rightarrow \rm{End}\,V$ will be denoted by 
$\gA\oplus_{\rho}V$.
\begin{proc}\label{s-tr}
Let the Lie algebra $\gG$  be the semi-direct product of a subalgebra $\gG_0$ and of an 
ideal $\gH$. Identifying $|\gG|$ with $|\gG_0|\oplus |\gH|$ we obtain the 
simple disassembling 
\begin{equation}\label{csp}
\gG =(\gG_0 \oplus_{\rho}|\gH |)\,+\,(\bg_m\oplus \gH),
\end{equation}
where $\rho$ is the
canonical representation of $\gG_0$ in $|\gH|$\,  and \, 
$\bg_m, \,m=\dim\,\gG_0$, is the abelian structure on $|\gG_0|$. 
\end{proc}
\begin{proof}
The proof is clear from the definitions. 
\end{proof} 

Applying Proposition \ref{s-tr} to the Levi-Malcev decomposition 
$\gG=\gH\oplus_{\rho}\gR$ of a Lie algebra $\gG$, where $\gR$ 
is the radical of $\gG$  and $\gH\subset\gG$ is a 
semi-simple subalgebra complementary to $\gR$, we obtain: 
\begin{cor}\label{dis-reduct}
The disassembling problem for Lie algebras over a field $\gk$ of 
characteristic zero reduces to the case of solvable algebras $($over $\gk)$ and to the case of
abelian extensions of semi-simple algebras $($over $\gk)$, i.e., to algebras of 
the form $\gH\oplus_{\rho}W$ where $\rho:\gH\rightarrow \rm{End}\,W$  
is a finite-dimensional representation of the semi-simple algebra $\gH$.
\end{cor}

%%%%%%%%%%%%%%%%%%%%%%
\par\smallskip
\noindent\textit{Solvable Lie algebras.} 
%%%%%%%%%%%%%%%%%%%%%

\noindent The first of these two problems admits a simple solution. 
\begin{proc}\label{dis-solv}
Any solvable Lie algebra over a field $\gk$ can be completely disassembled. 
\end{proc}
\begin{proof}
Let $\gG$  be a solvable algebra. Then any subspace of $|\gG|$ containing the 
derived algebra $[\gG , \gG]$ is, obviously, an ideal of $\gG$. Let $\gs$ be such 
an ideal of codimension one and $\bg$ one-dimensional 
subalgebra complementary to $\gs$. Thus $\gG$ is a semi-direct product,
$\gG=\bg\oplus_{\rho}\gs$.Applying Proposition 
\ref{s-tr}, we see that $\gG$ can be disassembled into 
two structures, one of which is $\bg\oplus\gs$ where $\gs$ is a solvable algebra, 
while  the other is of the form $\bg\oplus_{\rho}V$ where $\rho$ is a 
representation of $\bg$ in the vector space $V=|\gs|$. Now easy induction 
arguments reduce the problem to the disassembling of algebras of the second type. 
This can be done as follows. 

Fix a base element $\nu\in \bg$ and put $A=\rho(\nu):V\rightarrow V, 
\,\,0\neq\nu\in \bg$. Then the Lie product in the algebra $\bg\oplus_{\rho}V$ is  
described by the relations
$$
[\nu,v]=Av, \,\,[v_1,v_2]=0, \,\,\,\nu\in\bg, \,v_1,v_2\in V,
$$
which show that it is completely determined by the operator $A$. Denote the Lie
algebra thus defined by $\Gamma_A$ and remark that the Lie algebras $\Gamma_A$  
and $\Gamma_{\lambda A}, \,\,0\neq\lambda\in \gk$,  \,are isomorphic, since the 
operator $A$ in this construction is defined up to a scalar factor.  
Hence it remains to show that the algebras $\Gamma_A$'s 
can be completely disassembled.

To this end, first, note that \\
1) $\between=\Gamma_A$  if $A$ is the identity operator on a 
1-dimensional vector space $V$;  \\
2) $\pitchfork =\Gamma_A$  if $A$ is a 
nontrivial nilpotent operator on a 2-dimensional vector space $V$. 

Secondly, if $\|a_{ij}\|$  is the matrix of $A$ in a basis $\{ e_i\}$ of $V$, 
then $A=\sum a_{ij}E_{ij}$, \,where the operator $E_{ij}:V\rightarrow V$   is 
defined by $E_{ij}(e_i)=e_j$ and $E_{ij}(e_k)=0, \,\,k\neq i$. Now observe 
that the structure $\Gamma_{E_{ij}}$ is isomorphic to $\pitchfork_n$,  
if $i\neq j$,  and to $\between_n$, if $i=j$. Finally,  since 
$\Gamma_{a_{ij}E_{ij}}$ is isomorphic to $\Gamma_{E_{ij}}$, if $a_{ij}\neq 0$, 
we see that the desired disassembling is
\begin{equation}\label{g-dec}
\Gamma_A=\sum_{i,j}\Gamma_{a_{ij}E_{ij}}=\sum_{i,j,\,a_{ij}\neq 0}L_{ij}.
\end{equation}
where the notation $L_{ij}=\Gamma_{a_{ij}E_{ij}}$ is introduced to emphasise that 
$\Gamma_{a_{ij}E_{ij}}$ is a lieon.
\end{proof}

The disassembling (\ref{g-dec}) depends on the choice of a basis in $V$. This fact 
can be used to show that the process of complete disassembling of a given Lie 
algebra is not unique. For instance, let $\dim\,V=2$ and $A:V\rightarrow V$ be an operator
with eigenvalues $\pm 1$. In the basis of eigenvectors $\{e_1,e_2\}$, the disassembling (\ref{g-dec}) is $\Gamma_A=\Gamma_{E_{11}}-\Gamma_{E_{22}}$, 
or, symbolically, $\Gamma_A=2\between_3$. On the other hand, in the basis 
$\{e_1 + e_2, \,\,e_1-e_2\}$ we have 
$\Gamma_A=\Gamma_{E_{12}}+\Gamma_{E_{21}}\Leftrightarrow \Gamma_A=2\pitchfork$. 
This proves Formula (\ref{2h=2b}). 

%%%%%%%%%%%%%%%%%%%%%%%%%%%%%%%%%%%%%%%%%%%%%%%%%%%%%%%%%%%%%%%%%%%%%%%%%%%%%%
\par\medskip
\noindent\textit{From semi-simple to simple Lie algebras.} 

\noindent The disassembling problem  for algebras $\gG\oplus_{\rho}V$, where $\gG$ is
semi-simple, is easily reduced to the case where $\gG$ is simple.
Indeed, observe that, for  a direct sum of Lie algebras, 
$\gG=\gG_1\oplus\dots\oplus\gG_k$,  there is a natural assemblage, $\gG=\bar{\gG}_1+\dots +\bar{\gG}_k$, in which $\bar{\gG}_i$ is the direct sum of abelian structures
on subspaces $|\gG_j|, j\neq i$, and of $\gG_i$ on $|\gG_i|$. If  
$\rho$ is a representation of $\gG$ in $V$, then the representation $\rho_i$
of $\bar{\gG_i}$ in $V$ is defined to be trivial on $|\gG_j|, \;j\neq i,$ and to coincide 
with $\rho$ on $\gG_i$.  
\begin{proc}\label{dec-sum}
Let $\gG$ be the direct sum of Lie algebras, 
$\gG_1\oplus\dots\oplus\gG_k$. Then
\begin{equation}\label{dec-repr}
\gG\oplus_{\rho}V=\bar{\gG}_1\oplus_{\rho_1}V+\dots+\bar{\gG}_k\oplus_{\rho_k}V,
\end{equation}
 is a simple disassembling of $\gG\oplus_{\rho}V$. 
\end{proc} 
\begin{proof}
This proposition follows from the definitions. 
\end{proof}
If $W_i=\oplus_{j\neq i}|\gG_j|$ and $\tilde{\rho}_i=\rho_i^0\oplus\rho\vert_{\gG_i}$
where $\rho_i^0$ is the trivial representation of $\gG_i$ in $W_i$, then
$$
\bar{\gG}_i\oplus_{\rho_i}V=\gG_i\oplus_{\tilde{\rho}_i}(W_i\oplus V).
$$
Now this observation and Proposition\,\ref{dec-sum} give the desired reduction to 
simple Lie algebras:
\begin{cor}\label{dec-ss}
If $\gG=\gG_1\oplus\dots\oplus\gG_k$ is a decomposition of the semi-simple Lie algebra
$\gG$ into a direct sum of simple algebras, then 
\begin{equation}
\gG\oplus_{\rho}V=\gG_1\oplus_{\tilde{\rho}_1}(W_1\oplus V)+\dots +
\gG_k\oplus_{\tilde{\rho}_k}(W_k\oplus V)
\end{equation} 
is a simple disassembling of $\gG\oplus_{\rho}V$. 
\end{cor}

Thus, we have reduced the general disassembling problem to the disassembling of Lie algebras
of the form $\gG\oplus_\rho V$ with simple $\gG$. This case presents the main difficulties
and its solution is essentially based on the {\it stripping procedure} which we shall now describe. 
%%%%%%%%%%%%%%%%%%%%
\subsection{The stripping procedure}\label{striptease}
%%%%%%%%%%%%%%%%%

First, we shall introduce some special algebras which are used in this procedure.

%%%%%%%%%%%%%%%%%%%%
\par\medskip
\noindent\textit{Dressing algebras.} 
%%%%%%%%%%%%%%%%%%%%

\noindent A {\it dressing algebra} is defined on the direct sum $W_0\oplus W$ 
of two vector spaces $W_0$ and $W$ supplied with a bilinear skew-symmetric 
${W_0}$-valued form $\beta :W\times W\rightarrow W_0$. The product in this algebra
is defined by the formula   
\begin{equation}
[(w_0,w),(w_0',w')]=(\beta (w,w'),0),\quad w_0, w_0^{\prime}\in W_0,\,\,w, w^{\prime}\in W.
\end{equation}
Denote the Lie algebra thus defined by $\gA_{\beta}$. If $\dim\,W=2, \,\,\dim\,W_0=1$ 
and $\beta\neq 0$, then $\gA_{\beta}$ is isomorphic to 
$\pitchfork$.

A Lie algebra $\gA$ is isomorphic to a dressing algebra if and only if the derived subalgebra 
$[\gA,\gA]$ is contained in its center. Indeed, if this condition is satisfied, one can set $W_0$ to be the center and $W$ to be any subspace complementary to the center. 

\begin{proc}\label{dress}
Let $\beta$  and $\beta '$ be $W_0$-valued skew-symmetric bilinear forms on 
$W$. Then, the  Lie algebras $\gA_{\beta}$  and $\gA_{\beta '}$ are compatible. 
Moreover, any dressing algebra can be simply disassembled into a number of 
triadons. 
\end{proc}
\begin{proof} The compatibility of $\gA_{\beta}$  and $\gA_{\beta '}$ immediately follows from
\begin{equation}
\gA_{\beta + \beta '}=\gA_{\beta}+\gA_{\beta '}.
\end{equation}

If $e_1,\dots,e_m$ is a basis in $W_0$ and 
$\varepsilon_1\dots,\varepsilon_{n-m}$ is a basis in $W$, then
$$
\beta =\sum_{i,j,k}\beta_{ij}^k\varepsilon_k
$$
for some $\gk$-valued skew-symmetric bilinear forms $\beta_{ij}^k$ on $W$ 
such that $\beta_{ij}^k(e_p,e_q)=0$ if $(p,q)$ differs from both $(i,j)$ and $(j,i)$. If 
$\beta_{ij}^k\neq 0$, then the algebra $\gA_{\beta_{ij|k}}$ with 
$\beta_{ij|k}=\beta_{ij}^k\varepsilon_k$ is isomorphic to $\pitchfork_n$. 
Hence, 
\begin{equation}\label{dec-dress}
\gA_{\beta}={\sum_{i,j,k}}'\gA_{\beta_{ij|k}}, \quad
\end{equation}
where the summation $\sum_{i,j,k}'$ runs over all triples $i,j,k$ for 
which $\beta_{ij}^k\neq 0$. 
\end {proof}

%%%%%%%%%%%%%%%%%%%%%
\medskip
\noindent\textit{Definition and examples of $d$-pairs}.
%%%%%%%%%%%%%%%%%%%%%%%

\noindent Let $\gs$ be a subalgebra of a Lie algebra $\gG$ and let $W$  be a 
complement of $|\gs|$ in $|\gG|$. If $[W,W]\subset\gs$ and $[\gs,W]\subset W$, 
the pair $(\gs,W)$  \,is called a $d$-{\it pair} in $\gG$. A $d$-pair $(\gs,W)$ 
is {\it trivial} if $W$ \,is an abelian subalgebra. 

The dressing algebra $\gA_{\beta}$  defined on $|\gG|$ with $W_0=|\gs|$ and 
$\beta(w_1,w_2)=[w_1,w_2], \\ \,w_1,w_2\in W$ will be called {\it associated} with 
$(\gs,W)$. 

\begin{ex}\label{pair}
Let $V$ be a vector space and $V_1, V_2$ be complementary subspaces. Consider the  subalgebra $\gs=\gs(V_1,V_2)$ of the Lie algebra 
$\gG\gl(V)$ composed of operators, leaving $V_1$ and $V_2$ invariant. The 
linear subspace $W=W(V_1,V_2)$ formed by operators sending $V_1$ to $V_2$ and 
vice-versa is a complement to $\gs(V_1,V_2)$ in $\gG\gl(V)$. Then ($\gs,W$) is 
a $d$-pair in $\gG\gl(V)$. 
\end{ex}
\begin{rmk}
A $d$-pair $(\gs,W)$ in $\gG$ supplies $\gG$ with the structure of  
$\dF_2$-graded algebra structure $(\dF_2=\Z/2\Z)$ and vice-versa. Namely, if
$\gG$=$\gG_0\oplus\gG_1$, then $\gs=\gG_0, \,W=\gG_1$. 
\end{rmk}
With a $d$-pair $(\gs,W)$ is naturally 
associated the involution $I:|\gG| \rightarrow |\gG|, \,I^2=id_{|\gG |}$ for which $|\gs |$ and $W$ are the  
eigenspaces corresponding to the eigenvalues $1$ and $-1$, respectively. Obviously, $I$ is an automorphism of 
$\gG$. Conversely, the $\pm 1$-eigenspaces of an involutive automorphism $I$ of 
$\gG$ form a $d$-pair in $\gG$. So, there is a one-to-one correspondence between
$d$-pairs and involutions on $\gG$. Which of these points of
view is more convenient depends on the context.
\begin{ex}\label{d-gl-pair}
The transposition of matrices, $T:M\mapsto M^t, \,M\in \gG\gl(n,\gk)$, is an 
anti-automorphism of $\gG\gl(n,\gk)$, i.e., $[M,N]^t=[N^t,M^t]$. So,  
$\mathfrak{t}=-T$ is an involution  on $\gG\gl(n,\gk)$. The $1$-eigenspace of 
$\mathfrak{t}$ is the space of skew-symmetric matrices and is identified with the 
special orthogonal subalgebra $\gs\go(n,\gk)\subset \gG\gl(n,\gk)$, while the 
$(-1)$-eigenspace $S(n,\gk)$ consists of the symmetric matrices. So, 
$(\gs\go(n,\gk),S(n,\gk))$ is a $d$-pair in $\gG\gl(n,\gk)$. The subalgebra 
$\gs\gl(n,\gk)\subset\gG\gl(n,\gk)$ of traceless matrices is 
$\mathfrak{t}$-invariant. Hence $\mathfrak{t}_0=\mathfrak{t}|_{\gs\gl(n,\gk)}$ 
is an involution in $\gs\gl(n,\gk)$ and $(\gs\go(n,\gk),S_0(n,\gk))$, where
$S_0(n,\gk)$ is the space of traceless symmetric matrices, is the 
$d$-pair corresponding to $\mathfrak{t}_0$.  
\end{ex} 

%%%%%%%%%%%%%%%%%%%%%
\bigskip
\noindent\textit{The stripping procedure}. 
%%%%%%%%%%%%%%%%%%%%

 \noindent The following evident fact is one of the most efficient 
disassembling instruments.
 \begin{lem}[The Stripping Lemma]\label{str}
Let $(\gs,W)$  be a $d$-pair in a Lie algebra $\gG$. If $\gA_{\beta}$ is 
the associated dressing algebra and $\rho$ is the restriction of the adjoint 
representation of $\gG$ to $\gs$, then 
\begin{equation}\label{strip}
\gG=(\gs\oplus_{\rho}W)+\gA_{\beta},
\end{equation}
is a simple disassembling of $\gG$. 
 \end{lem}
 \begin{proof} The proof follows directly from the definitions.
 \end{proof}
 \begin{rmk}The dressing algebra $\gA_{\beta}$ may be viewed as a ``mantle" that 
 covers the ``shoulder" $\gs$ of $\gs\oplus_{\rho}W$. This motivates the  terminology. 
 According to Proposition\,\ref{dress},  $\gA_{\beta}$ can be completely 
 disassembled. This reduces the disassembling problem  for $\gG$ to a simpler algebra, 
 namely, $\gs\oplus_{\rho}W$. 
 \end{rmk}

The {\it stripping procedure} is a series of applications of the 
Stripping Lemma which gradually simplify the algebras to which it can be applied.
In order to make the term ``simplification" precise, we 
define the {\it complexity} $l(\gG)$ of a Lie algebra  $\gG$ to be  the dimension 
of its ``semi-simple part", i.e., the dimension of its Levi subalgebra. Lie 
algebras of complexity zero are 
solvable and, according to Proposition \ref{dis-solv}, can be completely 
disassembled. Thus, we can conclude:
\begin{quote}
\emph{All Lie algebras over a given ground field $\gk$ can be 
completely disassembled if any algebra of the form $\gG\oplus_{\rho}V$, with $\gG$ 
simple, admits a $d$-pair $(\gs,W)$ such that $l(\gs)<l(\gG)$.} 
\end{quote}
In fact, the dressing algebra $\gA_{\beta}$  in (\ref{strip}) can be 
completely disassembled (Proposition \ref{dress}).  On the other hand, 
$l(\gs\oplus_{\rho'}W)=l(\gs)<l(\gG)$. So, by applying Poposition \ref{s-tr} 
to the Levi-Malcev decomposition of the algebra $\gs\oplus_{\rho'}W$, we reduce 
the problem to an algebra of the form $\bar{\gG}\oplus_{\bar{\rho}}\bar{V}$ where 
$\bar{\gG}$ is the semi-simple  part of $\gs$, and 
$l(\bar{\gG}\oplus_{\bar{\rho}}\bar{V})=l(\bar{\gG})=l(\gs)<\l(\gG)$. Finally, 
according to Proposition \ref{dec-sum}, the algebra 
$\bar{\gG}\oplus_{\bar{\rho}}\bar{V}$ disassembles into algebras of the form 
$\gH\oplus_{\tau}U$ with $l(\gH)\leq l(\bar{\gG})$ and $\gH$ simple. 

We shall call a $d$-pair $(\gs,W)$ in a Lie algebra $\gG$, as well as the 
corresponding $d$-involution, {\it simplifying} if $l(\gs)<l(\gG)$. 
In the rest of this section, we shall 
concentrate on the existence of simplifying $d$-pairs for Lie algebras of the 
form $\gG\oplus_{\rho}V$ \,with $\gG$ simple.  

%%%%%%%%%%%%%%%%%%%%%%%%
\bigskip
\noindent\textit{Multi-involution disassembling.} 
%%%%%%%%%%%%%%%%%%%%%%%%

\noindent 
Keeping in mind that the disassembling problem is reduced to that for abelian extensions of simple 
Lie algebras, we shall adapt the Stripping Lemma to abelian extensions of arbitrary
Lie algebras, because it is convenient to set the problem in this general context.

Let $P_1,\dots ,P_l$ be commuting involutions in a Lie algebra $\gG$. Denote by 
$\dF_2^l$ the algebra of $\dF_2$-valued $l$-vectors with coordinate-wise 
multiplication. Let $\varsigma=(\varsigma_1,\dots,\varsigma_l)\in\dF_2^l$.
The common eigenspace  of the involutions $P_1,\dots,P_l$ that corresponds to 
eigenvalues $\lambda_i=(-1)^{\varsigma_i},  \;i=1,\dots,l$, 
will be denoted by $|\gG |_{\varsigma}$. Then 
\begin{equation}
|\gG |=\bigoplus_{\varsigma\in \dF_2^l}|\gG |_{\varsigma}
\end{equation}
Obviously, $[|\gG |_{\varsigma},|\gG |_{\sigma}]=|\gG |_{\varsigma+\sigma}$, for $\varsigma$ and $\sigma\in \dF_2^l$.
For any $\varsigma\in \dF_2^l$, we consider the skew-symmetric and bilinear product 
$[\cdot,\cdot]_{\varsigma}$
defined on homogeneous elements of $\gG$ by the formula 
\begin{equation}
[u,v]_{\varsigma}=[u,v], \;\mbox{if} \;\;\xi \cdot\tau=\varsigma, 
\quad u\in |\gG |_{\xi}, \;v\in |\gG |_{\tau},\quad
\mbox{and zero otherwise}. 
\end{equation}
\begin{proc}\label{multi-inv} The product $[ \cdot , \cdot]_{\varsigma}$ defines a Lie 
algebra structure, $\gG_{\varsigma}$, on $|\gG |$. 
\end{proc}
\begin{proof}We have to check the Jacobi identity for the bracket 
$[\cdot,\cdot]_{\varsigma}$.  It follows directly from the definition that 
the double bracket $[u,[v,w]_{\varsigma}]_{\varsigma}$
with $\,u\in\gG_{\mu}, 
v\in\gG_{\nu}, w\in\gG_{\xi},$ is different from zero only if $\varsigma=0$ 
and $\mu\nu=\nu\xi=\xi\mu=0$. In this case 
$[u,[v,w]_{\varsigma}]_{\varsigma}=[u,[v,w]]$
and  $[u,[v,w]_{\varsigma}]_{\varsigma}+cycle=[u,[v,w]]+cycle=0$.
On the other hand, if $\varsigma, \mu, \nu, \xi$ do not satisfy the above condition, 
all the double commutators of elements $u,v$ and $w$ with respect to 
$[\cdot,\cdot]_{\varsigma}$ vanish. 
\end{proof}

The Lie algebras $\gG_{\varsigma}$'s are not in general mutually compatible. Nevertheless, 
some combinations appear implicitly in the \emph{multi-involution disassembling 
procedure} which is described below.

Let $I_0$ be the involution of $\gG=\gH\oplus_{\rho}V$ whose eigenspaces 
corresponding to eigenvalues $1$ and $-1$ are $|\gH|$ and $V$, respectively. An 
involution $I$ in $\gG$ commuting with $I_0$, and the corresponding $I$ $d$-pair,
will be called \emph{adapted} (to the semi-direct sum structure of $\gG$).
Obviously, both $|\gH|$ and $V$ are $I$-invariant. Let 
$\gH=\gH_0\oplus\gH_1, \,V=V_0\oplus V_1$ be the splittings into eigenspaces   
corresponding to the eigenvalues $1$ and $-1$ of $I$, respectively.  The $d$-pair associated with 
$I$ is $(\gH_0\oplus_{\rho_0} V_0, |\gH_1\oplus_{\rho_1} V_1|)$ where $\rho_i$ 
stands for the restriction of $\rho$ to $\gH_i$ and $V_i$. Removing, according to the 
Stripping Lemma, the associated dressing algebra, we get the following Lie algebra
\begin{equation}\label{newSemidirect}
(\gH_0\oplus_{\rho_0}V_0)\oplus_{\varrho}(|\gH_1|\oplus V_1)
\end{equation}
with the representation $\varrho$ defined by formulas
$$
\varrho(h_0)(h_1)=[h_0,h_1],  \;\varrho(h_0)(v_1)=\rho(h_0)(v_1), \;\varrho(v_0)(h_1)=
-\rho(h_1)(v_0), \;\varrho(v_0)(v_1)=0
$$
where $h_i\in \gH_i, v_i\in V_i, \;i=0,1$.  On the other hand, Algebra (\ref{newSemidirect}) 
may be viewed as a semi-direct product of $\gH_0$ and the ideal $\mathcal{I}$ whose 
support is $V_0\oplus|\gH_1|\oplus V_1$. The product $[\cdot,\cdot]'$ in this ideal is such 
that $[V_0,|\gH_1|]'\subset V_1, \,[V_0,V_1]'=[|\gH_1|,V_1]'=0$. So, $\mathcal{I}$ is 
nilpotent and as such can be completely disassembled. Since
$$
(\gH_0\oplus_{\rho_0}V_0)\oplus_{\varrho}(|\gH_1|\oplus V_1)=
\gH_0\oplus_{\varrho_0}|\mathcal{I}|+\gamma_m\oplus\mathcal{I}, \quad m=\dim\,\gH_0, 
$$
where $\varrho_0$ is the direct sum of the natural actions of $\gH_0$ on $|\gH_1|, \,V_0$ 
and $V_1$, the disassembling problem for Algebra (\ref{newSemidirect}) and, 
therefore, for $\gG$ is reduced to that for $\gH_0\oplus_{\varrho_0}|{\mathcal{I}}|$. This 
passage from $\gH\oplus_{\rho}V$ to $\gH_0\oplus_{\varrho_0}|{\mathcal{I}}|$ will be 
referred to as the \emph{stripping \emph{(of $\gH\oplus_{\rho}V$)} by $I$}.
\begin{proc}\label{stripping-semi}
Let $I_1,\dots ,I_l$ be commuting involutions on a Lie algebra $\gG$. Then $\gG$ can be assembled from lieons
and the algebra $\gG_{(0,\dots,0)}\oplus_{\rho} W$ where $W=
\oplus_{\zeta \neq  0} |\gG_{\zeta}|$
and $\rho$ is the direct sum of the  natural actions of $\gG_{(0,\dots,0)}$ on the $\gG_{\zeta}$'s.
\end{proc}
\begin{proof}
We shall proceed by induction. First, the Stripping Lemma applied to involution $I_1$ 
shows that $\gG$ can be assembled from lieons and the algebra  
$\gG_0\oplus_{\varrho_0}W_0, \,W_0=|\gG_1|$, where $|\gG_0|$ and $|\gG_1|$ are the
eigenspaces of $I_1$ corresponding to the eigenvalues $1$ and $-1$, respectively, 
and $\varrho_0$ is the natural action of $\gG_0$ on $|\gG_1|$. Since the involution 
$I_2$ commutes with $I_1$, it leaves invariant both $|\gG_0|$ and $|\gG_1|$ and, 
therefore, induces an adapted involution $I$ on the algebra 
$\gG_0\oplus_{\varrho_0}W_0$. Now, by stripping the semi-direct product 
$\gG_0\oplus_{\varrho_0}W_0$ by $I$, we see that $\gG$ can be assembled from lieons 
and $\gG_{(0,0)}\oplus_{\varrho_1}W_1$ where $W_1=\oplus_{\zeta \neq  0}|\gG_{\zeta}|$ 
with $\zeta\in\dF_2^2$. Continuing this process, we get the desired result.
\end{proof}

%%%%%%%%%%%%%%%%%%%%%%%%%%%%%
\bigskip
\noindent\textit{Complete disassembling of classical Lie algebras.}
%%%%%%%%%%%%%%%%%%%%%%%%%%%%%%

\noindent A simple application of Proposition\,\ref{stripping-semi}
is the following.
\begin{proc}\label{dis-classics}
Lie algebras  $\gs\gl(n,\gk), \,\go(n,\gk), \,\gs\go(n,\gk), \,\gu$ and $\gs\gu$ can 
be completely disassembled.
\end{proc}
\begin{proof}
Let $\{e_1,\dots,e_n\}$ be a basis in a $\gk$-vector space $V$. We shall 
identify operators in $\mathrm{End}\,V$ and their matrices in this basis. 
Consider the $d$-pair $(\gs_j,W_j)$ in $\gG=\gG\gl(V)$ associated 
with the subspaces $V_1=\mathrm{span}\{e_j\}$ and 
$V_2=\mathrm{span}\{e_1,\dots,\hat e_j,\dots,e_n\}$ by the construction of 
Example \ref{pair}  and denote the corresponding involution in 
$\gG\gl(V)$  by $I_j$ . Note that involutions $I_1,\dots,I_n$ commute.  

The common eigenspace  
$|\gG_{\zeta}|, \,\zeta\in \dF_2^n$, of involutions $I_1,\dots,I_n$, as it is easy 
to see, is different from zero if and only if the number of non-zero components in $\zeta$ 
is zero or $2$. In the first case the subspace $|\gG_{(0,\dots,0)}|$ is 
composed of operators for which $e_1,\dots,e_n$ are eigenvectors, i.e., 
it consists of diagonal matrices. In the second case, let $\zeta=(ij)\in \dF_2^n, 
i\neq j$, be the $\dF_2$--vector with two nonzero components at the $i$-th 
and $j$-th places. Then the subspace $|\gG_{(ij)}|$ consists of operators 
$\lambda \varepsilon_{ij}+\mu\varepsilon_{ji}, \;\lambda,\mu\in \gk$, where 
$\varepsilon_{ij}$ is the operator sending $e_j$ to $e_i$ and 
annihilating  the $e_k$'s for $k\neq j$. So the algebra 
$\gG_{(0,\dots,0)}$ is abelian. Therefore, the algebra $\gG_{(0,\dots,0)}\oplus_{\rho} W$
is solvable and can be completely disassembled. Now it directly follows from 
Proposition\,\ref{stripping-semi} that the algebra $\gG\gl(n,\gk)=\gG\gl(V)$ can be 
completely disassembled as well.

Algebras $\gs\gl(n,\gk), \,\go(n,\gk), \,\gs\go(n,\gk)$ are invariant with respect to the
involutions $I_j$'s. Obviously, for each of them the subspace
$|\gG_{\zeta}|$ is a subspace of  $|\gG\gl(n,k)_{\zeta}|$.
In particular, this shows that the algebra $\gG_{(0,\dots,0)}\oplus W$ is solvable, and
Proposition\,\ref{stripping-semi} gives the desired result.

In order to completely disassemble the symplectic algebra $\gs\gp(n,\gk)$, the preceding
procedure must be slightly modified. Let $\sigma(u,v)$ be a symplectic form on 
$V, \,\dim\,V=2n$. We interpret the algebra  $\gs\gp(n,\gk)$ as the algebra
$$
\gs\gp(\sigma)=\{A\in\mathrm{End}\;V \;|\; \sigma(Au,v)+\sigma(u,Av)=0, \,\forall u,v\in V\}.
$$
Let $V=V_1\oplus\dots\oplus V_n, \,\dim\,V_i=2, \forall i$, be a 
$\sigma$-orthogonal decomposition of $V$ and $P_i:V\rightarrow V$ the associated 
projector on $V_i$. Then $I_i=\id_V-2P_i$ is an involution of $\gG\gl(V)$. It is easy to 
see that involutions $I_i$'s commute and their common eigenspace, on which they 
are the identity maps, is $\gs\gp(\sigma_1)\oplus,\dots,\oplus\gs\gp(\sigma_n)$ with  \quad 
$\sigma_i=\sigma|_{V_i}$ (=$\gG_{(0.\dots,0)}$ in the notation of 
Proposition\,\ref{stripping-semi}). Note that $\gs\gp(\sigma_i)$ is isomorphic to 
$\gs\gp(2,\gk)$. So, in contrast with the preceding case, 
the algebra $\gG_{(0.\dots,0)}\oplus W$ of Proposition\,\ref{stripping-semi} is not 
solvable. Therefore, it cannot be completely disassembled by the techniques developed previously. This small difficulty can be resolved by the introduction of an additional involution.  

Let $J_i:V_i\rightarrow V_i$ be a complex structure on $V_i$ compatible with 
$\sigma_i$. This means that $J_i^2=-\id_{V_i}$ and $\sigma_i(J_iu,v)+\sigma_i(u,J_iv)=0$. 
If $J=J_1\oplus\dots\oplus J_n$, then $\sigma(Ju,v)+\sigma(u,Jv)=0$ and 
$$
I_0: \End\,V\rightarrow\End\,V, \quad \quad A\mapsto -JAJ
$$ 
is an involution,  which leaves invariant 
the subalgebra $\gs\gp(\sigma)$. Moreover, $I_0$ commutes with the involutions 
$I_1,\dots,I_n$, and their common eigenspace, on which they act as the identity, 
is the abelian subalgebra $\gH$ composed of elements 
$\lambda_1J_1\oplus\dots\oplus\lambda_nJ_n, \,\lambda_1,\dots,\lambda_n\in\gk$.   
Now, applying Proposition\,\ref{stripping-semi} to the involutions 
$I_0,I_1,\dots,I_n$ and taking into account that in this case the algebra 
$\gG_{(0.\dots,0)}\oplus W=\gH\oplus W$ is solvable, we see that $\gs\gp(\sigma)$ 
can be completely disassembled.

Similar arguments can be applied to the algebras $\gu(n)$ and $\gs\gu(n)$.
Interpret an $n$-dimensional complex vector space as a $2n$--dimensional 
$\R$--vector space $V$ supplied with a complex structure $J, \,J^2=-\id_V$. Then 
split $V$ into a direct sum of 2-dimensional $J$-invariant subspaces and  consider, 
as above, the corresponding involutions $I_1,\dots,I_n$. In this case the algebra 
$\gG_{(0.\dots,0)}$ is abelian. Exactly as in the previous case, it consists of operators 
$\lambda_1J_1\oplus\dots\oplus\lambda_nJ_n, \,\lambda_1,\dots,\lambda_n\in\gk$, 
with $J_i=J|_{V_i}$ and Proposition\,\ref{stripping-semi} gives the desired result again.
\end{proof}

This proof of Proposition\,\ref{stripping-semi} is not very constructive in the sense 
that the corresponding a-scheme is rather complicated. It
was given here with the aim of illustrating the Stripping Lemma at work. A short and 
constructive procedure of complete disassembling of
classical Lie algebras will be described in 
Section\,\ref{classical}.

We remark that any simple Lie algebra $\gG$ over a ground field $\gk$ possesses a nontrivial involution. A more difficult question is whether such an involution can be extended to the algebra $\gG\oplus_{\rho}V$.

%%%%%%%%%%%%%%%%%%%%%%%%%%
\subsection{Extensions of $d$-pairs and involutions} 
%%%%%%%%%%%%%%%%%%%%%%%%%%

Let $\rho$ be a representation of $\gG$ in a vector space $V$ and let $(\gs,W)$ be 
a $d$-pair in $\gG$. 
\begin{defi}A decomposition 
$V=V_0\oplus V_1$ is called a \emph{$\rho$-extension} of $(\gs,W)$  if 
\begin{enumerate}
\item $V_i$ \,is invariant with respect to the operators $\rho (s), 
s\in\gs, \;i=0,1$; 
\item $\rho(w)(V_0)\subset V_1,\,\, \rho(w)(V_1)\subset V_0$, \,if  $w\in W$. 
\end{enumerate}
\end{defi}
\begin{lem}
If $V=V_0\oplus V_1$ is a $\rho$--extension of $(\gs,W)$, then 
$(\gs\oplus_{\rho|_{_{\gs}}}V_0,\,W\oplus V_1)$ is a $d$-pair in the Lie 
algebra $\gG\oplus_{\rho}V$. 
\end{lem}
\begin{proof}
Obviously, $(\gs\oplus_{\rho|_{_{\gs}}}V_0) $  is a subalgebra of 
$\gG\oplus_{\rho}V$. Denoting the Lie product in 
$\gG\oplus_{\rho}V$  by $[\cdot,\cdot]_{\rho}$, we have
\begin{align*}
[(s,v_0),(w,v_1)]_{\rho} &=([s,w],\rho(s)(v_1)-\rho(w)(v_0))\in
W\oplus V_1\\
[(w,v_1),(w',v_1')]_{\rho} &=([w,w'],\rho(w)(v_1')-\rho(w')(v_1)) \in 
\gs\oplus_{\rho|_{_{\gs}}}V_0
\end{align*}
where $v_0\in V_0, v_1\in V_1, w\in W,$ etc. In other words, 
$$
[(\gs\oplus_{\rho|_{_{\gs}}}V_0,\,W\oplus V_1]_{\rho}\subset W\oplus V_1,\quad 
[W\oplus V_1,W\oplus V_1]_{\rho}\subset\gs\oplus_{\rho|_{_{\gs}}}V_0.
$$
\end{proof}

Let $(\gs,W)$ and $\rho$ be as above. A linear operator $A:V\rightarrow V$ is 
called {\it splitting} (with respect to $(\gs,W)$ and $\rho$) if 
\begin{equation}\label{spl-oper}
\rho(s)\circ A - A\circ\rho(s)=0,\quad \rho(w)\circ A+A\circ\rho(w)=0,\quad 
s\in\gs,\quad w\in W.
\end{equation}
In particular, $A$ is an endomorphism of the $\gs$-module $(V,\rho|_{_{s}})$. A 
splitting operator $A$ is called a {\it splitting involution} if, in addition, 
$A^2=id_V$. The splitting involution $id_{V_1}\oplus (-id_{V_2})$ is naturally 
associated with a $\rho$-extension $V=V_0\oplus V_1$ and vice-versa. Note also that the splitting operators with respect to $(\gs,W)$ \,and $\rho$ form a 
vector space, which will be denoted by $\mathcal{S}_1=\mathcal{S}_1(\gs,W,\rho)$, and that 
the product of two splitting operators is an endomorphism of the $\gG$-module $(V,\rho)$. 
So, denoting by $\mathcal{S}_0(\rho)$ the algebra of these endomorphisms, we see that
$$
\mathcal{S}(\gs,W,\rho)=\mathcal{S}_0(\rho)\oplus\mathcal{S}_1(\gs,W,\rho)
$$
is an associative $\dF_2$--graded algebra

Denote by $V_{(\lambda)}$ the root space of $A$ corresponding to an eigenvalue 
$\lambda\in \gk$ \, of $A$. 
\begin{lem}
Let $A$ be a splitting operator. Then 
$$
\rho(s)(V_{(\lambda)})\subset V_{(\lambda)}, \quad 
\rho(w)(V_{(\lambda)})\subset V_{(-\lambda)}, \quad \mathrm{if}\quad s\in\gs,\quad w\in W.
$$
\end{lem}\label{pre-spl}
\begin{proof}
Let $I=id_V$. Then, obviously, 
$$
(\lambda I-A)\circ\rho(s)=\rho(s)\circ (\lambda I-A),\quad (\lambda 
I+A)\circ\rho(w)=\rho(w)\circ (\lambda I-A)
$$
for any $s\in\gs,\quad w\in W$. Since $V_{(\lambda)}= \ker (\lambda I-A)^k$ for 
some $k\in\N$, the assertion directly follows from the above relations. 
\end{proof}
\begin{cor}\label{ospl}
Assume that the eigenvalues of a non-degenerate splitting operator $A$  belong to $\gk$ 
and divide them into two parts $\Lambda_0$ and $\Lambda_1$, in such a way 
that opposite eigenvalues $\lambda$ and $-\lambda$ do not belong to the same 
part. Then the pair 
$$
V_0=\sum_{\lambda\in\Lambda_0}V_{(\lambda)},\quad 
V_1=\sum_{\lambda\in\Lambda_1}V_{(\lambda)}
$$
is an extension of $(\gs,W)$. In particular, if $A$ is a splitting involution, 
then the pair $(V_{(1)},V_{(-1)})$ \,is an extension of $(\gs,W)$. 
\end{cor}
\begin{proof}
The proof follows from the above lemma. 
\end{proof}

A division of eigenvalues of a splitting operator $A$ as in corollary\,\ref{ospl} 
is possible only if $A$ is non-degenerate. The following proposition simplifies the 
search for non-degenerate 
operators in $\mathcal{S}_1(\gs,W,\rho)$.
\begin{proc}\label{ext-ext}
Let $\bar{\gk}$ \,be an extension of the ground field $\gk$ and 
$\bar{\gs},\bar{W},\bar{\rho}$ be the corresponding extensions of
$\gs, W$ and $\rho$, respectively. Then 
\begin{enumerate}
\item
$\mathcal{S}_0(\bar{\rho})=\mathcal{S}_0(\rho)\otimes_{\gk}\bar{\gk},  
\quad\mathcal{S}_1(\bar{\gs},\bar{W},\bar{\rho})=\mathcal{S}_1(\gs,W,\rho)
\otimes_{\gk}\bar{\gk}$;  
\item If \;$\mathcal{S}_1(\bar{\gs},\bar{W},\bar{\rho})$ \,contains a
non-degenerate operator, then $\mathcal{S}_1(\gs,W,\rho)$ does also. 
\end{enumerate}
\end{proc}
\begin{proof}
The inclusion $\mathcal{S}_1(\gs,W,\rho)\otimes_{\gk}\bar{\gk}\subset 
\mathcal{S}_1(\bar{\gs},\bar{W},\bar{\rho})$ is obvious. But $\mathcal{S}_1(\bar{\gs},\bar{W},\bar{\rho})$ 
\,is the solution space of the linear system (\ref{spl-oper}) interpreted as a 
system over $\bar{\gk}$. So, its dimension over $\bar{\gk}$ coincides with 
that of  the solution space of (\ref{spl-oper}) over $\gk$, i.e., with the dimension of
$\mathcal{S}_1(\gs,W,\rho)$.
Hence $\mathcal{S}_1(\gs,W,\rho)\otimes_{\gk}\bar{\gk}= 
\mathcal{S}_1(\bar{\gs},\bar{W},\bar{\rho})$. Similar arguments prove that
$\mathcal{S}_0(\bar{\rho})=\mathcal{S}_0(\rho)\otimes_{\gk}\bar{\gk}$.

To prove the second assertion, fix a base in $\mathcal{S}_1(\gs,W,\rho)$ and
consider the polynomial 
$P(t)=\det\,(t_1A_1+\dots+t_mA_m)$ in the variables $t_i$'s. The zero set
of $P(t)$ in $\gk^m$ is in one-to-one correspondence with the degenerate operators in 
$\mathcal{S}_1(\gs,W,\rho)$. On the other hand, the extended operators 
$\bar{A}_i=A_i\otimes_{\gk}\bar{\gk}, \,i=1,\dots,m$, form a basis in 
$\mathcal{S}_1(\bar{\gs},\bar{W},\bar{\rho})$ and the degenerate operators in
$\mathcal{S}_1(\bar{\gs},\bar{W},\bar{\rho})$ are in one-to-one correspondence
with the zeros of the polynomial $\bar{P}(t)=\det\,(t_1\bar{A}_1+\dots+t_m\bar{A}_m)$
in $\bar{\gk}^m$. Since $\mathcal{S}_1(\bar{\gs},\bar{W},\bar{\rho})$ contains a 
non-degenerate operator, the polynomial $\bar{P}(t)$ is nontrivial. But, by 
construction,  $P(t)=\bar{P}(t)$. Since the field $\gk$ is infinite, any nontrivial
polynomial with coefficients in $\gk$ is non-zero as a function on $\gk^m$.
\end{proof}

Since the structure of representations of simple Lie algebras over 
algebraically closed fields is well-known, this proposition is useful in the search for $d$-pairs in abelian extensions of simple Lie algebras
over arbitrary ground fields.

%%%%%%%%%%%%%%%%%%%%%%%%%%%%%%%
\subsection{Some properties of the algebra $\mathcal{S}(\gs,W,\rho)$}
%%%%%%%%%%%%%%%%%%%%%%%%%%%%%%%
 In this subsection we keep the notation of the previous one.
\begin{lem}\label{k-rootsInS_1}
Assume that $0\neq A\in\mathcal{S}_1(\gs,W,\rho)$ and that $\rho$ is irreducible. If one of the 
eigenvalues $\lambda$ of $A$ belongs to $\gk$, then 
\begin{enumerate}
\item $V=\Ker(A^2-\lambda^2I)$ and $\lambda\neq 0$;
\item $V_0=\Ker(A-\lambda I), \,V_1=\Ker(A+\lambda I)$ is a $\rho$--extension of 
$(\gs, W)$. 
\end{enumerate}
\end{lem}
\begin{proof} 
If $\lambda=0$, then, obviously,
$\Ker\,A$ is $\rho$--invariant and hence $ \Ker\,A=V$, i.e., $A=0$ in contradiction 
with the assumption. So, $\lambda\neq 0$.

Then we have $0\neq \Ker(A-\lambda I)\subset\Ker(A^2-\lambda^2I)$.
On the other hand, $\rho(w)$ sends $\Ker(A-\lambda I)$ to $\Ker(A+\lambda I)$
and vice-versa.
Therefore $\Ker(A^2-\lambda^2I)$ is $\rho$--invariant, and $V=\Ker(A^2-\lambda^2I)$,
since $\rho$ is irreducible. 

The second assertion directly follows from Corollary\,\ref{ospl}.
\end{proof}
An immediate consequence of this lemma is
\begin{cor}\label{cor-k-rootsInS_1}
Let $\gk$ be algebraically closed and $\rho$ irreducible. If $\mathcal{S}_1(\gs, W, \rho)$ 
is nontrivial, then there is a $\rho$--extension of $(\gs, W)$. 
\end{cor}
\begin{proc}\label{S-asTelo}
Let $\gG$ be simple and let $\rho$ be irreducible. Then
\begin{enumerate}
\item $\mathcal{S}_0(\rho)$ is a division algebra $($over $\gk)$.
\item If $\mathcal{S}(\gs, W, \rho)$ is not a division algebra, then
the $d$-pair $(\gs, W)$ admits a $\rho$--extension.
\end{enumerate}
\end{proc}
\begin{proof}
The first assertion is the classical Schur lemma. Let $0\neq A=A_0+A_1
\in\mathcal{S}(\gs, W, \rho)$  be a degenerate operator with 
$\,A_0\in\mathcal{S}_0(\rho)$  and $A_1\in\mathcal{S}_1(\gs, W, \rho)$. Such an operator exists, since $\mathcal{S}(\gs, W, \rho)$ is not a division algebra. The first assertion of the proposition implies that $A_1\neq 0$ and $A_0^{-1}\in\mathcal{S}_0(\rho)$, if $A_0\neq 0$. In this case the operator $B=AA_0^{-1}=I+A_1A_0^{-1}$
is degenerate too, and, as a consequence, one of the eigenvalues of 
$B_1= A_1A_0^{-1}\in \mathcal{S}_1(\gs, W, \rho)$ is $-1$. Now  the second assertion
follows from Lemma \ref{k-rootsInS_1}. Moreover, this lemma shows that, in fact, 
$A_0\neq 0$. Indeed, assuming the contrary we see that $A_1$ is degenerate and, therefore, one of its eigenvalues is $0$ in contradiction with the lemma.
\end{proof}

Now we shall specialise the above results in the case where $\gk=\R$.
\begin{proc}\label{AlmExt}
Let $\gG$ be a simple Lie algebra over $\R$, $\rho:\gG\to \End\,V$ an irreducible 
representation of $\gG$, and $(\gs, W)$ a $d$-pair in $\gG$. If $\mathcal{S}_1(\gs, W, \rho)$ 
is not trivial, then $(\gs, W)$
admits a $\rho$--extension except, possibly, in the case where  $\mathcal{S}(\gs, W, \rho)$
is isomorphic to $\C$.
\end{proc}
\begin{proof}
Proposition \ref{S-asTelo} allows us to restrict to the case when  
$\mathcal{S}(\gs, W, \rho)$ is a division algebra. Since $\mathcal{S}_1(\gs, W, \rho)$ 
is nontrivial, the dimension of $\mathcal{S}(\gs, W, \rho)$ is greater than 1. 
Now  the classical Frobenius theorem tells us that this algebra is isomorphic either 
to $\C$ or to $\Q$, and we have to analyse the second alternative only. 

In this case the dimension of the division algebra $\mathcal{S}_0(\rho)$ is 
less than 4, since $\mathcal{S}_1(\gs, W, \rho)$ is nontrivial. By the Frobenius theorem
this implies that $\mathcal{S}_0(\rho)$ is isomorphic either to $\R$ or to $\C$.
The first case is impossible. Indeed, in this case the dimension
of the subspace $\mathcal{S}_1(\gs, W, \rho)\cdot\mathcal{S}_1(\gs, W, \rho)$ is
greater than 1 in contradiction with the fact that it should be contained in
the 1-dimensional space $\mathcal{S}_0(\rho)$.

So, $\mathcal{S}_0(\rho)$ is isomorphic to $\C$, and $V$ acquires the structure of a vector 
space over $\C$ by means of one of two operators $J\in\mathcal{S}_0(\rho)\subset\End\,V$ 
such that $J^2=-\id_V$. 

Denote by $V_{\C}$ this complex vector space. The representation
$\rho$ naturally extends to a representation $\rho_{\C}:\gG^{\C}\to \End_{\C}V_{\C}$ of the 
complexification $\gG^{\C}=\gG\otimes_{\R}\C$ of $\gG$ in $V_{\C}$: 
$$
\rho_{\C}(g\otimes \sqrt{-1})\df J(\rho(g)), \,g\in\gG.
$$
By Corollary \ref{cor-k-rootsInS_1}, the $d$-pair $(\gs\otimes_{\R}\C, W\otimes_{\R}\C)$ 
in $\gG^{\C}$ admits a $\rho_{\C}$--extension whose restriction to $\rho$ is, obviously, 
a $\rho$--extension of $(\gs, W)$.
\end{proof}

%%%%%%%%%%%%%%%%%%%%%%%%%%%%%%%
\subsection{Some $d$-pairs associated with 3-dimensional simple subalgebras}
\label{d-3}
%%%%%%%%%%%%%%%%%%%%%%%%%%%%%%%

Here we shall construct  simplifying $d$-pairs for abelian extensions of Lie algebras
possessing a simple $3$--dimensional subalgebra. 

First, we shall collect some necessary facts about simple $3$-dimensional 
algebras (see, for instance, \cite{Jac}). Let $\gH$ be a  simple $3$-dimensional 
algebra and let  
$h\in \gH$ be a regular element. Then there exists a basis 
$(e_1,e_2,e_3=h)$ in $|\gH|$ such that $[e_1,e_2]=h, \;[h,e_1]=\alpha e_2, 
\;[h,e_2]=\beta e_1, \;, \alpha,\beta\in\gk$ (the ground field), 
\;$\alpha\beta\neq 0$. Set $\kappa=\alpha\beta$. So, the characteristic 
polynomial of $\textrm{ad}\,h$ is $t(t^2-\kappa)$. If $\kappa=\lambda^2, 
\;\lambda\in \gk$, then $\gH$ splits and there exists a basis 
$(h'=2\lambda^{-1}h,x,y)$ \;of $\gH$,\;called an $\gs\gl_2$-triple, such that 
$[h',x]=2x, \;[h',y]=-2y, \;[x,y]=h'$. If $\kappa$ is not a square in 
$\gk$, i.e., if the polynomial $t^2-\kappa$ is irreducible, consider the 
extension $\bar{\gk}$ \;of $\gk$ \; by adding the square  roots of 
$\kappa$ to $\gk$. We still denote these roots by 
$\pm\lambda\in\bar{\gk}$. The extended algebra 
$\bar{\gH}=\gH\otimes_{\gk}\bar{\gk}$ splits over $\bar{\gk}$ and, as 
before, one can find an $\gs\gl_2$-triple $(h'=2\lambda^{-1}h,x,y)$ in it. 
Recall also, that if $\varrho$ \;is a representation of $\gH$, or of 
$\bar{\gH}$, then the eigenvalues of $\varrho(h')$ are integers and 
the multiplicities of opposite eigenvalues are equal. So, the eigenvalues of $h$ 
are of the form $\pm(m/2)\lambda$ \; with $\lambda^2= \kappa, \;m\in \Z$. 
Since the element $h$ is semi-simple, the operator $\varrho(h)$ is semi-simple 
as well (see \cite{Hum}). Therefore, the representation space $U$ of $\varrho$ 
splits into a direct sum of 1-dimensional and 2-dimensional 
$\varrho(h)$-invariant subspaces in such a way that the 1-dimensional subspaces 
belong to $\ker\,\varrho(h)$ \;while each of the 2-dimensional ones is 
annihilated by an operator $\varrho(h)^2-(1/4)m^2\kappa$ for a suitable 
integer $m\neq 0$. We shall call them \emph{eigenlines} and 
\emph{$m$-eigenplanes}, respectively. Obviously, if 
$\gH$  splits, then any eigenplane splits into the two eigenlines generated 
by the eigenvectors of eigenvalues $\pm(m/2)\lambda$. In the non-split case, the 
eigenplanes are irreducible with respect to  $\varrho(h)$. 
 
Now we shall associate a $d$-pair to a simple 3-dimensional subalgebra $\gH$
of a Lie algebra $\gG$. First, we recall the following elementary fact. Let $\gA$ be a Lie 
algebra, $x\in \gA$  and $\gA_{\mu}$ the root space of the 
operator $\textrm{ad}\,x$ \;corresponding to the eigenvalue $\mu$. Then 
\begin{equation}\label{comut}
[\gA_{\mu},\gA_{\nu}]\subset\gA_{\mu+\nu}.
\end{equation}

Let $h\in\gH$ be as above and  $A=\textrm{ad}_{\gG}h$. Set 
$$
\gG_0=\ker\,A, \quad\gG_m=\ker(A^2-(m^2/4)\kappa\;\textrm{id}_{\gG}).
$$ 
Then $\gG=\bigoplus_{m\geq 0}\gG_m$ and there are commutation relations
\begin{equation}\label{g-com}
[\gG_k, \gG_l]\subset \gG_{k+l}\oplus \gG_{k-l}.
\end{equation}
If $\gH$ splits, this directly follows from (\ref{comut}). Indeed, 
in this case  $\gG_m$, for $m>0$, \;splits into a direct sum of root spaces 
corresponding to the eigenvalues $\pm(m/2)\lambda$ (remark that $\gG_k=\gG_{-k}$). 
If $\gH$ does not split, one obtains the result by extending the scalars from 
$\gk$ to $\bar{\gk}$. In fact, the extended subalgebra $\bar{\gH}$ of the 
extended algebra $\bar{\gG}$ splits and hence the extended analogues 
$\bar{\gG}_m$'s of the subspaces $\gG_m$'s commute according to (\ref{g-com}), 
while $\gG_m\subset \bar{\gG}_m$. 

Relations (\ref{g-com}) show that
\begin{equation}\label{1stD-pair}
\gs=\bigoplus_{m\geq 0}\gG_{2m}, \quad  W=\bigoplus_{m\geq 0}\gG_{2m+1}
\end{equation}
is a $d$-pair in $\gG$, which will be called the \emph{first d-pair associated with} 
$\gH$ if $W\neq\{0\}$. 
Since $\gH\subset\gs$, this $d$-pair is trivial if and only if $\gG=\gs\oplus_{\rho} W$. 
In particular, it is nontrivial and simplifying if  $\gG$ is semi-simple.

If $W=\{0\}$, i.e., $\gG=\oplus_{m\geq 0}\gG_{2m}$, then
\begin{equation}\label{2dD-pair}
\gs=\bigoplus_{m\geq 0}\gG_{4m}, \quad  W=\bigoplus_{m\geq 0}\gG_{4m+2}
\end{equation}
is the \emph{second d-pair associated with} $\gH$. Since $h\in\gG_0\subset\gs$ 
and $x,y\in\gG_2\subset W$, this $d$-pair is nontrivial and, obviously, simplifying 
if $\gG$ is simple.

%%%%%%%%%%%%%%%%%%%%%%%%%%%%%%%%
\subsection{Solution of the disassembling problem for algebraically closed fields}
%%%%%%%%%%%%%%%%%%%%%%%%%%%%%%%%

With $d$-pairs (\ref{1stD-pair}) and (\ref{2dD-pair}) at our disposal we can solve
the disassembling problem for Lie algebras over algebraically closed fields.
In this subsection, we keep the notation of the previous one and we assume the ground 
field $\gk$ to be algebraically closed.

\begin{proc}\label{split-by-3d}
Let $\gG$ be a Lie algebra over an algebraically closed field which possesses a 
simple $3$-dimensional subalgebra $\gH$,
and let $\rho$ be a representation of $\gG$ in $V$. Then one of the two $d$-pairs associated
with $\gH$ admits a $\rho$-extension which is nontrivial if $\gG$ is simple.
\end{proc}
\begin{proof}
We identify $\gG$ (resp., $V$) with the subalgebra $\gG\oplus_{\rho}\{0\}$ 
(resp., the subspace $\{0\}\oplus_{\rho}V$) in the algebra $\gG\oplus_{\rho}V$. 
Let $h\in\gH\subset\gG$ be as in Subsection \ref{d-3}. Set 
$$
B=\rho(h), \quad V_0=\ker\,B, \quad V_m=\ker(B^2-(m^2/4)\kappa\;\textrm{id}_V).
$$ 
Then $V=\oplus_{m\geq 0}V_m$. If $\gG$ is simple, then the first $d$-pair 
$(\gs, W)$ associated with $\gH$ is nontrivial if $W\neq\{0\}$, as we have already
observed after Formula (\ref{1stD-pair}). If $W\neq\{0\}$, then
\begin{equation}\label{1st-ext}
\gs_{\rho}\df\gs\oplus\left(\oplus_{m\geq 0}V_{2m}\right), 
\quad W_{\rho}\df W\oplus_{m\geq 0}V_{2m+1}
\end{equation}
is the required extension. Indeed, this directly follows from the commutation relations
\begin{equation}\label{V-com}
[\gG_k, V_l]\subset V_{k+l}\oplus V_{k-l},
\end{equation}
which can be proved by the same arguments as for (\ref{g-com}). Note that
this part of the proof does not require $\gk$ to be algebraically closed.

If $W=0$, we consider finer decompositions of $\gG$ and $V$ using the fact 
that $\gH$ splits if $\gk$ is algebraically closed. Set
\begin{equation}
\gG_m^{\prime}=\ker\;(A-\frac{m}{2}\lambda\id_{\gG}), \quad
V_m^{\prime}=\ker\;(B-\frac{m}{2}\lambda\id_V). 
\end{equation}
Then, obviously, $\gG_m=\gG_m^{\prime}\oplus\gG_{-m}^{\prime}, 
\;V_m=V_m^{\prime}\oplus V_{-m}^{\prime}$ and (see (\ref{comut}))
\begin{equation}\label{l-com}
[L_k, L_l]\subset L_{k+l}.
\end{equation}
where $L_s$ stands for one of the subspaces $\gG_s^{\prime}, \,V_s^{\prime}$.
Now it immediately follows from Relations (\ref{l-com})  that the subspaces
\begin{equation}\label{fine-V} 
 \mathcal{V}_0=\bigoplus_{k\in\Z}(V_{4k}^{\prime}\oplus V_{4k+1}^{\prime}),
 \quad \mathcal{V}_1=\bigoplus_{k\in\Z}(V_{4k+2}^{\prime}\oplus V_{4k+3}^{\prime}).
\end{equation}
provide a $\rho$-extension of the second  $d$-pair associated with $\gH$. As we have
already observed the second $d$-pair is nontrivial.
\end{proof}

An important consequence of Proposition\,\ref{split-by-3d} is
\begin{cor}\label{nontrivS_1}
Let $\gH$ be a simple 3-dimensional  subalgebra of a Lie algebra $\gG$
over a ground field $\gk$ and let $(\gs, W)$ be the nontrivial $d$-pair associated with $\gH$. 
Then $\mathcal{S}_1(\gs, W, \rho)$ is nontrivial. 
\end{cor}
\begin{proof}
This corollary follows from Proposition\,\ref{ext-ext} (1).
\end{proof}
\begin{thm}\label{C-dis}
Any finite-dimensional Lie algebra over an algebraically closed field of 
characteristic zero can be completely disassembled. 
\end{thm}
\begin{proof}
By Morozov's lemma, any simple Lie algebra $\gG$ over an algebraically closed field $\gk$ of 
characteristic  zero possesses a 3-dimensional subalgebra $\gH$ isomorphic to 
$\gs\gl(2,{\gk})$ (see \cite{Jac}, \cite{Hum}). If the algebra $\gG$ in Proposition
\,\ref{split-by-3d} is simple, then the $\rho$-extension of one of the $d$-pairs
$(\gs, W)$ associated with $\gH$ is simplifying. Indeed, the Stripping Lemma 
applied to this extended $d$-pair leads to an algebra of the form $\gs\oplus_{\rho'}V'$
(see the proof of Proposition\,\ref{split-by-3d})   whose semi-simple part coincides
with that of $\gs$. Hence $l(\gs)<l(\gG)$ since $\gs$ is a proper subalgebra of $\gG$.
\end{proof}

%%%%%%%%%%%%%%%%%%%%%%%%%%%%%
\subsection{Hyper-simple Lie algebras}
%%%%%%%%%%%%%%%%%%%%%%%

\noindent In this section, we discuss those simple Lie algebras  which cannot be 
directly disassembled by the preceding methods. 
\begin{defi}
A simple Lie algebra is called \emph{hyper-simple} if all its proper 
subalgebras are abelian. 
\end{defi}
This definition is justified by the following
\begin{proc}\label{simplest-in}
A simple Lie algebra $\gG$ over a field $\gk$ of characteristic zero contains either 
a hyper-simple  subalgebra or a subalgebra isomorphic to $\gs\gl(2,\gk)$. 
\end{proc}
\begin{proof}
If $\gG$ is not hyper-simple, then it contains a proper non-abelian subalgebra 
$\gH$. If the semi-simple part of $\gH$ is nontrivial, then $\gH$ contains a proper
simple subalgebra, and one gets the desired result by obvious induction arguments. 
If, on the contrary, the semi-simple part is trivial, then $\gH$ is a non-abelian
solvable algebra and, therefore, it contains a nontrivial nilpotent element 
$g$. Hence the endomorphism $ad_{\gG}g$ of $|\gG|$ has a nontrivial 
nilpotent part. In other words, $g$, considered as an element of $\gG$, has 
a nontrivial nilpotent part $g_n$ which, according to a well-known property 
of semi-simple algebras, belongs to $\gG$, i.e., $\gG$ possesses a nontrivial 
nilpotent element. By Morozov's lemma, such an element is contained in a 
3-dimensional subalgebra of $\gG$ isomorphic to $\gs\gl(2,\gk)$. 
\end{proof}

\begin{cor}\label{semisimple}
All the elements of a hyper-simple Lie algebra are semi-simple. 
\end{cor}
\begin{proof}
Assume that $\gG$ has a non-semi-simple element. Then the nilpotent part
of such an element is nontrivial and, by a well-known property of semi-simple
Lie algebras, belongs to $\gG$. Hence $\gG$ has a nontrivial nilpotent
element. This element is contained in a $3$-dimensional 
subalgebra of $\gG$ which is isomorphic to $\gs\gl(2,\gk)$ (see the proof 
of the above proposition). But $\gs\gl(2,\gk)$ and, therefore, $\gG$ contains a 
$2$-dimensional non-abelian subalgebra in contradiction with the fact that
$\gG$ is a hyper-simple algebra. 
\end{proof}

The existence and diversity of hyper-simple algebras depend exclusively on the
arithmetic properties of the ground field $\gk$. For instance, there are no 
hyper-simple Lie algebras over algebraically closed fields. Indeed, any simple Lie 
algebra over such a field $\gk$ contains a $3$-dimensional simple 
subalgebra isomorphic to $\gs\gl(2,\gk)$, which, in turn, contains a 
proper $2$-dimensional non-abelian subalgebra. On the contrary, there is 
only one (up to isomorphism) hyper-simple algebra over $\R$, namely, 
$\gs\go(3,\R)$ (see Proposition\,\ref{R-simp} below). 

Now we shall collect some elementary properties of hyper-simple algebras. Below
$C_x$ stands for the centralizer of an element $x\in\gG$. 
. 

\begin{proc}\label{solid property}
Let $\gG$ be a hyper-simple Lie algebra and $0\neq x\in \gG$. Then 
\begin{enumerate}
\item $C_x$ is abelian;
\item if\, $y,z\in\gG, \,y\neq 0$, and
$[x,y]=[z,y]=0$, then $[x,z]=0$; 
\item if\, $0\neq y\in\gG$, then either $C_x=C_y$ or $C_x\cap
C_y=\{0\}$; 
\item $C_x$ is a Cartan subalgebra of $\gG$, i.e., $y\in C_x$ if\, $[x,y]\in 
C_x$;
\item all nonzero elements of $\gG$ are regular;
\item $[\gG,C_x]\cap C_x=\{0\}$.
\end{enumerate}
\end{proc}
\begin{proof}

(1) Since $0\neq x\in C_x$, the center of $C_x$ is nontrivial. But  the center
of $\gG$ is trivial. Hence $C_x$ does not coincide with $\gG$, i.e., it is a 
proper subalgebra of $\gG$. As such it is abelian. 

\noindent(2) Obviously, $x$ and $z$ belong to $C_y$, which, by (1), is abelian. 

\noindent(3) If $0\neq y\in C_x$, then, according to (2), any $z\in C_y$ 
belongs to $C_x$. 

\noindent(4)  $C_x$ is an ideal in its normalizer $N_x$. Since $\gG$ is simple, 
$N_x$ is a proper subalgebra of $\gG$ and, as such, must be abelian. 
So, $N_x\subset C_x$ and hence $N_x=C_x$. 

\noindent(5) Directly from (4). 

\noindent(6) We have to prove that $[y,C_x]\cap C_x=\{0\}, \,\forall y\in 
\gG$. Assume the contrary and consider an element $z\in C_x$ such that $0\neq 
[y,z]\in C_x$. It follows, according to (4), that $C_x=C_z$ implies $[y,z]\in C_z$. 
Again by (4), this implies  that $y\in C_z$ and, therefore, $[y,z]=0$ in contradiction 
with the assumption. 
\end{proof}

Now we need some information on the operators of the adjoint 
representation of a hyper-simple algebra. 

%%%%%%%%%%%%%%%%%%%%%%%%%%%%%% 
\medskip
\subsection{On the adjoint representation of hyper-simple Lie algebras}
%%%%%%%%%%%%%%%%%%%%%%%%%%%%%%%

First, we mention without proof the following elementary facts.
\begin{lem}\label{sk-sym}
Let $V$ be a finite-dimensional vector space and let $b(\cdot,\cdot)$ be a 
non-degenerate symmetric bilinear form on $V$. If $A:V\rightarrow V$ is a 
linear operator, which is skew-symmetric with respect to $b$, i.e., 
$b(Au,v)+b(u,Av)=0, \,u,v\in V$, then the minimal polynomial of $A$ is of 
the form $t^r\varphi(t^2), \varphi(0)\neq0, \,r\geq 0$. If $A$ is 
semi-simple and $\varphi=\varphi_1^{n_1}\cdot,\dots,\cdot \varphi_m^{n_m}$ 
is the canonical factorisation of the polynomial $\varphi$ into irreducible 
and relatively prime factors, then $r=0$ or $1$ and $n_1=\dots =n_m=1$.  
\end{lem}
\begin{cor}\label{sk-sym-sim}
The assertion of the above lemma is valid for operators of the adjoint 
representation of a hyper-simple algebra and  in this case $r=1$. 
\end{cor}
\begin{proof}
Recall that the operators of the adjoint representation of a Lie algebra $\gG$ 
are skew-symmetric with respect to the Killing form, which is non-degenerate 
when $\gG$ is semi-simple. Moreover, according to Corollary \ref{semisimple}, 
for a hyper-simple Lie algebra  
these operators are semi-simple and hence $n_i=1$ for all $i$. 
Since the kernel of an adjoint representation operator is nontrivial, $r=1$. 
\end{proof}
 
It is not difficult to see that if $g(\tau)$ is an irreducible polynomial, 
then the polynomial $h(t)=g(t^2)$ either is irreducible or 
of the form $\psi(t)\psi(-t)$ with irreducible and relatively prime $\psi(t)$ and 
$\psi(-t)$. Hence, by Lemma \ref{sk-sym},
the minimal polynomial 
$f(t)$ of a semi-simple skew-adjoint operator $A$ is of the form 
\begin{equation}\label{s-factor}
F(t)=t^{\epsilon}f_1(t^2)\cdot\dots\cdot f_k(t^2)\psi_1(t)\psi_1(-t)\cdots\psi_l(t)\psi_l(-t), 
\quad \epsilon=0,1,
\end{equation}  
with relatively prime and irreducible factors. Under the hypothesis of 
Lemma \ref{sk-sym} with $\epsilon=1$ we have the following direct sum 
decomposition 
\begin{equation}\label{sum-dir} 
V=\ker\,A\oplus \ker\,f_1(A^2)\oplus\ldots\oplus \ker\,f_k(A^2)\oplus 
\ker\,g_1(A^2) \oplus\ldots\oplus \ker\,g_l(A^2) 
\end{equation}  
with $g_i(t^2)=\psi_i(t)\psi_i(-t)$.
\begin{lem}\label{sk-sym1}
The subspaces in decomposition \emph{(\ref{sum-dir})} are mutually orthogonal with 
respect to the form $b$.
\end{lem}
%\begin{proof}
\noindent {\it Proof}. This is a direct consequence of the fact that 
$b(\ker\,\varphi(A^2),\ker\,\phi(A))=0$, if polynomials $\varphi(t^2)$ and 
$\phi(t)$ are relatively prime. To prove this assertion, consider the identity
$$
\varphi(t^2)\alpha(t)+\phi(t)\beta(t)=1
$$
where $\alpha(t), \,\beta(t)$ are some polynomials . Let $u\in 
\ker\,\varphi(A^2), \,v\in\ker\,\phi(A)$ and 
$\alpha(t)=\alpha_0(t^2)+t\alpha_1(t^2)$. Then 
\begin{align*}
0=b([(\alpha_0(A^2)-A\alpha_1(A^2))\varphi(A^2)]u,v)=
b(u,[(\alpha_0(A^2)+A\alpha_1(A^2))\varphi(A^2)]v)=\\b(u,[\alpha(A)\varphi(A^2)]v)=
b(u,v-\phi(A)\beta(A)v)=b(u,v). \qquad \square
\end{align*}
%\end{proof} 

Finally, we shall prove some properties of the operators of the adjoint representation of a hyper-simple 
Lie algebra $\gG$ (over $\gk$). Below $(\cdot,\cdot)$ stands for the 
Killing form on $\gG$ and  $\gk_f$ for the splitting field of a 
polynomial $f\in \gk[t]$. 
\begin{proc}\label{simpl-ad}
Let $\gG$ be a hyper-simple Lie algebra, $0\neq x\in\gG$ and $A=ad\,x$. Then
\begin{enumerate}
\item the minimal polynomial $F(t)$ of $A$ has the form \emph{(\ref{s-factor})} with
$\epsilon=1$ and the decomposition \emph{(\ref{sum-dir})} is valid for $V=|\gG|$ with summands orthogonal with respect to the Killing form;
\item $C_x=\ker\,A$; 
\item the nonzero roots of $F(t)$ do not belong to $\gk$;
\item the lattice (in $\gk_F$) generated by the roots of $F(t)$ coincides with that
of $f_i(t)$ and with that of $\varphi_j(t), \,i=1,\ldots,k, 
\;j=1,\ldots,l$; 
\item $\gk_F=\gk_{f_i}=\gk_{\varphi_j}, \,i=1,\ldots,k, \;j=1,\ldots,l$;
\end{enumerate}
\end{proc} 
\begin{proof}
(1) Directly from Corollary\,\ref{sk-sym-sim} and Lemma\,\ref{sk-sym1}.

(2) Directly from Proposition\,\ref{solid property}.

(3) Assume the contrary. Then $x$ and an eigenvector 
$y\in \gG$  corresponding to a non-zero eigenvalue of $A$ span a 
$2$-dimensional non-abelian subalgebra of 
$\gG$. Since $\gG$ is hyper-simple, it must coincide with this subalgebra 
in contradiction with the simplicity of $\gG$.

(4) Let $p(t)$ be one of the polynomials $f_i$'s, $\varphi_j$'s and let
$\lambda_1,\ldots,\lambda_m\in\gk_p$ be its roots. Consider the subalgebra 
$\gH=\gH_p$ of $\gG$ generated by $W=\ker\,p(A)$ and $C_x$. Obviously, $\gH$ is non-abelian 
and hence $\gG=\gH$. On the other hand, in view of (\ref{comut}) applied to 
the $\gk_F$-extension of $\gG$, the eigenvalues of $A\mid_{\mid\gH\mid}$ belong 
to the lattice $L_h$
in $\gk_F$ generated by $\lambda_1,\ldots,\lambda_m$. 
Lattices $L_{f_i}$'s and $L_{\varphi_j}$'s must coincide, since, otherwise, one
of the subalgebras $\gH$ would be proper in $\gG$.

(5) Directly from (4).
\end{proof}

%%%%%%%%%%%%%%%%%%%%%%%%%%%%%%%%%
\subsection{Complete disassembling of real Lie algebras} 
%%%%%%%%%%%%%%%%%%%%%%%%%%%%%%%%%

We are now ready to prove that any Lie algebra over $\R$ can be 
completely disassembled. The stripping procedure in this case is based 
on the following fact.
\begin{proc}\label{R-simp}
Hyper-simple Lie algebras over $\R$ are isomorphic to $\gs\go(3,\R)$.
\end{proc} 
\begin{proof}
Let $\gG$ be a hyper-simple real Lie algebra and let $A=ad\,x, \,0\neq x\in\gG$.
According to Proposition \ref{simpl-ad} (3), none of the nonzero roots of $A$ belongs to $\R$ and hence the minimal polynomial $F(t)$ of $A$ is of the form 
$F(t)=t(t^2+\lambda_1^2)\ldots(t^2+\lambda_k^2)$, 
$\lambda_1,\ldots,\lambda_k\in\R, \,k\geq 1$.
If $C$ is a Cartan subalgebra of $\gG$, then it follows from 
Proposition \ref{solid property}, that $\{ad\,y\}_{y\in C}$ is a family of 
commuting semi-simple operators. Recall that such a family possesses a 
\emph{primitive} element, i.e., an element such that any of its invariant subspaces
is also invariant with respect to all the operators of the family. Let $A=ad\,x, 
\,x\in C,$ be primitive for the family $\{ad\, y\}_{y\in C}$ and let
$f(t)=t^2+\lambda^2$ be one of the irreducible factors of its minimal 
polynomial $F(t)$. We consider the subspace 
$$
P=\mathrm{span}<y,Ay>\subset |\gG| \quad \mbox{with} \quad 0\neq y\in\ker\,f(A),
$$
and we list some of its properties.\\
1) $[C,P]\subset P$: $P$ is $A$-invariant, since $A^2y=-\lambda^2y$. But
$A$ is primitive and hence $P$ is $(ad\,z)$-invariant for any $z\in C$. \\
2) $\dim\,P=2$: This follows from $\ker\,A\cap\ker\,f(A)=\{0\}$ (see (\ref{sum-dir})).\\
3) $[P,P]\subset C$: $A([y,Ay])=[Ay,Ay]+[y,A^2y]=0$, since $A^2y=-\lambda^2y$. But, 
 according to Proposition\,\ref{solid property}, $C=C_x=\ker\,A$.\\
4) $[P,P]\neq 0$: Since $\dim\,P=2$, $[P,P]=\mathrm{span}\,([y,Ay])$. If $[y,Ay]=0$,
then $\mathrm{span}(x,y,Ay=[x,y])$ is a 3-dimensional subalgebra of $\gG$. It is
non-abelian, since $[x,y]=Ay\neq 0$, and solvable. On the other hand, $\gG$
is hyper-simple and this subalgebra must coincide with $\gG$. But this is impossible,
since $\gG$ is simple.

Thus $[P,P]=\mathrm{span}\,([y,Ay])$ is 1-dimensional and $P$ is $[P,P]$-invariant,
since $[P,P]\subset C$. So, $P\oplus [P,P]$ is a 3-dimensional subalgebra, which
is non-abelian, since $[y,Ay]\neq 0$. Being hyper-simple, $\gG$ coincides with
this subalgebra, so that $\dim\,\gG=3$. But any 3-dimensional simple Lie 
algebra over $\R$ is isomorphic either to $\gs\go(3,\R)$ or to  $\gs\gl(2,\R)$ (according to Bianchi's classification) and the second one is not hyper-simple.
\end{proof}
\begin{thm}\label{R-dis}
Any finite-dimensional Lie algebra over $\R$ can be completely 
disassembled. 
\end{thm}
\begin{proof}
As we have seen earlier, the problem reduces
to proving the existence of simplifying $d$-pairs for algebras of the form $\gG\oplus_{\rho}V$ where
$\gG$ is simple and $\rho:\gG\to\End V$ is irreducible. We shall construct such a pair by means of a simple 3-dimensional subalgebra $\gH$ of $\gG$. By Propositions\,\ref{simplest-in}
and \ref{R-simp}, $\gG$ contains either a subalgebra isomorphic to $\gs\gl(2,\R)$
or a subalgebra isomorphic to $\gs\go(3,\R)$. In the first  case, the "splitting" arguments
in the proof of Proposition\,\ref{split-by-3d} prove the existence of a  simplifying $d$-pair. Hence  we have to analyse the situation when $\gH$ is isomorphic to $\gs\go(3,\R)$. If,
in this case, the first  $d$-pair associated with $\gH$ is nontrivial, then the $d$-pair 
(\ref{1st-ext}) in the proof of Proposition\,\ref{split-by-3d} solves the problem.

So, we shall assume that the $d$-pair $(\gs, W)$ associated with $\gH$ is of the
second type. In particular, in the notation of Proposition\,\ref{split-by-3d}, we 
have $\gG=\oplus_{k\geq 0}\gG_{2k}$. Set also
$$
V_{even}=\bigoplus_{k\geq 0}V_{2k} \quad \mathrm{and} \quad
V_{odd}=\bigoplus_{k\geq 0}V_{2k+1}.
$$
Commutation relations (\ref{V-com}) show that the subspaces $V_{even}$ and 
$V_{odd}$ are $\rho$-invariant. Since $\rho$ is irreducible, one of these 
subspaces is trivial.

First, assume that $V_{odd}$ is trivial. Once again, relations (\ref{V-com})
show that 
$$
\boldsymbol{V}_0=\bigoplus_{k\geq 0}V_{4k}, \quad\quad  
\boldsymbol{V}_1=\bigoplus_{k\geq 0}V_{4k+2}
$$
is a $\rho$--extension of $(\gs, W)$.

Finally, assume that $V_{even}$ is trivial. First we 
note that, by Corollary\,\ref{nontrivS_1}, $\mathcal{S}_1(\gs, W, \rho)$ is nontrivial. 
Moreover, Proposition\,\ref{AlmExt} reduces the problem to the case where the 
$\dF_2$-graded algebra $\mathcal{S}(\gs, W, \rho)$ is isomorphic to $\C$. In this 
case, $\mathcal{S}_1(\gs, W, \rho)$ is 1-dimensional and contains an operator $J$
such that $J^2=-1$. So, $J$ supplies $V$ with a $\C$-vector space structure
which will be denoted by $V_{\C}$.

Below we shall maintain the notation used in the proof of  Proposition \ref{split-by-3d}. 
By definition, the operators 
$\rho(z), \,z\in\gs$, commute with $J$, i.e., they are $\C$-linear in $V_{\C}$. 
In particular, $h\in\gs$ and hence $B=\rho(h)$ is such an operator.
The eigenvalues of  $B$ are $\frac{m}{2}\sqrt{-1}, \,m\in\Z$, and
\begin{equation}
V_{\C}=\oplus_{m\in\Z}\V_m, \quad \V_m=\ker(B-\frac{m}{2}\sqrt{-1}\id_V)=
\ker(B-\frac{m}{2}J)
\end{equation}
We stress that the $\V_m$'s are \emph{complex} subspaces of  $V_{\C}$. 

In the case under  consideration, $\V_m$ may be nontrivial only for
odd $m$ and $V_{\C}=\boldsymbol{W}_1\oplus\boldsymbol{W}_2$ with
$$
\boldsymbol{W}_{\boldsymbol{1}}=\bigoplus_{m\in\Z}\V_{4m+1} \quad\mathrm{and}
\quad \boldsymbol{W}_{\boldsymbol{2}}=\bigoplus_{m\in\Z}\V_{4m-1}.
$$
Moreover, $\dim\boldsymbol{W}_1=\dim\boldsymbol{W}_2$ as a consequence of 
$V_m=\V_m\oplus\V_{-m}$.

We shall prove that $\boldsymbol{W}_1$ and $\boldsymbol{W}_2$
are $\rho$--invariant. 
Let $v\in\V_m$ and let $w\in \gG_{4k+2}\subset W$. Then 
$$
Bv=\frac{m}{2}\sqrt{-1}\,v= 
\frac{m}{2}Jv, \quad [h,w]\in W, \quad [h,[h,w]]=-(2k+1)^2w,
$$ 
and, therefore, $$
\rho(w)J+J\rho(w)=0=\rho([h,w])J+J\rho([h,w]).
$$ 
Now we have
\begin{eqnarray}\label{happy1}
B(\rho(w)v)=\rho(w)(Bv)+[B,\rho(w)]v=\frac{m}{2}\rho(w)(Jv)+[\rho(h),\rho(w)]v=
\nonumber\\-\frac{m}{2}J(\rho(w)v)+\rho([h,w])v=
 -\frac{m}{2}\sqrt{-1}\,\rho(w)v+\rho([h,w])v
 \end{eqnarray}
 and
 \begin{eqnarray}\label{happy2}
B(\rho([h,w])v)=\rho([h,w])(Bv)+[B,\rho([h,w])]v=\frac{m}{2}\rho([h,w])(Jv)+
\nonumber \\  +[\rho(h),\rho([h,w])]v=-\frac{m}{2}J(\rho([h,w])v)+\rho([h,[h,w]])v=
\nonumber \\ -\frac{m}{2}\sqrt{-1}\,\rho([h,w])v-(2k+1)^2\rho(w)v
\end{eqnarray}
It follows from (\ref{happy1}) and (\ref{happy2}) that vectors $\rho(w)v$ and 
$\rho([h,w])v$ span a 2-dimensional subspace $\Pi$ in $V_{\C}$ which is 
invariant with respect to $B$. It is easy to see, that the eigenvalues of $B\!\mid_{\Pi}$ are 
$\frac{-m\pm(4k+2)}{2}\sqrt{-1}$. This shows that $\Pi$ is spanned by the eigenvectors of 
$B\!\mid_{\Pi}$ corresponding to these eigenvalues and, as a consequence, that the vectors 
$\rho(w)v$ and $\rho([h,w])v$ belong to $\V_{-m+4k+2}\oplus\V_{-(m+4k+2)}$.
Since $m=4s\pm 1$, the residue of $m \;\mathrm{mod} \;4$ coincides with that 
of $-m\pm(4k+2)$. Therefore, vectors $\rho(w)v$ and $\rho([h,w])v$ belong to 
the same subspace $\boldsymbol{W}_{\boldsymbol{i}}$ as $v$.

If $z\in\gG_{4k}\subset\gs$, then $[h,z]\in\gs, \,[h,[h,z]]=-4k^2z$ and the operators
$\rho(z)$ and $\rho([h,z])$ commute with $J$. The same arguments as above
show that the vectors $\rho(z)v$ and $\rho([h,z])v$ belong to 
$\V_{m+4k}\oplus\V_{m-4k}$, i.e., to the same subspace 
$\boldsymbol{W}_{\boldsymbol{i}}$ as $v$. 

Thus, the subspaces $\boldsymbol{W}_{\boldsymbol{i}}$'s are $\rho$--invariant.
On the other hand, $\rho$ is irreducible and, therefore, one of these subspaces 
is trivial. But $\dim\boldsymbol{W}_1=\dim\boldsymbol{W}_2$ and hence the other 
subspace is trivial too, as well as $V_{\C}=\boldsymbol{W}_1\oplus\boldsymbol{W}_2$. 
But $V$ is nontrivial and this contradiction shows that the situation when $V_{even}$
is trivial is impossible.
\end{proof}

%%%%%%%%%%%%%%%%%%%%%%%%%%
\medskip
\noindent\textit{On the disassembling problem for Lie algebras over arbitrary fields.} 

%%%%%%%%%%%%%%%%%%%%%%%%%%%%%

\noindent It is plausible that any finite-dimensional Lie algebra 
over a field of  characteristic zero can be completely disassembled. In view 
of Proposition\,\ref{simplest-in}, the disassembling problem is reduced to 
the case of hyper-simple Lie algebras and their finite-dimensional representations.  
This approach 
presumes a description of hyper-simple Lie algebras over arbitrary ground fields and 
of their involutions. This problem does not appear to be insolvable 
and one can find useful ideas in order to solve it in the last chapter of \cite{Jac}.
The derived algebras $[\mathcal{A}_{\mathrm{Lie}}, \mathcal{A}_{\mathrm{Lie}}]$  
of the Lie algebras $\mathcal{A}_{\mathrm{Lie}}$  associated with 
division algebras $\mathcal{A}$ over a field $\gk$ are examples of hyper-simple Lie
algebras. In this connection, we stress
that the hyper-simple Lie algebras and their representations amply merit 
study in their own right. For instance, they appear to be  
natural substitutes for $\gs\gl_2$-triples in the search for
analogues of root space decompositions for simple Lie algebras 
over arbitrary fields. 

An alternative approach to the disassembling problem could be a direct
description of suitable $d$-pairs in simple Lie algebras over a given field. The starting point could be the description of simple Lie algebras given in the last
chapter of \cite{Jac}.   However, the necessary extension of such a description
to the representations of these algebras seems particularly problematic.

The conjecture that all Lie algebras can be assembled from lieons and the 
fact that lieons are, in fact, Lie algebras over $\Z$ suggest that the
Lie algebras over a field $\gk$ of characteristic zero are  
specialisations  to $\gk$ of some universal assemblage procedures.
In the following two sections, the reader will find other indications
in favour of this idea.

%%%%%%%%%%%%%%%%%%%%%%%%%

\section{Compatibility of lieons and first-level Lie algebras}\label{1st-level}

%%%%%%%%%%%%%%%%%%%%%%%%%%%%%%%%%%

A Lie algebra is of the \emph{first level} if it can be completely 
disassembled in one step. In other words, the first-level Lie algebras are those that 
can be simply assembled from a number of mutually compatible lieons. In this section we characterise compatible pairs of lieons geometrically and, on this basis, we construct 
examples of first-level Lie algebras.  

Throughout this section, $P_i, P_{\between}$ and $P_{\pitchfork}$ stand for the Poisson 
bi-vectors associated with the Lie algebras denoted by $\gG_i, \gG_{\between}$ and
$\gG_{\pitchfork}$, respectively.

%%%%%%%%%%%%%%%%%%%%%%%%

\subsection{Compatibility of triadons} 
 
%%%%%%%%%%%%%%%%%%%%%%%%%%%%%
Let $\gG$ be an $n$-triadon and  $V=|\gG|$. 
Denote by $C=C_{\gG}$ the center of $\gG$ and put $l=l_{\gG}=[\gG,\gG]$. 
Then $\mathrm{dim}\,C=n-2, 
\,\mathrm{dim}\,l =1$ and $l\subset C$. A basis $e_1,\dots,e_n$ of $V$ such 
that $e_3\in l$ and $e_i\in C$, if $i>2$, will be called \emph{normal} for 
$\gG$. The only nontrivial product in this basis is $[e_1,e_2]=\alpha e_3, 
\,\alpha\neq 0$. The associated Poisson bi-vector on $V^*$ in the 
corresponding coordinates is $P=P_{\gG}=\alpha x_3\xi_1\xi_2$. This shows 
that, up to a factor of  proportionality, $\gG$ is uniquely defined by the pair $(C,l)$. 

Consider now two $n$-triadons $\gG_1$ and $\gG_2$ on  
$V=|\gG_1|=|\gG_2|$, and put $C_i=C_{\gG_i}, \,l_i=l_{\gG_i}, 
\,P_i=P_{\gG_i}\,i=1,2, \;C_{12}=C_1\cap C_2$. Obviously, 
$n-4\leq\mathrm{dim}\,C_{12}\leq n-2$. 

Below we use Formula (\ref{Schouten in coordinates}) for computations of the
occurring  Schouten brackets.
\begin{lem}\label{a1}
If $\mathrm{dim}\,C_{12}=n-4$, then $\gG_1$ and $\gG_2$ are compatible if and only if
\, $l_i\subset C_{12}, i=1,2$.
\end{lem} 
\begin{proof}
We have to examine the following four cases.

A$_1$\,: \,$l_i$ \emph{does not belong to} $C_{12}, \,i=1,2$. In this case 
a basis $e_5,\dots,e_n$ in $C_{12}$  can be completed by some vectors 
$e_1\in l_1, e_2\in C_1, e_3\in l_2, e_4\in C_2$ up to a basis in $V$. The 
only nonzero product $[e_i,e_j]_1$ in $\gG_1$ is $[e_3,e_4]_1=\alpha_1e_1$, 
and $[e_1,e_2]_2=\alpha_2e_3$ in $\gG_2$ with some
$\alpha_1, \alpha_2\in \gk$. Hence $P_1=\alpha_1x_1\xi_3\xi_4$ and 
$P_2=\alpha_2x_3\xi_1\xi_2$ and a direct computation using Formula 
(\ref{Schouten in coordinates}) shows that $\ls P_1,P_2\rs\neq 0$, i.e., 
that $\gG_1$ and $\gG_2$ are not compatible. 

A$_2$\,: \,\emph{One of the subspaces} $l_i$\emph{'s, say,} $l_1$, 
\emph{belongs to} $C_{12}$ \emph{and the other,} $l_2$, \emph{does not}. In 
this case we complete a basis $e_5\in l_1,e_6,\dots,e_n$ in $C_{12}$ by 
some vectors $e_1,e_2\in C_1, \,e_3\in l_2,  \,e_4\in C_2$ up to a basis in 
$V$. By similar reasons, $P_1=\alpha_1x_5\xi_3\xi_4, 
\,P_2=\alpha_2x_3\xi_1\xi_2$ in the corresponding coordinates. Since $\ls 
P_1,P_2\rs\neq 0$, $\gG_1$ and $\gG_2$ are not compatible. 

A$_3$\,: \,$l_i\subset C_{12}, \,i=1,2,$ \emph{and} $l_1\neq l_2$. In this 
case we consider a basis $e_5\in l_1,e_6\in l_2,\dots,e_n$ in $C_{12}$ and 
complete it by vectors $e_1,e_2\in C_1, 
\,e_3,e_4\in C_2$  independent $\mathrm{mod}\,C_{12}$ up to a basis in $V$. In such a basis 
$P_1=\alpha_1x_5\xi_3\xi_4, \,P_2=\alpha_2x_6\xi_1\xi_2$ and 
$\ls P_1,P_2\rs=0$. So, $\gG_1$ and $\gG_2$ are compatible and it is 
easy to see that $\gG_1+\gG_2$ is isomorphic to 
$\pitchfork\oplus\pitchfork\oplus\gamma_{n-6}$. 

A$_4$\,: \;$l_1=l_2\subset C_{12}$. A basis $e_5\in l_1=l_2,e_6,\dots,e_n$ 
in $C_{12}$ can be completed by vectors 
$e_1,e_2\in C_1, \,e_3,e_4\in C_2$  independent $\mathrm{mod}\,C_{12}$ up to a basis in $V$. Then 
$P_1=\alpha_1x_5\xi_3\xi_4, \,P_2=\alpha_2x_5\xi_1\xi_2$ and $\ls 
P_1,P_2\rs=0$, i.e., $\gG_1$ and $\gG_2$ are compatible. The Poisson 
bi-vector corresponding to $\gG_1+\gG_2$ is equivalent to 
$x_5(\xi_1\xi_2+\xi_3\xi_4)$. 
\end{proof}
\begin{lem}\label{b1}
If $\mathrm{dim}\,C_{12}\geq n-3$, then $\gG_1$ and $\gG_2$ are compatible.
\end{lem} 
\begin{proof}
Let $V=L\oplus C$ with $C\subset C_{12}, \,\dim\,C=n-3$ and $L$ containing $l_1$ and 
$l_2$. Then $[L,L]_i\subset L, \,i=1,2$, and, so, $L$ supports a subalgebra $\gG_i^0$ 
of $\gG_i, \,i=1,2$. It is easy to see that the $\gG_i^0$'s are triadons supported by $L$.
By Corollary\,\ref{prod} they are compatible. On the other hand, 
$\gG_i=\gG_i^0\oplus\gamma_{n-3}$ where $\mathbf{\gamma}_{n-3}$ is the abelian subalgebra
supported by $C$. For this reason, the compatibility of $\gG_1$ and $\gG_2$ is equivalent to the
compatibility of $\gG_1^0$ and $\gG_2^0$.
\end{proof}

Lemmas \ref{a1} and \ref{b1} allow us to construct various families of mutually 
compatible triadons on a vector space $V$. Two of them are described below. 

Let $V$ be an $n$-dimensional vector space. A (finite) family 
$\{C_i\}$ of $(n-2)$-dimensional subspaces of $V$ will be called  \emph{tight} if 
$\mathrm{dim}\,C_i\cap C_j>n-4$. Tight families are easily described. 

\begin{lem}\label{tight}
A family $\{C_i\}$ of $(n-2)$-dimensional subspaces of $V$ is \emph{tight} 
if either all $C_i$'s are contained in a common 
$(n-1)$-dimensional subspace \emph{(``co-pencil")}, or  all $C_i$'s have a 
common $(n-3)$-dimensional subspace \emph{(``pencil")}. $\square$
\end{lem}

A finite family $\{\gG_i\}$ of triadons on $V$ will be called 
\emph{tight} if the family $\{C_i\}$ of their centers is tight. It follows 
from Lemma \ref{b1} that triadons forming a tight family are mutually compatible. 
This observation and Lemma \ref{tight} prove 

\begin{proc}\label{pencil}
Let $\{C_i\}$ be a co-pencil (resp., pencil) of $(n-2)$-dimensional 
subspaces of $V$. Assign to each $C_i$ a 1-dimensional subspace 
$l_i\subset C_i$. Then the triadons characterised by pairs $(C_i, l_i)$ 
are mutually compatible and the linear combinations of these triadons are first-level Lie algebras. $\square$
\end{proc} 

We also obtain the following result.

\begin{proc}\label{pencil1}
Let $\{\gG_1,\dots,\gG_m\}$ be a family of triadons 
characterised by pairs $(C_1,l_1),\dots,(C_m,l_m)$. If 
$\mathrm{span}\,(l_1,\dots,l_m)\subset\bigcap_{i=1}^m C_i$, then the $\gG_i$'s 
are mutually compatible and the linear combinations of these triadons are first-level Lie 
algebras.  
\end{proc}
\begin{proof}
If $\mathrm{dim}\,C_i\cap C_j>n-4$, then $\gG_i$ and $\gG_j$ are compatible 
by Lemma \ref{b1}. Otherwise, they are compatible by Lemma 
\ref{a1}.
\end{proof}

These constructions illustrate the diversity of combinations 
of triadons that produce first-level Lie algebras. 

%%%%%%%%%%%%%%%%%%%%%%%%%%%%%%%%%%%%%%%%%%%%%%%%%%%%%
\subsection{Compatibility of dyons and triadons} 
%%%%%%%%%%%%%%%%%%%%%%%%%%%%%%%%%%%%%%%%%%%%%%%%%%%%

A dyon  $\gG$ on $V$  is characterised up to a scalor factor by its center 
$C=C_{\gG}$ and the derived algebra $\Delta=\Delta_{\gG}=[\gG,\gG]$. Since both 
$\Delta$ and $C$ are abelian, we identify them with subspaces 
of $V$. Obviously, $\mathrm{dim}\,C=n-2, \,\mathrm{dim}\,\Delta=1, \,C\cap 
\Delta=\{0\}$. So, up to a scalar factor $\gG$ is completely determined 
by the pair $(C,\Delta)$ of subspaces of $V$ and vice-versa.

Consider now a triadon $\gG_{\pitchfork}$ and  a dyon $\gG_{\between}$ on $V$ and the pairs
$(C_{\between},\Delta)$ and $(C_{\pitchfork},l)$ that characterise them. Set $C_{12}=C_{\between}\cap C_{\pitchfork}$. Then 
$n-4\leq\mathrm{dim}\,C_{12}\leq n-2$.
 
\begin{lem}\label{a2} If $\mathrm{dim}\,C_{12}=n-4$, then $\gG_{\between}$ and 
$\gG_{\pitchfork}$ are incompatible.
\end{lem}
\begin{proof} 
First, note that  $\gG_{\between}$ and $\gG_{\pitchfork}$ are compatible 
if and only if the factorised lieons $\gG_{\between}/C_{12}$ and 
$\gG_{\pitchfork}/C_{12}$ are compatible. So, we can assume that $n=4$ and 
$C_{12}=0$. Then $V=C_{\between}\oplus C_{\pitchfork}$ and 
$\mathrm{dim}\,C_{\between}=\mathrm{dim}\,C_{\pitchfork}=2$. The projections 
defined by this splitting of $V$ send the line $\Delta$ to subspaces 
$\Delta^{\between}\subset C_{\between}$ and $\Delta^{\pitchfork}\subset 
C_{\pitchfork}$, respectively. Since $\Delta$ does not belong to 
$C_{\between}, \,\mathrm{dim}\,\Delta^{\pitchfork}=1$. There may occur one 
of the following three cases. 

I$_1$ : $\mathrm{dim}\,\Delta^{\between}=1$ and $C_{\pitchfork}=l\oplus 
\Delta^{\pitchfork}$. Let $e_1\in\Delta^{\between}, ,\ 
e_2\in\Delta^{\pitchfork}$ be such that $e_1+e_2$ generates $\Delta$. If 
$e_3$ generates $l$ and $\{e_1,e_4\}$ generate $C_{\between}$, then 
$e_1,\dots,e_4$ is a basis in $V$. In the corresponding coordinates, the  
Poisson bi-vectors  $P_{\between}$ and $P_{\pitchfork}$ 
are proportional to $(x_1+x_2)\xi_2\xi_3$ and to $x_3\xi_1\xi_4$, 
respectively, and a direct computation shows that 
$\ls P_{\between},P_{\pitchfork}\rs\neq 0$. 

I$_2$ : $\mathrm{dim}\,\Delta^{\between}=0$ and $C_{\pitchfork}=l\oplus 
\Delta^{\pitchfork}$. Consider a basis 
$e_1,\dots,e_4$ in $V$ with $e_1,e_2\in C_{\between}, \,e_3\in l, \,e_4\in 
\Delta^{\pitchfork}$. In this case, $P_{\between}$ and $P_{\pitchfork}$ 
are proportional to $x_4\xi_3\xi_4$ and to $x_3\xi_1\xi_2$, respectively, and 
$\ls P_{\between},P_{\pitchfork}\rs\neq 0$

I$_3$ : $\mathrm{dim}\,\Delta^{\between}=0$ and $l=\Delta^{\pitchfork}$. 
Similarly, in a basis $e_1,\dots,e_4$ in $V$ with $e_1,e_2\in C_{\between}, 
\,e_3\in l, \,e_4\in \C_{\pitchfork}$,  the Poisson bi-vectors $P_{\between}$ and
$P_{\pitchfork}$ are proportional to $x_3\xi_3\xi_4$ and to $x_3\xi_1\xi_2$, respectively, 
and  $\ls P_{\between},P_{\pitchfork}\rs\neq 0$
\end{proof}

Set $C=C_{\between}+C_{\pitchfork}$ and note that $\mathrm{dim}\,C=n-1$ if and only if 
$\mathrm{dim}\,C_{12}=n-3$. In this case there are two possibilities : $\Delta\cap 
C=\{0\}$ and $\Delta\subset C$. 

\begin{lem}\label{b2} If $\mathrm{dim}\,C_{12}=n-3$ and $\Delta\cap 
C=\{0\}$, then $\gG_{\between}$ and $\gG_{\pitchfork}$ are incompatible.
\end{lem}
\begin{proof} 
In this case, $V=\Delta\oplus C$. If $l$ does not belong to $C_{12}$, then, 
as in the preceding lemma, the factorisation by $C_{12}$ reduces the problem 
to $n=3$. If $n=3$, then 
$\mathrm{dim}\,C_{\between}=\mathrm{dim}\,C_{\pitchfork}=1, 
\,l=C_{\pitchfork}$ and $V=\Delta\oplus C_{\between}\oplus l$. In coordinates 
associated with a basis $\{e_1,e_2,e_3\}$ in $V$ such that
$e_1\in l, \,e_2\in C_{\pitchfork}, \,e_3\in\Delta$, the Poisson bi-vectors 
$P_{\between}$ and $P_{\pitchfork}$ are proportional
to $x_3\xi_1\xi_3$ and to $x_1\xi_2\xi_3$, respectively, 
and we see that $\ls P_{\between},P_{\pitchfork}\rs\neq 0$. 

If $l\subset C_{12}$, then $C_{12}=l\oplus C', \;\mathrm{dim}\,C'=n-4$, and 
the factorisation by $C'$ reduces the situation to $n=4$. In this 
particular case $\mathrm{dim}\,C_{\between}=\mathrm{dim}\,C_{\pitchfork}=2$ 
and $C_{\between}\cap C_{\pitchfork}=l$. In 
a basis $e_1,\dots,e_4$  such that $e_1\in C_{\pitchfork}, e_2\in 
C_{\between}, e_3\in\Delta, e_4\in l$ Poisson bi-vectors $P_{\between}$ and 
$P_{\pitchfork}$ are proportional to $x_3\xi_1\xi_3$ and to $x_4\xi_2\xi_3$,
respectively, and $\ls P_{\between},P_{\pitchfork}\rs\neq 0$.  
\end{proof}

\begin{lem}\label{c2} If $\mathrm{dim}\,C_{12}=n-3$ and $\Delta\subset C$, then 
$\gG_{\between}$ and $\gG_{\pitchfork}$ are compatible.
\end{lem}
\begin{proof} 
In this case $C=\Delta\oplus C_{\between}$ and $V=W\oplus C$ for a 
$1$-dimensional subspace $W$ of $V$. If $l$ does not belong to $C_{12}$, the 
factorisation of $V$ by $C_{12}$ reduces the situation to a 
$3$-dimensional space in which $C=C_{\pitchfork}\oplus C_{\between}$, \, 
$\mathrm{dim}\,C_{\between}=\mathrm{dim}\,C_{\pitchfork}=1, 
\,l=C_{\pitchfork}$ and $V=C_{\pitchfork}\oplus C_{\between}\oplus W$. 
Consider a basis $e_1,e_2,e_3$ in $V$ with $e_1\in C_{\pitchfork}, \,e_2\in 
C_{\between}, e_3\in W$. Since $\Delta\subset C_{\pitchfork}\oplus 
C_{\between}$ and $\Delta\cap C_{\between}=\{0\}, \,\Delta$ is generated by 
a vector of the form $e_1+\lambda e_2, \lambda\in \gk$. In the 
associated coordinates, the  Poisson bi-vectors 
$P_{\between}$ and $P_{\pitchfork}$ are proportional to $x_1\xi_2\xi_3$ 
and to $(x_1+\lambda x_2)\xi_1\xi_3$, respectively, 
and $\ls P_{\between},P_{\pitchfork}\rs=0$. 

If $l\subset C_{12}$, then $C_{12}=l\oplus C', \mathrm{dim}\,C'=n-4$, and 
the factorisation of $V$ by $C'$ reduces the situation to $n=4$. In this 
case $\mathrm{dim}\,C_{\between}= \mathrm{dim}\,C_{\pitchfork}=2$ and 
$C_{\between}\cap \C_{\pitchfork}=l$. Consider a 
basis $e_1\dots,e_4$ in $V$ such that $e_1\in C_{\pitchfork}, e_2\in 
C_{\between}, e_3\in l, e_4\in W$. Since $\Delta\cap C_{\between}=\{0\}$  
and $\Delta\subset C=C_{\between}+C_{\pitchfork}$, the line $\Delta$ is 
generated by a vector of the form  $e_1+\lambda e_2+\mu e_3$. It follows that
in the associated coordinates, the Poisson bi-vectors  $P_{\between}$ and $P_{\pitchfork}$
are proportional to $x_3\xi_2\xi_4$ and to  
$(x_1+\lambda x_2+\mu x_3)\xi_1\xi_4$, respectively, and
$\ls P_{\between},P_{\pitchfork}\rs=0$. 
\end{proof}
\begin{lem}\label{c2} 
If $\mathrm{dim}\,C_{12}=n-2$, then $\gG_{\between}$ and $\gG_{\pitchfork}$ 
are compatible. 
\end{lem}
\begin{proof} 
In this case, $C_{\between}=C_{\pitchfork}=C$ and $V=\Delta\oplus W\oplus C, 
\,\mathrm{dim}\,W=1$. In a basis $e_1\dots,e_n$ in $V$ such that $e_1\in 
\Delta, \,e_2\in W, \,e_3\in l\subset C, \,e_4,\dots,e_n\in C$, the 
Poisson bi-vectors $P_{\between}$ and $P_{\pitchfork}$ are proportional to 
$x_3\xi_1\xi_2$ and to $x_1\xi_1\xi_2$, respectively, and 
$\ls P_{\between},P_{\pitchfork}\rs=0$.   
\end{proof}

A summary of the above lemmas is
\begin{proc}\label{puni}
A dyon and  a triadon on $V$ are compatible if and only if the intersection 
of their centers is not generic, i.e., is of dimension greater than $n-4$. 
\end{proc}

The following assertion immediately follows from Propositions \ref{pencil} 
and \ref{puni}.  
 
\begin{cor}
Let $\gG_1,\dots,\gG_m$ be triadons and $\gG$ a 
dyon. If the centers of these lieons are contained in a 
common hyperplane in $V$, then  
$$
\gG+\alpha_1\gG_1+\dots+\alpha_m\gG_m, \quad\alpha_1,\dots,\alpha_m\in \gk,
$$ 
is a non-unimodular first-level Lie algebra. 
\end{cor}

%%%%%%%%%%%%%%%%%%%%%%%
\medskip
\subsection{Compatibility of dyons} 
%%%%%%%%%%%%%%%%%%%%%%%%%%%
 
Consider two dyons $\gG_i, \,i=1,2,$ on a vector space 
$V$ and the pairs $(\Delta_i,C_i), \,i=1,2$.  that characterise them. Recall 
that $\mathrm{dim}\,\Delta_i=1, \,\mathrm{dim}\,C_i=n-2$ and $\Delta_i\cap 
C_i=\{0\}$. Set $C_{12}=C_1\cap C_2$. Obviously, $n-4\leq 
\mathrm{dim}\,C_{12}\leq n-2$ and the compatibility of the $\gG_i$'s is 
equivalent to the compatibility of the factorised dyons $\gG_i/C_{12}$'s.
\begin{lem}\label{a3} 
If $\mathrm{dim}\,C_{12}=n-4$, then $\gG_1$ and $\gG_2$ are compatible if and only if
$\Delta_1\subset C_2$ and $\Delta_2\subset C_1$. 
\end{lem}
\begin{proof} 
Passing to the factorised structures $\gG_i/C_{12}$'s we can assume that 
$\mathrm{dim}\,V=4$. In this particular case 
$\mathrm{dim}\,C_1=\mathrm{dim}\,C_2=2$ and $V=C_1\oplus C_2$. Denote by 
$p_i : V\rightarrow C_i$ the  projections associated with this splitting of $V$
and set $L=\mathrm{span}(\Delta_1,\Delta_2)$. Now we have to examine the following situations.

K$_1$ :  $\mathrm{dim}\,L=2$ and $L\cap C_i=\{0\}, \,i=1,2$. Then 
$p_i|_L:L\rightarrow C_i$ is an isomorphism, $i=1,2,$ and hence there is a 
basis  $e_1\dots,e_4$ in $V$ such that $e_1, e_2\in C_1, \,e_3, e_4\in C_2$ 
and $e_1+e_3\in \Delta_1, e_2+e_4\in \Delta_2$. In this case 
$P_1$ and $P_2$ are proportional to $(x_1+x_3)\xi_3\xi_4$ ant to 
$(x_2+x_4)\xi_1\xi_2$, respectively, and $\ls P_1,P_2\rs\neq 0$. 

K$_2$ :  $\mathrm{dim}\,L=2, \,\mathrm{dim}\,L\cap C_1=1$ and $L\cap 
C_2=\{0\}$. Then $p_1|_L$ is an isomorphism. So, if $0\neq\varepsilon_i\in 
\Delta_i, \,i=1,2$, then $e_1=p_1(\varepsilon_1), e_2=p_1(\varepsilon_2)$ 
is a basis in $C_1$. Also $e_3=p_2(\varepsilon_1)\neq 0$, since 
$\Delta_1\cap C_1=\{0\}$, and $p_2(\varepsilon_1)$ and $p_2(\varepsilon_2)$ 
are proportional. If $e_4\in C_2$ is not proportional to $e_3$, then 
$e_1,\dots,e_4$ is a basis in $V$. Then $\varepsilon_1=e_1+e_3$ 
and $\varepsilon_2=e_2+\lambda e_3$. So, in the corresponding coordinates, 
$P_1$ and $P_2$ are proportional to $(x_1+x_3)\xi_3\xi_4$ and to 
$(x_2+\lambda x_3)\xi_1\xi_2$, respectively, and $\ls P_1,P_2\rs\neq 0$. 

K$_3$ :  $\mathrm{dim}\,L=2, \,\mathrm{dim}\,L\cap C_i=1, \,i=1,2$. If 
$\varepsilon_1, \,\varepsilon_2$ are as above, then 
$e_3=p_2(\varepsilon_1)\neq 0, \,e_1=p_1(\varepsilon_2)\neq 0$ and 
$p_2(\varepsilon_2)=\lambda e_3, \,p_1(\varepsilon_1)=\mu e_1$ for some 
$\lambda, \mu\in\gk$. Then $\varepsilon_1=\mu e_1+e_3, 
\,\varepsilon_2=e_1+\lambda e_3$. Complete vectors $e_1, e_3$ to a basis in 
$V$ by vectors $e_2\in C_1, e_4\in C_2$. In the associated coordinates, the  Poisson
bi-vectors $P_1$ and $P_2$ are proportional to $(\mu x_1+x_3)\xi_3\xi_4$ and to 
$(x_1+\lambda x_3)\xi_1\xi_2$, respectively. Now one can see that
$\ls P_1,P_2\rs$ is proportional to $ -\lambda(\mu 
x_1+x_3)\xi_1\xi_2\xi_4-\mu(x_1+\lambda x_3)\xi_2\xi_3\xi_4$. It follows that
$\ls P_1,P_2\rs=0$ if and only if $\mu=\lambda=0$. Geometrically, 
this means that $\Delta_1\subset C_2, \,\Delta_2\subset C_1$, or, 
equivalently, that $\gG_1+\gG_2$ is isomorphic to $\between\oplus\between$ 
for $n=4$ and to  $\between\oplus\between\oplus\gamma_{n-4}$ in the general 
case. 

K$_4$ :  $\mathrm{dim}\,L=1 \Leftrightarrow \Delta_1=\Delta_2$. In this 
case one easily constructs a basis $e_1,\dots,e_4$ in $V$ with $e_1, e_2\in 
C_1$ and $e_3, e_4\in C_2$ and $e_1+e_3\in \Delta_1=\Delta_2$. As earlier we 
see that $P_1$ and $P_2$ are proportional to $(x_1+x_3)\xi_3\xi_4$ and to 
$(x_1+x_3)\xi_1\xi_2$, respectively, and that $\ls P_1,P_2\rs\neq 0$. 
\end{proof}

\begin{lem}\label{b3} 
If $\mathrm{dim}\,C_{12}=n-3$, then $\gG_1$ and $\gG_2$ are compatible 
either if $\Delta_1=\Delta_2 \,\mathrm{mod} \, C_{12}$, or  if $\Delta_i\subset C_1+C_2, \,i=1,2$. 
\end{lem}
\begin{proof} 
As above, the factorisation $\mathrm{mod}\,C_{12}$ reduces the problem  to 
$n=3$. In this case $\mathrm{dim}\,C_1=\mathrm{dim\,C_2}=1$ and $C_1\cap 
C_2=\{0\}$. Equivalently, if $C=C_1+C_2$, then $\mathrm{dim}\,C=2$. Here two 
possibilities occur: 

J$_1$ : One of $\Delta_i$'s, say, $\Delta_1$, does not belong to $C$. Let 
$\{e_1, e_2, e_3\}$ be a basis in $V$ such that $e_i\in C_i, \,i=1,2, \,e_3\in \Delta_1$.
If $\sum_{i=1}^3\alpha_ie_i$ is a base vector in $\Delta_2$, then a computation shows
that $P_1$ and $P_2$ are proportional to $x_3\xi_2\xi_3$ and to 
$\sum_{i=1}^{3}\alpha_ix_i\xi_1\xi_3$, respectively. It follows that
$\ls P_1,P_2\rs$ is proportional to $(\alpha_1x_1+\alpha_2x_2)\xi_1\xi_2\xi_3$ and 
we see that $\ls P_1,P_2\rs=0$ iff 
$\alpha_1=\alpha_2=0$, or, equivalently, iff $\Delta_1=\Delta_2$. 
For arbitrary $n$ the last condition means that 
$\Delta_1=\Delta_2 \,\mathrm{mod} \, C_{12}$. 

J$_2$ : $\Delta_i\subset C, \,i=1,2$. Let $0\neq e_i\in C_i, \,i=1,2$. Then 
$e_1+\lambda e_2\in \Delta_2$ and $\mu e_1+e_2\in \Delta_1$ for some 
$ \lambda, \mu\in\gk$. Complete $e_1, e_2$ to a basis in $V$ by a vector 
$e_3$. Then, in the associated coordinates, $P_1$ and $P_2$ are proportional to  
$(\mu x_1+x_2)\xi_2\xi_3$ and to $(x_1+\lambda x_2)\xi_1\xi_3$, respectively,
and we see that$\ls P_1,P_2\rs=0$. 
\end{proof}
\begin{rmk}
The results of this section show that the compatible configurations of lieons
can be described in a manner which does not refer explicitly to the
ground field. This description is in terms of the relative positions 
of pairs of subspaces that characterise the dyons and triadons under consideration. The 
specificity of the ground field $\gk$ is exclusively confined to the coefficients 
of the linear combinations of the ``abstract" lieons in a ``compatible"
position from which the first-level Lie algebras over  $\gk$ are constructed.
\end{rmk}
It is not difficult to deduce from the proofs of Theorems \ref{C-dis} and 
\ref{R-dis} that there exists a number $\nu(n)$ such 
that any $n$-dimensional Lie algebra can be assembled from  no more than 
$\nu(n)$ lieons. On the other hand, the results of this section show that 
even first-level Lie algebras can be assembled from an unbounded number of 
dyons and triadons, intertwined in a chaotic manner. This makes
the problem of recognising isomorphic Lie algebras on the basis of their 
construction from lieons  nontrivial.  For this reason, we need to introduce 
assembling procedures that are more regular and we describe such a procedure for the classical simple Lie
algebras in the next section.

\section{Canonical disassemblings of classical Lie algebras}\label{classical}

In this section we shall describe some complete
disassemblings of classical Lie algebras that are, in a sense, \emph{canonical}. 
The classical Lie algebras are symmetry algebras of either bilinear forms or volume forms, and this interpretation suggests how to disassemble them completely  
over arbitrary ground fields of characteristic zero. 
In particular, this method does not distinguish between real and complex Lie algebras. 
%%%%%%%%%%%%%%%%%%%%%%%%%%%%%%%%%
\medskip
\subsection{Disassemblings of orthogonal Lie algebras}
\label{SO-dis} 
%%%%%%%%%%%%%%%%%%%%%%%%%%%%%%%%%

\noindent Let $g=\sum_1^n a_ix_i^2, \,0\neq a_i\in\gk,$ be a nondegenerate 
quadratic form on a $\gk$-vector space $V$. The Lie algebra $\gs\go(g)$ of 
(infinitesimal) symmetries of $g$ is composed of the linear vector fields $X$ 
on $V$ such that $X(g)=0$. Obviously, 
$e_{ij}=a_ix_i\partial_j-a_jx_j\partial_i\in \gs\go(g)$ and 
$[e_{ij},e_{jk}]=a_je_{ik}$. Moreover, the fields $\{e_{ij}\}_{i<j}$  form a 
basis of $\gs\go(g)$. For instance, if $a_1=\cdots a_p=1, \,a_{p+1}=\cdots=a_n=-1$ and $\gk=\R$, this is the standard basis of $\gs\go(p,q), \,q=n-p$.  

Let $x_{ij}$ be the linear function on $|\gs\go(g)|^*$ corresponding to 
$e_{ij}$. Clearly, $x_{ij}=-x_{ji}$ and  $\{x_{ij}\}_{i<j}$ is a 
cartesian chart on $|\gs\go(g)|^*$. The only nontrivial commutation relations 
involving the functions $x_{ij}$'s are $\{x_{ij},x_{jk}\}=a_jx_{ik}$ and   
$$
P=\sum_{i<j,\alpha}a_{\alpha}x_{ij}\xi_{i\alpha}\wedge\xi_{\alpha j} \quad 
\mathrm{with} \quad \xi_{ij}=\frac{\partial}{\partial x_{ij}} 
$$ 
is the Poisson structure on 
$|\gs\go(g)|^*$  associated with the Lie algebra $\gs\go(g)$. Observe that
\begin{equation}\label{dec-ort}
P=\sum_{\alpha}a_{\alpha}P_{\alpha}\quad\mbox{with}\quad 
P_{\alpha}=\sum_{i<j}x_{ij}\xi_{i\alpha}\wedge\xi_{\alpha j}
\end{equation}
Since $P$ is a Poisson bi-vector for arbitrary $a_{\alpha}$'s, this shows
that $P_1,\dots, P_n$ are mutually compatible Poisson bi-vectors. 

The same result may be obtained by observing that
$$
2P_{\alpha}=\ls P,X_{\alpha}\rs=\ls P_{\alpha},X_{\alpha}\rs \;\mathrm{with} \; X_{\alpha}=
\sum_sx_{\alpha s}\xi_{\alpha s},
\quad\mathrm{and} \quad \ls P_{\alpha},X_{\beta}\rs=0, \;\forall \alpha\neq\beta .
$$
Indeed, $\ls P, P_{\alpha}\rs=\frac{1}{2}\partial_P^2(X_{\alpha})=0$ and
$$
\ls P_{\alpha}, P_{\beta}\rs=\frac{1}{2}\ls P_{\alpha}, \ls P, X_{\beta}\rs\rs=
\frac{1}{2}\ls\ls P_{\alpha}, P\rs, X_{\beta}\rs
+ \frac{1}{2}\ls P,\ls P_{\alpha}, X_{\beta}\rs\rs=0.
$$
Let us write
$$
P_{\alpha}=\sum_{i<j}P_{\alpha,ij}\quad\mbox{with}\quad P_{\alpha,ij}=
x_{ij}\xi_{i\alpha}\wedge\xi_{\alpha j} .
$$ 
For a fixed index $\alpha$, the Poisson bi-vectors $P_{\alpha,ij}$'s are mutually 
compatible and each is associated with an algebra isomorphic to $\pitchfork_m, 
\,m=n(n-1)/2$. This shows that the Lie algebra $\gs\go(g)$ can be assembled in two steps 
from $n(n-1)(n-2)/2$  triadons. 

Translating the above results in terms of Lie brackets, one easily finds 
that
$$
[\cdot,\cdot]=[\cdot,\cdot]_1+\cdots+[\cdot,\cdot]_n
$$ 
where $[\cdot,\cdot]$ stands for the Lie bracket in $\gs\go(g)$ and 
the structure $[\cdot,\cdot]_{\alpha}$ is defined by the relations 
$$
[e_{i\alpha},e_{\alpha j}]_{\alpha}=a_{\alpha}e_{ij} \quad \mathrm{and}
\quad [e_{ij},e_{kl}]_{\alpha}=0, \;\mathrm{if} \;\alpha\notin \{i,j\}\cap\{k,l\}.
$$  
In turn, $[\cdot,\cdot]_{\alpha}=\sum_{i<j}[\cdot,\cdot]_{\alpha,ij}$ 
where the only nontrivial product $[\cdot,\cdot]_{\alpha,ij}$ involving the basis vectors $e_{kl}$'s is $[e_{i\alpha},e_{\alpha 
j}]_{\alpha,ij}=a_{\alpha}e_{ij}$. 
\begin{rmk}\label{ort-any-fields}
The Poisson bi-vectors $P_{\alpha}=\sum_{i<j}x_{ij}\xi_{i\alpha}\wedge\xi_{\alpha j}$
may be interpreted as bi-vectors over the ring $\Z[x_{ij}]_{1\leq i < j\leq n}$. Any formal
linear combination of these bi-vectors with coefficients in a field $\gk$ is naturally 
interpreted as a linear bi-vector over the polynomial algebra $\gk[x_{ij}]_{1\leq i < j\leq n}$,
i.e., as a Lie algebra over $\gk$. In this sense, the $P_{\alpha}$'s are universal building blocks
for $g$--orthogonal algebras. For instance,  if $\gk=\R$, then
$$
P_1+\dots+P_s-P_{s+1}-\dots-P_n, \quad s=n-r.
$$
is the Poisson bi-vector associated with $\gs\go(r,s)$.
\end{rmk}

%%%%%%%%%%%%%%%%%%%%%%%%%%%%%%%
\medskip
\subsection{Disassembling of symplectic Lie algebras}
\label{SP-dis}  
%%%%%%%%%%%%%%%%%%%%%%%%%%%%%%%%%

\noindent Let $\beta(v,w)$ be a nondegenerate 
skew-symmetric form on a $\gk$--vector space $V$. Then the dimension of $V$  
is even, say, $2n$, and there exists  a (canonical) basis  
$\{e_1,\ldots,e_n,e_1^{\prime},\ldots,e_n^{\prime}\}$ in $V$ such that 
$$
\beta(e_i,e_j^{\prime})=\delta_{ij}, \quad\beta(e_i,e_j)= 
\beta(e_i^{\prime},e_j^{\prime})=0, \quad i,j=1,\ldots,n.
$$ 
The symplectic Lie algebra  
$\gs\gp(\beta)$ consists of the operators 
$A\in \mathrm{End}\,V$ such that 
$$
\beta(Av,w)+\beta(v,Aw)=0, \;v,w\in V .
$$ 
The algebra $\gs\gp(\beta)$ can be completely disassembled by essentially the 
same method as that we applied to orthogonal algebras. Below we describe
a method better adapted to the symplectic case. 

Let $(p_1,\ldots,p_n,q_1,\ldots,q_n)$ be coordinates in $V$ with respect to 
the above basis and $\omega=\sum_idp_i\wedge dq_i$. Then the algebra  
$\gs\gp(\beta)$ may be interpreted as the algebra of linear vector fields 
$X$ on $V$ such that $L_X(\omega)=0$. They are the
Hamiltonian (with respect to $\omega$) vector fields $X_f$ with Hamiltonians 
$f=f(p,q)$ which are quadratic in the $p_i$'s and $q_i$'s.
In this interpretation, the Hamiltonian vector fields corresponding to the monomials, 
$p_ip_j, q_iq_j, p_iq_j, \,i,j=1,\ldots,n$, form a basis of $\gs\gp(\beta)$ and the 
Lie product in $\gs\gp(\beta)$ is interpreted as the commutator of vector fields. 
Alternatively, the identification $f\leftrightarrow X_f, 
\;\{f,g\}\leftrightarrow[X_f,X_g]=X_{\{f,g\}}$ allows us to interpret 
$\gs\gp(\beta)$ as the Lie algebra of quadratic polynomials 
$\gk_2[p,q]=\gk_2[p_1,\ldots p_n,q_1,\ldots,q_n]$ with 
respect to the Poisson bracket $\{\cdot,\cdot\}$ determined by the Poisson 
bi-vector $\Pi=\sum_i\partial_{p_i}\wedge\partial_{q_i}$. In other words, we interpret the Lie algebra $\gs\gp(2n,\gk)$ as the vector space $\gk_2[p,q]$ 
equipped with the bracket $[\cdot,\cdot]=\{\cdot,\cdot\}\vert_{\gk_2[p,q]}$. 
Observe that $\Pi=\Pi_1+\ldots+\Pi_n$ with 
$\Pi_i=\partial_{p_i}\wedge\partial_{q_i}$ and denote by 
$\{\cdot,\cdot\}^i$ the bracket associated with the Poisson bi-vector 
$\Pi_i$. Then $[\cdot,\cdot]=[\cdot,\cdot]_1+\cdots+[\cdot,\cdot]_n$ with  
$[\cdot,\cdot]_i=\{\cdot,\cdot\}^i|_{\gk_2[p,q]}$. Obviously, the Poisson bi-vectors 
$\Pi_i$'s and hence the brackets $[\cdot,\cdot]_i$'s are 
mutually compatible and we obtain the disassembling 
$$
(\gk_2[p,q],\,[\cdot,\cdot])=(\gk_2[p,q],\,[\cdot,\cdot]_1)+
\ldots+(\gk_2[p,q],\,[\cdot,\cdot]_n).  
$$
The Lie algebras $\gs\gp_i(2n,\gk)=(\gk_2[p,q],\,[\cdot,\cdot]_i), 
\,i=1,\ldots,n$, are isomorphic. So, it suffices to 
completely disassemble one of them, say, $\gs\gp_1(2n,\gk)$. To this end, 
observe that the Levi--Malcev decomposition of $\gs\gp_1(2n,\gk)$ is
$$
\gs\gp_1(2n,\gk)=\langle p_1^2,p_1q_1,q_1^2\rangle\oplus
\langle p_1p_i,p_1q_i,q_1p_i,q_1q_i,p_ip_j,p_iq_j,q_iq_j\rangle_{1<i,j\leq n}
$$
where $\langle a,\ldots,b\rangle$ denotes the subalgebra of 
$\gs\gp_1(2n,\gk)$ spanned by $a,\ldots,b$.

The semi-simple part $\mathfrak{s}=\langle p_1^2,p_1q_1,q_1^2\rangle$ of 
$\gs\gp_1(2n,\gk)$ is isomorphic to $\gs\gl(2,\gk)$. The  radical  $\mathfrak{r}$ of $\gs\gp_1(2n,\gk)$  is 
$$
\mathfrak{r}=
\langle p_1p_i,p_1q_i,q_1p_i,q_1q_i,p_ip_j,p_iq_j,q_iq_j\rangle_{1<i,j\leq n}
$$
and
$$
\mathfrak{c}=\langle p_ip_j,p_iq_j,q_iq_j\rangle_{1<i,j\leq n}
$$ 
is the center of $\mathfrak{r}$. Note that 
$ [\mathfrak{r},\mathfrak{r}]\subset \mathfrak{c}$. 

According to 
Proposition \ref{str}, the algebra $\gs\gp_1(2n,\gk)$ is assembled from 
$\mathfrak{s}\oplus_{\rho}|\mathfrak{r}|$ and a Lie algebra of the form
$\gamma_m\oplus\,\mathfrak{r}$,  where $\rho$ stands for the action 
of $\gs$ on the ideal $\mathfrak{r}$. So, it remains to disassemble each of 
these two algebras. 

The algebra $\mathfrak{r}$ contains the following Heisenberg subalgebras: 
\begin{eqnarray}
\mathfrak{h}_{ij}^{pp}=\langle p_1p_i,q_1p_j,p_ip_j\rangle, \quad\quad
\mathfrak{h}_{ij}^{pq}=\langle p_1p_i,q_1q_j,p_iq_j\rangle \\
\qquad\mathfrak{h}_{ij}^{qp}=\langle p_1q_i,q_1p_j,p_jq_i\rangle, 
\quad\quad \mathfrak{h}_{ij}^{qq}=\langle p_1q_i,q_1q_j,q_iq_j\rangle 
\end{eqnarray}

\noindent Any Lie subalgebra $\mathfrak{h}_{ij}^{ab}$
from this list naturally extends to the unique $r$--triadon $\pitchfork_{ij}^{ab}$ 
on $|\mathfrak{r}|, \,r=\dim\,\mathfrak{r}$,  whose 
center contains all quadratic monomials, which do not appear in the
description of  $\mathfrak{h}_{ij}^{ab}$. It is easy to check (see Subsection\,6.1)
that triadons $\pitchfork_{ij}^{ab}$'s  are mutually
compatible and, therefore, completely disassemble $\mathfrak{r}$. 

In order to disassemble the Lie algebra 
$\mathfrak{s}\oplus_{\rho}|\mathfrak{r}|$ consider the following 
$d$-pair in it:
\begin{equation}\label{long-d-pair}
(\,\langle p_1q_1,V_p\rangle, \langle p_1^2,q_1^2,V_q,\mathfrak{c}\rangle\,)
\;\mathrm{with} \; V_p=\langle p_1p_i,p_1q_i\rangle_{1<i\leq n},
V_q=\langle q_1p_i,q_1q_i\rangle_{1<i\leq n}
\end{equation}
Easily verified commutation relations
$$
\begin{array}{rcc}
[V_p,V_p]_1=[V_q,V_q]_1=0,& [V_p,V_q]_1\subset\mathfrak{c}, &[\langle p_1^2\rangle,V_p]_1=
[\langle q_1^2\rangle,V_q ]_1=0, \\ 
\quad\quad [\langle p_1^2\rangle,V_q]_1\subset V_p,& [\langle q_1^2\rangle,V_p]_1\subset V_q,&
[\langle p_1^2,q_1^2\rangle,\langle p_1^2,q_1^2\rangle]_1\subset \langle p_1q_1\rangle, \\

&[\langle
p_1^2,q_1^2\rangle , 
V_p]_1\subset V_q,   & 
[\langle p_1^2,q_1^2\rangle,V_q]_1\subset V_p,
\end{array}
$$
prove that (\ref{long-d-pair}) is, in fact, a $d$-pair.

Nontrivial relations among quadratic monomials from the dressing algebra 
of this $d$-pair are $[p_1^2,q_1^2]_1=4p_1q_1, \,[p_1^2, q_1p_i]_1=2p_1p_i, 
\,[p_1^2, q_1q_i]_1=2p_1q_i, \,1<i\leq n,$. So, the triples 
$(p_1^2,q_1^2,p_1q_1), (p_1^2,q_1p_i,p_1p_i), \,(p_1^2,q_1q_i,p_1q_i), 
\,1<i\leq n,$ span subalgebras of the dressing algebra, which are isomorphic 
to $\pitchfork$. As in the case of subalgebras $\mathfrak{h}_{ij}^{ab}$, these
subalgebras naturally extend to some triadons. These extensions 
are mutually compatible and, therefore, disassemble the dressing algebra. 

Now it remains to disassemble the algebra 
$$
\langle p_1q_1,V_p\rangle \oplus_{\varrho}\langle p_1^2,q_1^2,V_q,\mathfrak{c}\rangle
$$
where $\varrho$ is the action of the subalgebra $\,\langle p_1q_1,V_p\rangle$ 
on the abelian ideal
$\langle p_1^2,q_1^2,V_q,\mathfrak{c}\rangle$. This algebra is, in fact, 
the semi-direct product
$$
\mathfrak{a}=\langle p_1q_1\rangle\oplus_{\varsigma}(V_p\oplus_{\rho}
\langle p_1^2,q_1^2,V_q,\mathfrak{c}\rangle)
$$
where the action $\varsigma$ of $\langle p_1q_1\rangle$ on $V_p$ and  
$\langle p_1^2,q_1^2,V_q,\mathfrak{c}\rangle$ is induced by the bracket 
$[\cdot,\cdot]_1$. It is easy to see that the nontrivial commutation relations 
of elements in the basis of the ideal 
$\mathfrak{i}=V_p\oplus_{\rho}\langle p_1^2,q_1^2,V_q,\mathfrak{c}\rangle
\subset\mathfrak{a}$ are 
$$
[q_1^2,p_1r]_1=2q_1r\in V_q, \;[p_1r,q_1s]_1=rs\in \mathfrak{c} \quad
\mathrm{with} \quad r,s=p_i,q_j, \,0<i,j\leq n.
$$
Now we see that the triples $(q_1^2,p_1r,2q_1r), \,(p_1r,q_1s,rs)$  
span Heisenberg subalgebras in $\mathfrak{i}$. By the same arguments 
as before, their natural extensions disassemble the algebra 
$\mathfrak{i}$  into a number of triadons.

The final step is to disassemble the algebra
$$
\langle p_1q_1\rangle\oplus_{\varrho}|\mathfrak{i}|=\langle p_1q_1\rangle
\oplus_{\varrho}\langle p_1^2,q_1^2,V_p,V_q,\mathfrak{c}\rangle.
$$
Observe that $|\mathfrak{i}|$ is the direct sum of 
$\varrho$-invariant subspaces 
$$
\langle p_1^2,q_1^2\rangle,\quad\langle p_1r,q_1r\rangle \;\mathrm{with} \;r=p_i,q_i, \;0<i\leq n,
\quad \mathrm{and} \quad|\mathfrak{c}|.
$$ 
The action $\varrho$ on $|\mathfrak{c}|$ is trivial and each of the  
subalgebras 
$$\langle 
p_1q_1\rangle\oplus_{\varrho}(\langle p_1^2,q_1^2\rangle\oplus|\mathfrak{c}|), 
\,\langle p_1q_1\rangle\oplus_{\varrho}(\langle p_1r,q_1r\rangle\oplus|\mathfrak{c}|)
$$ 
is isomorphic to $2\between_{m}=2\pitchfork_m$, for a suitable $m$ (see 
Formula (\ref{2h=2b}) and Subsection\,5.2). 
This disassembles the algebra $\langle 
p_1q_1\rangle\oplus_{\varrho}|\mathfrak{i}|$ into
$2n$ triadons.

%%%%%%%%%%%%%%%%%%%%%%%%%%%%%%%% %%%%
\medskip
\subsection{Disassemblings  of $\gG\gl(n,\gk), \gs\gl(n,\gk),\gu(n,\gk)$ and $\gs\gu(n,\gk)$}\label{gl-sl-u-dis} 
%%%%%%%%%%%%%%%%%%%%%%%%%%%%%%%%%%%%%

First, we shall construct a 3-step disassembling of $\gG\gl(n,\gk)$. It is
convenient to interpret this Lie algebra as the  Lie algebra of linear vector fields
on an $n$--dimensional vector space $V$, on which the vector fields 
$e_{ij}=x_i\frac{\partial}{\partial x_j}, \,1\leq i,j\leq n$, form a natural basis. 
The nontrivial Lie products in this algebra are
$$
[e_{i\alpha},e_{\alpha j}]=e_{ij}, \;\mathrm{if} \;i\neq j,  \quad\mathrm{and}\quad 
[e_{i\alpha},e_{\alpha i}]=e_{ii}-e_{\alpha\alpha},  \;\mathrm{if}  \;i\neq \alpha.
$$
Let $z_{ij}$'s be the coordinates on $|\gG\gl(n,\gk)|^*$ corresponding to the $e_{ij}$'s. 
The associated  Poisson bi-vector on $\gG\gl(n,\gk)$ is
\begin{equation}\label{gl-pois}
P=\sum_{1\leq i,j,\alpha\leq n}z_{ij}\xi_{i\alpha}\xi_{\alpha j} \quad \mathrm{with} \quad
\xi_{ij}=\frac{\partial}{\partial z_{ij}}.
\end{equation}
In the basis 
$$
\{e_{ij}^t=t_i t_j e_{ij}\}, \,0\neq t_i\in\gk, \,i=1,\dots,n,
$$ 
we have 
$P=\sum t_{\alpha}^2z_{ij}^t \xi_{i\alpha}^t\xi_{\alpha j}^t$ where
$z_{ij}^t$ and $\xi_{ij}^t$ stand for coordinates and partial derivatives with respect
to this basis. The isomorphism identifying the second basis with the first transforms
$P$ into the Poisson bi-vector
\begin{equation}\label{1-gl-dis}
P_t=\sum_{1\leq\alpha\leq n} t_{\alpha}^2P_{\alpha} \quad\mathrm{with} \quad P_{\alpha}=
\sum_{(i,j)\neq (\alpha,\alpha)}z_{ij}\xi_{i\alpha}\xi_{\alpha j}, \quad t=(t_1,\dots,t_n).
\end{equation}
This implies that the $P_{\alpha}$'s are mutually compatible Poisson bi-vectors, and, in particular,
that $P=P_1+\dots+P_n$. In turn, $P_{\alpha}$ disassembles as
\begin{equation}\label{2-gl-dis}
P_{\alpha}=\sum_{i,i\neq\alpha}(z_{i\alpha}\xi_{i\alpha}\xi_{\alpha\alpha}+
z_{\alpha i}\xi_{\alpha\alpha}\xi_{\alpha i})+\sum_{i,j,i\neq\alpha, j\neq \alpha}
z_{ij}\xi_{i\alpha}\xi_{\alpha j}
\end{equation}
The first Poisson bi-vector in the sum in right-hand side of (\ref{2-gl-dis}) corresponds to
an algebra of the form $\Gamma_A$, while the second corresponds to a dressing algebra. Each can be simply disassembled into a number of $n^2$--triadons (see (\ref{g-dec})).

However, the Poisson bi-vectors $P_{\alpha}$'s in (\ref{1-gl-dis}) do not restrict to the 
subalgebra $\gs\gl(n,\gk)$ of $\gG\gl(n,\gk)$ and hence cannot be used to 
disassemble it. For this purpose, we consider another basis in $\gG\gl(n,\gk)$:
\begin{equation}\label{basis-gl}
e_{ij}^0=e_{ij}-e_{ji}, \,1\leq i < j\leq n, \;e_{ij}^1=e_{ij}+e_{ji}, \,1\leq i\leq j\leq n,
\end{equation}
which respects the matrix transposition and for $i\geq j$, we set
$e_{ij}^0=-e_{ji}^0, \,e_{ij}^1=e_{ji}^1$. 
The upper indices in the symbols $e_{ij}^{\epsilon}$ are understood as elements of $\dF_2$.
In this notation, all nontrivial Lie products of elements
of basis (\ref{basis-gl}) are as follows:
\begin{eqnarray}\label{com-for-gl}
[e_{ij}^{\sigma},e_{jk}^{\tau}]=e_{ik}^{\sigma+\tau}, \mathrm{if} \;i\neq k,
\;i\neq j, \;i\neq k;  & \nonumber \\  
\,[e_{ij}^{\sigma},e_{ji}^{\tau}]=2(e_{ii}^1-e_{jj}^1), 
\mathrm{if} \;i\neq j, \;{\sigma}\neq\tau ; 
&\, [e_{ij}^{\sigma},e_{jj}^1]=2e_{ij}^{\sigma+1}, \;\mathrm{if} \;i\neq j.
\end{eqnarray}
The corresponding $d$-pair in $\gG\gl(n,\gk)$ is $(\gs=\langle e_{ij}^0\rangle, 
W=\langle e_{ij}^1\rangle)$. Obviously, $\gs$ is isomorphic to $\gs\go(g)$ for
$g=\sum_{i=1}^n x_i^2$, $\gs\subset \gs\gl(n,\gk)\subset \gG\gl(n,\gk)$ and 
$$
W_0\df W\cap\gs\gl(n,\gk)=\langle e_{ij}^1, \,e_{ii}^1-e_{jj}^1\rangle_{1\leq i\neq j\leq n}.
$$
In particular, $(\gs,W_0)$ is a $d$-pair in $\gs\gl(n,\gk)$. Denote by $\gA_{\beta}$
(resp., by $\gA_{\beta_0}$) the dressing algebra (see Subsection\,5.3) corresponding 
to the $d$-pair $(\gs,W)$ (resp.,  to $(\gs,W_0)$), and by $\rho$ (resp., $\rho_0$) 
the corresponding representation of $\gs$ in $W$ (resp., in $W_0$). So, by construction, 
we have
\begin{equation}\label{dis-gl-sl}
 \gG\gl(n,\gk)=\gs\oplus_{\rho}W+\gA_{\beta},  \qquad
 \gs\gl(n,\gk)=\gs\oplus_{\rho_0}W_0+\gA_{\beta_0}.
\end{equation}
Moreover, we have
\begin{lem}\label{dis-u-su}
For $\gk=\R$, the algebras
$$
\gu(n,\gk)=\gs\oplus_{\rho}W+\gA_{-\beta},  \qquad
 \gs\gu(n,\gk)=\gs\oplus_{\rho_0}W_0+\gA_{-\beta_0}.
$$
are isomorphic to the unitary  and special unitary  Lie 
algebras, respectively.
\end{lem}
\begin{proof}
The result follows from Relations (\ref{com-for-gl}).
\end{proof}
\begin{rmk}\label{field-as-glue}
The isomorphism classes of $\gG\gl_{\lambda}(n,\gk)\df\gs\oplus_{\rho}W+
\gA_{\lambda\beta}$ and of $\gs\gl_{\lambda}(n,\gk)\df\gs\oplus_{\rho_0}W_0+
\gA_{\lambda\beta_0}$, where $\lambda\in\gk$, depend on the quadratic part of $\lambda$. 
In fact, $\gG\gl_{\lambda}$ and  $\gG\gl_{\lambda^{\prime}}$ (resp., $\gs\gl_{\lambda}$ 
and  $\gs\gl_{\lambda^{\prime}}$) are isomorphic if and only if 
$\lambda^{\prime}=\lambda\mu^2, \,\mu\in\gk$. 
\end{rmk}

Since a dressing algebra can be simply disassembled into a number
of triadons, we shall focus on the algebras $\gs\oplus_{\rho}W$ and
$\gs\oplus_{\rho_0}W_0$. By virtue of (\ref{dis-gl-sl}) and Lemma\,\ref{dis-u-su},
a complete disassembling of such algebras automatically gives a complete disassemblings
of the algebras $ \gG\gl(n,\gk), \,\gs\gl(n,\gk), \,\gu(n,\gk)$ and $\gs\gu(n,\gk)$. 

Denote by  $z_{ij}$ (resp., $w_{pq}$) the linear functions on $|\gs|$ (resp., $|W|$) 
corresponding to $e_{ij}^0$ (resp., $e_{pq}^1$). Together they form a cartesian 
chart on  $|\gs\oplus_{\rho}W|=|\gs|\oplus|W|$, in terms of which the Poisson bi-vector 
associated with the algebra $\gs\oplus_{\rho}W$ is written as follows:
\begin{equation}\label{gl-2dis}
Q=\sum_{\alpha,  i<j}z_{ij}\xi_{i\alpha}\xi_{\alpha j}+
\sum_{p\neq\alpha\neq q}w_{pq}\xi_{p\alpha}\eta_{\alpha q}+
2\sum_{p\neq q}w_{pq}\xi_{pq}\eta_{qq}
\end{equation}
with $\xi_{ij}=\frac{\partial}{\partial z_{ij}}, \eta_{pq}=\frac{\partial}{\partial w_{pq}}$.
The same arguments as in Subsection\,7.1 show that if
$$
 Q_{\alpha}=
\sum_{i<j}z_{ij}\xi_{i\alpha}\xi_{\alpha j}+
\sum_{p,q, q\neq\alpha}w_{pq}\xi_{p\alpha}\eta_{\alpha q}+
2\sum_{p}w_{p\alpha}\xi_{p\alpha}\eta_{\alpha\alpha},
$$
then $Q=Q_1+\dots+Q_n$ is a simple disassembling of $Q$. 

In order to disassemble $Q_{\alpha}$ note that the single summands in the
right-hand side of (\ref{gl-2dis}) are $n^2$--triadons, 
and that the only incompatible pairs are
$$
w_{pq}\xi_{p\alpha}\eta_{\alpha q}  \quad\mathrm{and} \quad 
w_{q\alpha}\xi_{q\alpha}\eta_{\alpha\alpha}, \quad p\neq q, \,p\neq \alpha, \,q\neq \alpha,
$$
so that
\begin{equation}\label{3gl-dis}
Q_{\alpha}^1=\sum_{i<j}z_{ij}\xi_{i\alpha}\xi_{\alpha j},\quad
Q_{\alpha}^2=\sum_{p,q, q\neq\alpha}w_{pq}\xi_{p\alpha}\eta_{\alpha q}, \quad
Q_{\alpha}^3=\sum_{p}w_{p\alpha}\xi_{p\alpha}\eta_{\alpha\alpha}
\end{equation}
are Poisson bi-vectors and $\ls Q_{\alpha}^1,Q_{\alpha}^2\rs=
\ls Q_{\alpha}^1,Q_{\alpha}^3\rs=0$. Moreover,  by using Formula (\ref{Schouten 
in coordinates}) one easily finds that $\ls Q_{\alpha}^2,Q_{\alpha}^3\rs=0$. Hence
$Q_{\alpha}=Q_{\alpha}^1+Q_{\alpha}^2+Q_{\alpha}^3$ is a simple disassembling
of $Q_{\alpha}$. Finally, it follows from (\ref{3gl-dis}) that the Poisson bi-vectors 
$Q_{\alpha}^i$ are assembled from mutually compatible triadons. We thus obtain a complete  common disassembling  of $\gG\gl(n,\gk)$ and 
$\gu(n,\gk)$ in 4 steps.

In order to adapt this approach to the algebra $\gs\oplus_{\rho_0}W_0$, 
a suitable basis in $W_0$ should be chosen. Such a basis is 
$$
e_{ij}^1, \,1\leq i< j\leq n, \qquad e_i=\frac{1}{2}(e_{ii}-e_{11})=
x_i\partial_i-x_1\partial_1, \,1<i\leq n.
$$
Denote by $w_{ij}$ and $w_i$ the corresponding linear functions on 
$W_0^*$.  Together with the functions $z_{ij}$'s they form a cartesian chart on 
$|\gs\oplus_{\rho_0}W_0|^*=|\gs|^*\oplus|W_0|^*$.  
Also, set $\eta_{ij}=\frac{\partial}{\partial_{w_{ij}}}, 
\,\eta_i=\frac{\partial}{\partial_{w_{i}}}$.
It follows from (\ref{com-for-gl}) that, in terms of this chart, the Poisson bi-vector of 
$\gs\oplus_{\rho_0}W_0$ is written
\begin{eqnarray}\label{pois-for-sl}
Q^0=\sum_{j,  i<k}z_{ik}\xi_{ij}\xi_{jk}+
\sum_{\i\neq j,j\neq k,k\neq i }w_{ik}\xi_{ij}\eta_{jk}+
2\sum_{1< i,1< j, i<j }(w_i-w_j)\xi_{ij}\eta_{ji} \nonumber \\
-2\sum_{1<i}w_i\xi_{1i}\eta_{i1}+
\sum_{1<j,1<j,\i\neq j}w_{ij}\xi_{ij}\eta_{j}+
2\sum_{1<i}w_{1i}\xi_{1i}\eta_i+
\sum_{1<i,1<j,i\neq j}w_{1i}\xi_{1i}\eta_j.
\end{eqnarray}
Now we shall apply to $Q^0$ the change of variables
that we applied earlier to $P$ 
(see (\ref{1-gl-dis})). The expression of $Q^0$ in the basis
$$
t_it_je_{ij}^0, \,1\leq i< j\leq n, \qquad t_it_je_{ij}^1, \,1\leq i< j\leq n, 
\qquad t_i^2e_i, \,1<i\leq n.
$$
of  $\gs\oplus_{\rho_0}W_0$ is of the form $Q^0=t_1^2Q^0_1+\dots+t_n^2Q^0_n$, where the $Q^0_j$'s do not depend on the $t_i$'s. This shows that the $Q_j^0$'s are mutually
compatible Poisson bi-vectors. In particular, $Q^0=Q^0_1+\dots+Q^0_n$ is a 
simple disassembling of $Q^0$. Expressions of the bi-vectors $Q^0_j$'s are derived directly 
from from (\ref{pois-for-sl}):
\begin{equation}\label{pois-for-dis1-sl}
Q^0_1=\sum_{i<k}z_{ik}\xi_{i1}\xi_{1k}+
\sum_{i\neq 1,k\neq 1}w_{ik}\xi_{i1}\eta_{1k}
-2\sum_{1<i}w_i\xi_{1i}\eta_{i1}
\end{equation}
and, for $j>1$,
\begin{eqnarray}\label{pois-for-dis2-sl}
Q^0_j=\sum_{i<k,  j\notin\{i,k\}}z_{ik}\xi_{ij}\xi_{jk}+
\sum_{\ i\neq k, j\notin\{i,k\} }w_{ik}\xi_{ij}\eta_{jk}+
2\sum_{i,1< i\neq j }w_i\xi_{ij}\eta_{ji}+ \nonumber \\
\sum_{i,1<i\neq j}w_{ij}\xi_{ij}\eta_{j}+
2w_{1j}\xi_{1j}\eta_j+
\sum_{i,1<i\neq j}w_{1i}\xi_{1i}\eta_j.
\end{eqnarray}
Each single term in the summations (\ref{pois-for-dis1-sl}) and  (\ref{pois-for-dis2-sl})
is an $(n^2-1)$-triadon. It is easily verified by a direct 
computation or by using Lemmas\,\ref{a1} and \ref{b1} that
\begin{itemize}
\item all triadons in  (\ref{pois-for-dis1-sl}) are mutually compatible;
\item the incompatible pairs of triadons in (\ref{pois-for-dis2-sl}) are 
\begin{eqnarray}\label{comp-fork-pairs}
z_{1k}\xi_{1j}\xi_{jk} \;\mathrm{and}  \;w_{1k}\xi_{1k}\eta_j, \;k>1; \quad
w_{1k}\xi_{1j}\eta_{jk} \;\mathrm{and}  \;w_{kj}\xi_{kj}\eta_j, \;k>1;\nonumber \\
w_{1k}\xi_{kj}\eta_{j1} \;\mathrm{and}  \;w_{1j}\xi_{1j}\eta_j, \;k>1;
\end{eqnarray}
\end{itemize}
This shows that 
\begin{itemize}
\item $Q_1^0$ is simply disassembled into a number of triadons;
\item $Q_j^0, \,j>1$, is simply disassembled into a number of triadons
and structures
\begin{eqnarray}
Q_j^{\prime}=\sum_{k>1, k\neq j}z_{1k}\xi_{1j}\xi_{jk}+\sum_{k>1, k\neq j}w_{1k}\xi_{1j}\eta_{jk}+
2w_{1j}\xi_{1j}\eta_j \label{Q'}\\ \label{Q''}
Q_j^{\prime\prime}=\sum_{k>1, k\neq j}w_{1k}\xi_{1k}\eta_j+
\sum_{k>1, k\neq j}w_{kj}\xi_{kj}\eta_j+\sum_{k>1, k\neq j}w_{1k}\xi_{kj}\eta_{j1} 
\end{eqnarray}
\end{itemize}
Indeed, by (\ref{comp-fork-pairs}), $Q_j^{\prime}$ and $Q_j^{\prime\prime}$
are composed of mutually compatible triadons and, therefore,
are Poisson bi-vectors. Moreover, a direct computation using Formula 
(\ref{Schouten in coordinates}) proves that $\ls Q_j^{\prime},Q_j^{\prime\prime}\rs=0$.

The final, fourth step is to disassemble $Q_j^{\prime}$ and $Q_j^{\prime\prime}$ into
a number of triadons, which are single terms in the summations (\ref{Q'}) and (\ref{Q''}). We have thus
described a complete disassembling  in 4 steps of the simple Lie algebras we  considered.
\begin{rmk}\label{can-dis-repr}
The disassembling procedure described in this subsection is, essentially, a complete
disassembling of the canonical representations of the Lie algebras under consideration in $S^2V^*$.
In fact, the method can be extended to the canonical representations of these Lie algebras 
in higher tensor powers, $S^k V^*$.
\end{rmk}
In conclusion of this section, we note that the approach we  developed can be applied to the
central simple Lie algebras over arbitrary fields of characteristic zero, using the description given in chapter X of \cite{Jac}.
\begin{rmk}\label{can-dis-repr}
We stress that the first step in the disassembling procedures
we described for the classical Lie algebras reflects the fact that they are symmetry algebras of
\emph{compound} objects. For instance, for an algebra $\gs\go(g)$, such a compound object is
the form $g$ which is the direct sum of 1-dimensional quadratic forms, while for an algebra 
$\gs\gp(\beta)$, it is the form $\beta$ which is the direct sum of 2-dimensional skew-symmetric forms.
\end{rmk}
 
%%%%%%%%%%%%%%%%%%%%%%%%%%%%%%%%%%% 
\section{Some problems and perspectives}\label{problems}
%%%%%%%%%%%%%%%%%%%%%%%%%%%%%%%%%%%

This paper may be viewed as the first step toward what could be called the \emph{constructive
theory of Lie algebras}, while considerable systematic work remains to be done. Below we
mention some problems that remain to be solved and research directions that seem promising.

{\bf Assembling and disassembling of representations}.  Let $\gG$ be  a Lie algebra and 
$\rho\colon\gG\rightarrow\End V$ be a representation of $\gG$. Then $\gG\oplus_{\rho}V$ may be
viewed as a $\Z$-graded Lie algebra $\gG^{\rho}=\oplus_i\gG^{\rho}_i$, with 
$\gG^{\rho}_0=\gG, \,\gG^{\rho}_1=V$ and $\gG^{\rho}_i=0$ otherwise. Conversely, any
$\Z$-graded Lie algebra $\bar{\gG}=\oplus_i\bar{\gG}_i$ such that $\bar{\gG}_i=0$ if 
$i\neq0,1$ is, obviously, of the form $\gG^{\rho}$. Now a disassembling of $\rho$ may
be viewed as a disassembling of $\gG^{\rho}$ respecting this $\Z$-grading.
Thus, the methods described in this paper can be readily generalised in the framework of representation theory in order to
develop a \emph{constructive theory of representations}, an aim to be achieved.

{\bf Assembling and disassembling of graded Lie algebras.} All the basic definitions and 
techniques of this paper are valid for graded Lie algebras, thus they form a very natural framework for an extension of the constructive theory to the graded case. It naturally includes representations,
since a graded Lie algebra can be seen as a specific family of representations of its
zero component in all the others. An interesting question is to determine what may be the analogues 
of lieons for the graded algebras currently being studied, for instance, for super-algebras.

{\bf Disassembling problem for arbitrary ground fields}. It is very plausible 
that the complete disassembling theorem holds for finite-dimensional Lie algebras 
over arbitrary ground fields of characteristic zero (see the more detailed discussion above,
at the end of Section\,\ref{dis-probl}). Is it still true for non-zero characteristics?

{\bf Algebraic variety of Lie algebra structures.} The algebraic variety $\mathcal{L}ie(V)$ 
of all Lie algebra structures on a vector space $V$ is an intersection 
 of quadrics in the vector space $\mathcal{A}(V)=\mathrm{Hom}_{\gk}(V\otimes V,V)$. The 
subspace of $\mathcal{A}(V)$ spanned by a family of mutually compatible Lie algebra structures is contained in $\mathcal{L}ie(V)$. In this sense, $\mathcal{L}ie(V)$ is ``woven", like 
a one-sheet hyperboloid, from these subspaces. This suggests that we use this ``web 
structure" in order to find a description of $\mathcal{L}ie(V)$.
An instructive example of this kind is given in \cite{Mor}. Also, a deeper understanding
of the structure of $\mathcal{L}ie(V)$ for algebraically closed ground fields could
shed some light on the general disassembling problem.

{\bf Deformations.} On the basis of a disassembling of a Lie algebra $\gG$, one can 
construct deformations of its Lie algebra structure by substituting  $\lambda_v\gG_{v}, \,\lambda_v\in\gk,$ 
for $\gG_{v}$ in the corresponding a-scheme (see Subsection\,\ref{statemant}). 
The factors $\lambda_v$ are generally constrained by some relations, but these   
are trivial in the case of first-level algebras. In this connection the question whether
all (essential) deformations of a given Lie 
algebra are of this kind is on the table.

{\bf The length of complete disassemblings.} The procedure we have used in the proof 
of the complete disassembling theorem yields an estimate of  $n+\mathrm{const}$ for the minimal number of steps necessary to achieve it.
On the other hand, classical Lie algebras can be completely disassembled in at most four steps, independently of their dimension (see Section\,\ref{classical}). So, it is natural to ask whether there is a universal constant 
$N$ such that any Lie algebra can be completely disassembled in no more than $N$ steps. If the answer is positive, the problem of the description of the variety $\mathcal{L}ie(V)$ will rest on a more constructive basis.

{\bf Lie algebras of first and second levels.} 
The problem of an explicit description of the Lie algebras of the first two levels seems to
be attainable. Some useful suggestions on how to solve it can be found in \cite{V2} 
where a simplified version of this problem is solved in the case of 
first-level Lie algebras. A solution of this problem would significantly 
enrich the assembling techniques, since it presumes a systematic analysis of 
obstructions to compatibility.

{\bf Invariants of Lie algebras.} It is natural question to ask how can one deduce the invariants 
of a Lie algebra from a given disassembling? In this connection see \cite{V2}.

{\bf Cohomological aspects.} Poisson (resp., Lie algebra) structures, which are 
compatible with a given one, are closed 2-cochains in the associated Lichnerowicz--Poisson 
(resp.,  Chevalley--Eilenberg) complex. This fact was not explicitly exploited in 
this paper but a deeper understanding of assembing/disassembling procedures is
closely related to this cohomological aspect. As an example, we mention that the
first step in disassembling a classical Lie algebra $\gG$, as was done in
Section 6 is, in fact,  in the form $P_{\gG}=\sum_i\ls X_i, P_{\gG}\rs$ for some 
vector fields $X_i$'s on $|\gG|$. The terms $\ls X_i, P_{\gG}\rs$ are nothing but 
exact 2-cochains in the Lichnerowicz-Poisson complex associated with $P_{\gG}$.

Also, one may expect that the cohomology of a Lie algebra assembled from some other Lie
algebras is in its turn ``assembled" from the cohomologies of these algebras. However, 
an exact formalisation of this idea requires some new homological techniques, which 
will be discussed elsewhere. 

{\bf Generalisation to multiple Lie algebras.} 
The assemblage techniques and the results of this paper can be naturally  extended  
to multiple Lie algebras (see \cite{Fil, HW, MVV, VV1, VV2}). An interesting point is that the \emph{hereditary structures} associated with an 
$n$-ary Lie algebra are automatically compatible. This observation relates the disassembling 
problems for Lie algebras of different multiplicities. More generally, for any kind of Poisson structures, e.g., Lie algebroids and their $n$-ary
analogues (see \cite{Mac}), the compatibility 
problems of different multiplicities are closely related.

{\bf Applications to physics.} As far as we know, applications of compatible Poisson structures to various problems in physics and mechanics are to be found almost exclusively
in the context of integrable systems. We do not deal  here with these well-known aspects.
A quite different line of applications could come from a possible interpretation of Lie 
algebras as `` symmetry compounds". Various questions would arise, such as: can new chemical compounds that can be  synthesised from given ingredients
be described in terms of the compatibility conditions that relate the symmetry 
algebras of their ingredients? A similar question concerns elementary particles.
However, we stress that a certain maturity of this new theory is required
before it becomes useful to try to answer these questions. 

{\bf Acknowledgements.} The work, some results of which are presented in this
paper, is, in fact, a byproduct of author's study of natural structures in the theory 
of nonlinear partial differential equations.
This is why they were obtained in several approaches and at long intervals of time. It started with the author's visit to I.H.\'{E}.S. (Bures-sur-Yvette, France) in 2000 and
was completed under  programs PRIN (``Programmi di Ricerca di Rilevante Interesse Nazionale", Italia) of MIUR (``Ministero dell'Istruzione Pubblica, Universit\'{a} e Ricerca", Italia) from 2008 to 2012. The author is deeply grateful to these organisations
for their support. 

Special thanks go to the referee for his extraordinary work to
improve the original version of this paper, in particular for coining the terms ``dyon" 
and ``triadon" for the elementary constituents of Lie algebras, but also more generally 
for having transformed the ``totally unacceptable" English in the original into an 
acceptable form.

\end{document}